\documentclass[a4paper]{amsart}

\usepackage[a4paper,width=150mm,top=25mm,bottom=35mm]{geometry}

\usepackage{amsmath}
\usepackage{amssymb}
\usepackage{amsthm}
\usepackage{bbm}

\numberwithin{equation}{section}

\author{Edgar Assing}
\title{On the size of $p$-adic Whittaker functions}

\address{University of Bristol}
\email{edgar.assing@bristol.ac.uk}


\theoremstyle{plain}
\newtheorem{theorem}{Theorem}
\newtheorem{lemma}[theorem]{Lemma}

\numberwithin{theorem}{section}

\newtheorem{rem}{Remark}

\DeclareMathOperator{\R}{\mathbb{R}}
\DeclareMathOperator{\C}{\mathbb{C}}

\DeclareMathOperator{\Q}{\mathbb{Q}}
\DeclareMathOperator{\Z}{\mathbb{Z}}
\DeclareMathOperator{\N}{\mathbb{N}}

\DeclareMathOperator{\p}{\mathfrak{p}}

\DeclareMathOperator{\vol}{\text{Vol}}

\DeclareMathOperator{\Norm}{\text{Nr}_{E/F}}
\DeclareMathOperator{\Tr}{\text{Tr}_{E/F}}

\newcommand{\abs}[1]{\left\vert #1 \right\vert}
\newcommand{\du}[1]{\underline{\underline{#1}}}

\begin{document}

\begin{abstract}
In this paper we tackle a question raised by N. Templier and A. Saha concerning the size of Whittaker new vectors appearing in infinite dimensional representations of $GL_2$ over non-archimedean fields. We derive precise bounds for such functions in all possible situations. Our main tool is the $p$-adic method of stationary phase.
\end{abstract}

\maketitle

\setcounter{tocdepth}{1}
\tableofcontents

\section{Introduction}

The asymptotic behaviour of special functions plays an important role in number theory. Many recent results rely on almost excessive handling of Bessel functions, or more generally, hypergeometric functions. Such functions and their various asymptotic expansions have been studied for hundreds of years by many mathematicians in all areas. Their appearance in number theory stems from the theory of automorphic forms. Indeed, all sorts of special functions can be found in the archimedean components of automorphic representations. This however raises the question of non-archimedean analogues. From the perspective of automorphic representation it is clear that these $p$-adic analogues should just be vectors in a suitable model of the local representation at $p$. But what about suitable integral representations support properties and asymptotic size?

Some of these analogies are well known and have been studied for a long time. For example, very classical, the (quadratic) Gau\ss\  sum which can be thought of as a local version of Fresnel's integral. One sees this resemblance by writing
\begin{equation}
	\sum_{x\in \mathbb{F}_p} \exp\left( 2\pi i \frac{x^2}{p} \right) \leftrightsquigarrow \lim_{x\to \infty} \int_0^x \sin(t^2)dt. \nonumber
\end{equation}
This analogy shows up frequently in the $p$-adic method of stationary phase. Another classical couple would be Jacobi sum and Beta function. 

However, more interesting for us is the conceptual relationship
\begin{equation}
	\sum_{x\in \mathbb{F}_p^{\times}} \chi(x) \exp\left( 2\pi i \frac{ax+bx^{-1}}{p} \right) \leftrightsquigarrow K_{\nu}(x) = \frac{1}{2}\int_{0}^{\infty} t^{-\nu}\exp(-\frac{x}{2}(t+t^{-1}))\frac{dt}{t}. \nonumber
\end{equation}
On the left hand side we have a twisted version of the Kloosterman sum and on the right the $K$-Bessel function. This pair is well known to number theorists as it appears for example in the Kutznetsov trace formula. Furthermore, the $K$-Bessel function appears as the unique spherical vector in the Kirillov model of a (real) principal series representation, and is therefore part of the Whittaker expansion of Maa\ss\   forms.  If one considers non-spherical vectors, other types of representations, or even representations of $GL_2(\C)$, then the class of special functions gets bigger and bigger including classical Whittaker functions $W_{p,q}(x)$, for example. 

The representation theory of $GL_2(F)$ over non-archimedean fields $F$ is somewhat richer, in the sense that there are seemingly more  possible types of representations. Thus one could ask what the special functions featuring in the local Whittaker models are. 

The first natural choice (comparable to the spherical vectors in (real) principal series representations) are the unique Whittaker new vectors $W_{\pi}$. The goal of this paper is to shed some light on support and size of these functions. Being guided by the archimedean situation, we will find many parallels but also some differences.

If $\pi$ is spherical then there is a closed form for the Whittaker vector $W_{\pi}$ in terms of the Satake parameters of $\pi$ which relate to the Hecke eigenvalues of the associated automorphic form. If $\pi$ is non-spherical then the situation looks completely different. However, recent lower bounds for $W_{\pi}$ were used to justify large sup-norms of modular forms in the level aspect, see \cite{Sa15_2, Te12}. In \cite{Sa15}, the author used second moments of $W_{\pi}$ to prove the, at present, best hybrid bound for the sup-norm of Maa\ss\  forms. Bounds on $W_{\pi}$ were also used in \cite{CS16} to relate the ramification index of a modular parametrization of an elliptic curve at cusps to support properties of ramified Whittaker vectors. The stage for future application, of $p$-adic Whittaker functions, seems wide open in particular in the realm of Vorono\"i summation formula, see \cite{Te_vor}. The latter is linked to the study of $p$-adic Bessel distributions and the associated integral transforms as described in \cite{Co14}.

In \cite{Te12}, the author raised a question about the exact size of the functions $W_{\pi}$, and in \cite{Sa15_2} the author put forward a conjecture for the case of a weakly ramified central character. In this paper we will answer \cite[Question~1]{Sa15_2}. It turns out that \cite[Conjecture~2]{Sa15_2} is wrong in general. This was already noticed by A. Saha and Y. Hu, who obtained several counter examples (unpublished). In addition to the obvious connection to number theory and automorphic forms, these results may be interesting in their own right.

This paper exclusively deals with representations of $GL_2$ over non-archimedean fields and associated Whittaker new vectors. It would be interesting to see what happens for other groups. Such a program has been initiated for $GL_n(\R)$ in \cite{TB14}. But, apart from some $GL_3$-Kloosterman sums that have been studied in \cite{DF97}, the non-archimedean special functions of higher  rank groups are not understood. However, we don't think that our methods can be applied in the case of $GL_n$, $n\geq 3$. This is due to the fact that our calculations, which are still feasible for $n=2$, would be tremendously more complex for higher $n$.

\textbf{Acknowledgements.} I would like to thank A. Saha, who kindly informed me about the existence of counterexamples to \cite[Conjecture~2]{Sa15_2}. This led to the discovery of several mistakes in the previous version of the paper.  Furthermore, I thank A. Corbett for many valuable suggestions and feedback.

\subsection{The main result} \label{se:theorems}

We will now state the main results of this paper. We tried to keep the notation used up to this point to a minimum. Full definitions will be found in Section~\ref{se:notation}.
  
\begin{theorem} \label{th:main1}
Let $p$ be an odd prime and let $\pi$ be a supercuspidal representation of $GL_2(\Q_p)$ then
\begin{equation}
	\sup_{g\in GL_2(\Q_p)} \abs{W_{\pi}(g)} \leq 2p^{\frac{1}{2}+\frac{n}{12}}.\nonumber
\end{equation}
where $n$ is the log-conductor of $\pi$ and $W_{\pi}$ is the normalised Whittaker new vector associated to $\pi$. Furthermore, if $n$ is odd we have the better bound
\begin{equation}
	\sup_{g\in GL_2(\Q_p)} \abs{W_{\pi}(g)} \leq 2\sqrt{p}. \nonumber
\end{equation}
\end{theorem}

It turns out that this bound is sharp in general. The possible large exponent is due to the appearance of a degenerate critical point in our stationary phase argument. Since the central character of dihedral supercuspidal representations is always bounded by $\frac{n}{2}$, the bound stated above is worse than predicted in \cite[Conjecture~2]{Sa15_2}. Note that in some situations, for example $n$ odd, there are no degenerate critical points and we obtain the expected upper bound. If $n$ is even, our proof reveals some conditions which ensures that $W_{\pi}$ is uniformly bounded. However, these conditions are difficult to state and hard to check in practise.

\begin{theorem}\label{th:main2}
Let $p$ be an odd prime and let $\pi = \chi_1\boxplus\chi_2$ be a unitary principal series representation of $GL_2(\Q_p)$. We have the bound
\begin{equation}
	\sup_{g\in GL_2(\Q_p)} \abs{W_{\pi}(g)} \leq 2 p^{\frac{1}{2}\lfloor \frac{n}{2} \rfloor - \frac{a(\chi_2)}{3}}. \nonumber
\end{equation}
If the central character of $\pi$ satisfies $a(\omega_{\pi})<\frac{n}{2}$ and $p \equiv 3 \mod 4$, then
\begin{equation}
	\sup_{g\in GL_2(\Q_p)} \abs{W_{\pi}(g)} \leq 2. \nonumber
\end{equation}
\end{theorem}

One notes that the local bound stated in \cite[Corollary~2.35]{Sa15_2} reads
\begin{equation}
	\sup_{g\in GL_2(\Q_p)} \abs{W_{\pi}(g)} \leq \sqrt{2} p^{\frac{1}{2}\lfloor \frac{n}{2} \rfloor}. \label{eq:loc_bound_Sa}
\end{equation}
Thus, we observe that in most cases we improve upon this bound by a power saving. 

To derive these bounds we first prove integral representations of the Whittaker new vector on certain special matrices $g_{t,l,v}$. Roughly we will prove that\footnote{N. Templier kindly informed the author that this formula is originally due to him in an unpublished manuscript from 2011.}
\begin{equation}
	W_{\pi}(g_{t,l,v}) = C(t,\pi) \int_{\mathfrak{O}^{\times}} \xi(z) \psi(\Tr(A(t)z)+v\varpi^{-l}\Norm(z))d\mu_E \nonumber
\end{equation}
for a quadratic space $E$ with character $\xi$ associated to $\pi$ and explicit constants $C(t,\pi)\in \mathbb{C}$  and $A(t) \in E^{\times}$. The choices for $E$ and $\xi$ can be naturally explained for each $\pi$. Such integral representations are amenable to the $p$-adic method of stationary phase. It is a nice feature of $p$-adic analysis that instead of asymptotic expansions we gain explicit evaluations. Indeed, even if we are mostly interested in upper bounds, our methods are strong enough to yield tight lower bounds as well as precise formulas for $W_{\pi}(g_{t,l,v})$. To demonstrate this we worked out the explicit form of $W_{\pi}(g_{t,l,v})$ for $\pi = \chi \text{St}$ or $\pi=\chi\boxplus \chi$. This is Lemma~\ref{lm:St_twist_2_bound} and Lemma~\ref{lm:balanced_ps_bound} below. An interesting feature, which is similar to the archimedean case, is that in the transition region we encounter the $p$-adic analogue of the Airy function. In particular, we have the lower bound
\begin{equation}
	 \sup_{g\in \text{GL}_2(\Q_p)} \abs{W_{\chi\boxplus \chi}(g)} \gg p^{\frac{a(\chi)}{6}}. \label{eq:lower_bound_example}
\end{equation}

It was already noted in \cite{Sa15_2} that lower bounds on local Whittaker new vectors translate to lower bounds for classical modular forms (or Maa\ss\  forms). For example, as in the proof of \cite[Theorem~3.3]{Sa15_2} one can show that \eqref{eq:lower_bound_example} implies
\begin{equation}
	\frac{\Vert \chi\otimes f \Vert_{\infty}}{\Vert \chi \otimes f \Vert_2} \gg_{f,\epsilon} p^{\frac{n}{6}-\epsilon n} , \nonumber
\end{equation}
for a Maa\ss\  newform $f$ and a primitive Dirichlet character $\chi$ of conductor $p^n$ co-prime to the level of $f$.

Let us also mention that the local bound holds for any non-archimedean ground field $F$. This is also true for our results. Indeed, if the absolute ramification of $F$ over $\Q_p$ is small compared to $p$ then our proof carries over with simple notational changes. Otherwise one has to deal with some technicalities arising from the $p$-adic logarithm. But all this seems to work out, the only reason we restrict our attention to $\Q_p$ is to keep the calculations as clean as possible.

The assumption $p\neq 2$ seems to be more difficult to remove. First, there are many features arising from the method of stationary phase that are slightly different for $p=2$. However, with carefully chosen notation one could probably deal with these problems. Second, in order to treat supercuspidal representations we make heavy use of the fact that for odd $p$ every supercuspidal representation is dihedral. This fails for $p=2$ and one would have to treat the non-dihedral supercuspidal representations separately.

\subsection{Notation and pre-requests} \label{se:notation}

We assume the reader to be quite familiar with \cite{Sa15,Sa15_2}. For the sake of completeness we will briefly introduce necessary notation and some background material.

Throughout this paper we fix a non-archimedean local field $F$, which later on will be specialized to $F=\Q_p$ for some odd prime $p$. Let $\mathcal{O}$ be the ring of integers in $F$ and write $\p$ for the unique maximal ideal. We fix a uniformizer $\varpi$ and normalize the valuation $v$ such that $v(\varpi)=1$. Furthermore, we scale the absolute value such that $\abs{\varpi} = q^{-1}$ where $q  = \sharp F/\p$ is the number of elements in the residue field. We write $\mu$ (resp. $\mu^{\times}$) for the Haar measure on $(F,+)$ (resp. $(F^{\times},\cdot)$) and normalise it to satisfy $\mu(\mathcal{O})=1$ (resp. $\mu^{\times}(\mathcal{O}^{\times})=1$). These two measures satisfy the relation $\mu^{\times} = \frac{\zeta_F(1)}{\abs{\cdot}}\mu$ where $\zeta_F(s) = \frac{1}{1-q^{-s}}$ is the local zeta function of $F$.  

We fix an unramified additive character $\psi$ on $F$, in particular the exponent $n(\psi)=0$. In this case $\mu$ is normalized to be self dual with respect to the Fourier transform. By a multiplicative character we mean a continuous group homomorphism $F^{\times} \to \C^{\times}$. We define the sets
\begin{eqnarray}
	\mathfrak{X}_l &=& \{ \chi \text{ multiplicative character }\colon \chi(\varpi)=1 \text{ and } a(\chi) \leq l \},  \text{ and }\nonumber \\
	 \mathfrak{X}_l' &=& \{ \chi \in \mathfrak{X_l}\colon a(\chi) = l \}, \nonumber
\end{eqnarray}
where $a(\chi)$ denotes the log-conductor of $\chi$. Note that these definitions depend on the choice of $\varpi$. To each character we associate a $\epsilon$-factor $\epsilon(s,\chi)$ and a $L$-factor by
\begin{equation}
	L(s,\chi) = \begin{cases}
		\frac{1}{1-\chi(\varpi)q^{-s}} &\text{ if $\chi$ is unramified,}\\
		1 &\text{ else.}
	\end{cases} \nonumber	
\end{equation}
The $\epsilon$-factor depends on the additive character $\psi$. However, since we consider $\psi$ as fixed we drop this dependence from the notation. Many useful properties of $\epsilon$-factors can be found in \cite{Sc02}. The most important one for us is the link to the Gau\ss\  sum via the formula
\begin{equation}
	G(x,\mu) = \int_{\mathcal{O}^{\times}}\psi(xy)\mu(y)d^{\times}y = \begin{cases} 1 &\text{ if $\mu=1$ and $v(x)\geq 0$,} \\ -\zeta_F(1)q^{-1} &\text{ if $\mu=1$ and $v(x) = -1$,} \\ \zeta_F(1)q^{-\frac{a(\mu)}{2}}\epsilon(\frac{1}{2},\mu^{-1})\mu^{-1}(x) &\text{ if $\mu\neq 1$ and $v(x) =-a(\mu)$} \\ 0 &\text{ else.} \end{cases} \label{eq:evaluation_of_GS}
\end{equation}
This evaluation is taken from \cite[Lemma~2.3]{CS16} and will be used a lot throughout this text.

The letter $E$ is reserved for a quadratic space over $F$. Thus, basically $E$ is a quadratic extension of $F$ or it is simply the vector space $E=F\times F$. If we are dealing with a quadratic extension we let $e=e(E/F)$ be the ramification index and $f=f(E/F)$ be the degree of the residual extension. In particular we have $ef=2$. By $d=d(E/F)$ we denote the valuation of the discriminant of $E/F$, it satisfies $d=e-1$. The Galois group is $\text{Gal}_{E/F} = \{ 1, \sigma\}$. The norm and the trace are defined as usual by
\begin{equation}
	\Tr(z) = z+\sigma z\text{ and } \Norm(z) = z\cdot\sigma z. \nonumber 
\end{equation}
The Haar measure on $E$ will be normalised as follows:
\begin{eqnarray}
	\vol(\mathfrak{O},\mu_E) = q^{-\frac{d}{2}}, \nonumber
\end{eqnarray}
where $\mathfrak{O}$ is the ring of integers in $E$. The unique maximal ideal in $\mathfrak{O}$ is denoted by $\mathfrak{P}$ it will be generated by the uniformizer $\Omega$ of $E$. We will usually choose uniformizers such that $\Norm(\Omega) = \varpi^f$. Note that this determines a canonical valuation $v_E$ on $E$.

Further, let $\chi_{E/F}$ be the quadratic character on $F^{\times}$ which is trivial on $\Norm(E^{\times})$ and set
\begin{equation}
	\psi_E = \psi\circ \Tr. \nonumber
\end{equation}
By \cite[Lemma~2.3.1]{Sc02} we have
\begin{equation}
	n(\psi_E) = -\frac{d}{f}. \nonumber
\end{equation}
Which again implies that the Haar measure $\mu_E$ is normalised to be self dual with respect to $\psi_E$. Multiplicative characters on $E$ are usually denoted by $\xi$ and one can attach the same objects as we did over $F$.

If $E=F\times F$ we define the ring of integers to be $\mathfrak{O} = \mathcal{O}\times \mathcal{O}$ and the ideal $\mathfrak{P} = \p \times \p$. The Haar measure is simply the product measure  $\mu \times \mu$ and all multiplicative characters factor into two multiplicative characters on $F^{\times}$. To keep notation consistent we define
\begin{equation}
	\Tr(x_1,x_2) = x_1+x_2 \text{ and } \Norm((x_1,x_2)) = x_1x_2. \nonumber
\end{equation}

Next we define some standard subgroups of $G= GL_2(F)$ and equip them with measures. Let $Z$ be the center of $G$, 
\begin{equation}
	N = \left\{n(x)=\left(\begin{matrix}1&x\\0&1\end{matrix}\right)\colon x\in F \right\} \text{ and } A=\left\{a(y)=\left( \begin{matrix} y&0\\0&1 \end{matrix}\right)\colon y\in F^{\times} \right\}.\nonumber
\end{equation}
Further, let $K=GL_2(\mathcal{O})$ be a maximal compact subgroup of $G$ containing the closed subgroups
\begin{equation}
	K_1(n)=\left( \begin{matrix}1+\p^n & \mathcal{O} \\ \p^n & \mathcal{O} \end{matrix}\right)\cap K. \nonumber
\end{equation} 
As usual we define measures on $Z$, $A$ and $N$ via the identifications $Z=F^{\times}$, $A=F^{\times}$ and $N=F$. On $K$ we choose the probability Haar measure. This then defines a measure on $G = ZNAK$ via the Iwasawa decomposition. Another very useful decomposition of $G$ is 
\begin{equation}
	G = \bigsqcup_{t\in \Z} \bigsqcup_{0\leq l \leq n} \bigsqcup_{v\in \mathcal{O}^{\times}/(1+\p^{l_n})} ZNg_{t,l,v}K_1(n)  \label{eq:basic_decomp}
\end{equation}
with
\begin{equation} 
	g_{t,l,v} = \left( \begin{matrix} 0 & \varpi^t \\ -1 & -v\varpi^{-l}  \end{matrix} \right) \text{ and } l_n = \min(l,n-l)  \nonumber
\end{equation}
given in \cite[Lemma~2.13]{Sa15_2}.

Usually $\pi$ will denote a infinite dimensional, admissible, irreducible representation of $GL_2(F)$. Such a representation comes with several invariants. Namely, the log-conductor $n=a(\pi)$ and the central character $\omega_{\pi}$. We write $m=a(\omega_{\pi})$ for the log-conductor of the central character. The contragredient representation will be denoted by $\tilde{\pi}$. Attached to $\pi$ there are $L(s,\pi)$ and $\epsilon(\frac{1}{2},\pi)$ factors. Without loss of generality we may twist $\pi$ by an unramified character to ensure that $\omega_{\pi} \in \mathfrak{X}_m'$. Such representations are completely classified. More precisely we know each unitary, tempered, irreducible, admissible $\pi$ belongs to one of the following families.

\begin{enumerate}
\item \textbf{Twists of Steinberg:} $\pi = \chi St$, for some unitary character $\chi$ satisfying $\chi(\varpi)=1$. In this case we have $\omega_{\pi} = \chi^2$ and $a(\pi) = \max(1,2a(\chi))$. Furthermore, the $L$-factor as well as the $\epsilon$-factor are given by
\begin{equation}
	L(s,\pi) = \begin{cases} L(s,\abs{\cdot}^{\frac{1}{2}}) &\text{ if }\chi = 1, \\ 1 &\text{ if } \chi \neq 1,   \end{cases} \text{ and } \epsilon(\frac{1}{2}, \pi) = \begin{cases} -1 &\text{ if } \chi = 1, \\  \epsilon(\frac{1}{2},\chi)^2 &\text{ if } \chi \neq 1. \end{cases} \nonumber
\end{equation}

\item \textbf{Principal series:} $\pi = \chi_1\boxplus\chi_2$, for unitary characters $\chi_1$ and $\chi_2$. In particular, $a(\pi) = a(\chi_1)+a(\chi_2)$ and $\omega_{\pi} = \chi_1\chi_1$. Concerning the $L$-factor we know
\begin{equation}
	L(s,\pi) = L(s,\chi_1)L(s,\chi_2) \text{ and } \epsilon(\frac{1}{2},\pi) = \epsilon(\frac{1}{2},\chi_1)\epsilon(\frac{1}{2}, \chi_2). \nonumber
\end{equation}

\item \textbf{Supercuspidal representations:} If $\pi$ is supercuspidal, then $L(s,\pi)=1$ and all the other invariants are more difficult to describe. However, if $q$ is odd, we know that every supercuspidal representation is dihedral. In this case we find a quadratic extension $E/F$ and an unitary multiplicative character $\xi$ of $E^{\times}$ such that $\pi = \omega_{\xi}$ is the dihedral supercuspidal representation associated to $(E,\xi)$. We then find that $a(\pi) = fa(\xi)+d$ and 
\begin{equation}
	\epsilon(\frac{1}{2},\pi) = \gamma\epsilon(\frac{1}{2},\xi), \nonumber
\end{equation}
for some $\gamma \in S^1$, given in \cite[Section~2]{JL70}, depending only on $E$. The behaviour of $\pi$ under $GL_1$-twists is described by $\chi \pi = \omega_{\xi \cdot (\chi\circ \Norm)}$ and the central character is $\omega_{\pi} = \chi_{E/F} \cdot \xi\vert_{F^{\times}}$.
\end{enumerate}

This list can be extracted from \cite{JL70} and \cite{Sc02}. We will sometimes also allow principal series associated to non-unitary characters $\chi_1$ and $\chi_2$ since these appear as local components of Eisenstein series.

It is well known, that for $G$ each admissible, irreducible, infinite dimensional representation is generic. In other words it admits a unique $\psi$-Whittaker model $\mathcal{W}(\pi)$.  This Whittaker model contains the \textit{special functions on $G$} that we seek to understand. We will focus on studying the distinguished new vector $W_{\pi}$ which we normalise by $W_{\pi}(1)=1$. This vector is well understood on the diagonal. Indeed we have
\begin{equation}
	W_{\pi}\left(\left(\begin{matrix} v\varpi^t & 0 \\ 0 & 1 \end{matrix} \right) \right) = \begin{cases} 
		q^{-t(s+1)} &\text{ if $t\geq 0$ and $\pi =\abs{\cdot}^s St$ , } \\
		\chi_1(v\varpi^t)q^{-\frac{t}{2}} &\text{ if $t\geq 0$ and $\pi = \chi_1\boxplus \chi_2$ with $a(\chi_1)>a(\chi_2)=0$,}\\
		\omega_{\pi}(v) &\text{ if $t=0$ and $L(s,\pi)=1$,} \\
		0 &\text{ else.} 
	\end{cases} \label{eq:whitt_on_diag}
\end{equation}
for $t\in \Z$ and $v\in \mathcal{O}^{\times}$. This follows from \cite{Sc02}, is stated in this form in \cite[Lemma~2.5]{Sa15_2} and is proven in \cite[Lemma~2.10]{CS16}.

Finally, we define the constants $c_{t,l}(\mu)$ via finite Fourier expansion:
\begin{equation}
	W_{\pi}(g_{t,l,v}) = \sum_{\mu\in \mathfrak{X}_l}c_{t,l}(\mu)\mu(v). \label{eq:fourier_exp}
\end{equation}

The main tool to calculate them is using the basic identity given in \cite[Propositon~2.23]{Sa15_2}:
\begin{eqnarray}
	&&\sum_{t=-\infty}^{\infty} q^{(t+a(\mu\pi))(\frac{1}{2}-s)}c_{t,l}(\mu) \nonumber \\
	&&\qquad= \omega_{\pi}(-1)\epsilon(\frac{1}{2},\mu\pi)^{-1}\frac{L(s,\mu\pi)}{L(1-s,\mu^{-1}\omega_{\pi}^{-1}\pi)}\sum_{a=0}^{\infty} W_{\pi}(a(\varpi^a))q^{-a(\frac{1}{2}-s)}G(\varpi^{a-l},\mu^{-1}). \label{eq:basic_identity}
\end{eqnarray}
Note that the proof of this formula holds for any $l\geq 0$. Important is only that $\omega_{\pi}(\varpi) = \mu(\varpi) = 1$, otherwise one would have to introduce a simple unramified twist on the right hand side.

\section{Calculating the coefficients $c_{t,l(\mu)}$} \label{se:ctlmu}

In this section we will prove explicit formulas for the constants $c_{t,l}(\mu)$ defined above. The following calculations use nothing more than the \textit{basic identity} \eqref{eq:basic_identity} which holds in great generality. Thus, throughout this section, $F$ is any non-archimedean local field, $\pi$ is some (not necessarily unitary) irreducible, admissible representation with unitary central character $\omega_{\pi} \in \mathfrak{X}_n$. We split this section in subsections dealing with each type of $GL(2)$-representation on its own.

\subsection{Supercuspidal representations}

In this subsection we assume $\pi$ to be a supercuspidal representation. Because $L(s,\pi)=1$, the basic identity takes the very simple form
\begin{equation}
	\sum_{t=-\infty}^{\infty} q^{(t+a(\mu\pi))(\frac{1}{2}-s)} c_{t,l}(\mu) = \omega_{\pi}(-1) \epsilon(\frac{1}{2},\mu\pi)^{-1}G(\varpi^{-l},\mu^{-1}). \label{eq:basic_identity_L1}
\end{equation}
By comparing coefficients we arrive at
\begin{equation}
	c_{t,l}(\mu) = \begin{cases}
		\omega_{\pi}(-1)\frac{G(\varpi^{-l},\mu^{-1})}{\epsilon(\frac{1}{2},\mu\pi)} &\text{ if $t=-a(\mu\pi)$,} \\ 0 &\text{else.}
	\end{cases} \nonumber
\end{equation}
Evaluation of the Gau\ss\  sum yields
\begin{equation}
	c_{t,l}(\mu) = \begin{cases}
		\epsilon(\frac{1}{2},\omega_{\pi}^{-1}\pi) &\text{ if $l=0$, $t=-n$, and $\mu =1$,} \\
		-\zeta_F(1)q^{-1}\epsilon(\frac{1}{2},\omega_{\pi}^{-1}\pi)&\text{ if $l=1$, $t=-n$, and $\mu=1$,} \\
		\zeta_F(1)q^{-\frac{l}{2}}\epsilon(\frac{1}{2},\mu)\epsilon(\frac{1}{2},\mu^{-1} \omega_{\pi}^{-1}\pi)&\text{ if $\mu\in \mathfrak{X}_l'$, $t=-a(\mu\pi)$, and $l>0$,} \\
		0 &\text{else.}
	\end{cases} \nonumber
\end{equation}

These expressions are already implicit in \cite[Section~2.7]{Sa15_2}.

\subsection{Twists of Steinberg}

Let $\pi = \chi St$. Recall from Section~\ref{se:notation} that all attached invariants can be described by the invariants attached to $\chi$. We will see that this case gets slightly more complicated since $L(s,St)$ is not trivial. 

\begin{lemma} \label{lm:twis_of_ST}
Let $l\in \N_0$ and $\mu\in \mathfrak{X}_l$. If $\pi = \chi St$ for $\chi \neq 1$ then the constants $c_{t,l}(\mu)$ are given by
\begin{equation}
	c_{t,l}(\mu) = \begin{cases}
		\epsilon(\frac{1}{2},\mu^{-1}\chi^{-1})^2G(\varpi^{-l},\mu^{-1}) &\text{ if $\mu\neq \chi^{-1}$ and $t=-2a(\mu\chi)$,} \\
		q^{-1}G(\varpi^{-l},\mu^{-1}) &\text{ if $\mu=\chi^{-1}$ and $t=-2$,} \\
		-\zeta_F(2)^{-1}q^{-1-t}G(\varpi^{-l},\mu^{-1})&\text{ if $\mu=\chi^{-1}$ and $t>-2$,} \\
		0 &\text{ else.}
	\end{cases} \nonumber
\end{equation}
If $\pi= St$, then we have
\begin{equation}
	c_{t,l}(\mu) = \begin{cases}
		\zeta_F(1)q^{-l+\frac{a(\mu)}{2}}\epsilon(\frac{1}{2},\mu)^{-1} = q^{a(\mu)-l}G(-\varpi^{-a(\mu)},\mu) &\text{ if $\mu\neq 1$ and $t=-l-a(\mu)$,} \\
		-q^{-t-1} &\text{ if $\mu=1$,$l=0$, and $t\geq -1$,} \\
		q^{-t-2l} &\text{ if $\mu=1$, $l\geq1$, and $t\geq -l$,} \\
		-\zeta_F(1)q^{-l}  &\text{ if $\mu=1$, $l\geq1$, and $t=-l-1$,} \\
		0 &\text{ else.}
	\end{cases} \nonumber
\end{equation}
\end{lemma} 
\begin{proof}
If $\chi\neq 1$ and $\mu\neq \chi^{-1}$, then the basic identity is as in \eqref{eq:basic_identity_L1}. It is easy to compare coefficients.

We continue by considering $\chi\neq 1$ and $\mu = \chi^{-1}$. In this case we have
\begin{equation}
	\sum_{t=-\infty}^{\infty} q^{(t+1)(\frac{1}{2}-s)}c_{t,l}(\chi^{-1}) = -\omega_{\pi}(-1)\frac{L(s,\abs{\cdot}^{\frac{1}{2}})}{L(1-s,\abs{\cdot}^{\frac{1}{2}})}G(\varpi^{-l},\chi) \nonumber
\end{equation}
For suitable $s$ one can expand
\begin{equation}
	\frac{L(s,\abs{\cdot}^{\frac{1}{2}})}{L(1-s,\abs{\cdot}^{\frac{1}{2}})} = -q^{-\frac{3}{2}+s}+\zeta_F(2)^{-1}\sum_{a=0}^{\infty}q^{-\frac{a}{2}-as}. \nonumber
\end{equation}
Inserting this expression together with the explicit evaluation of the Gau\ss\  sum and comparing coefficients completes this case.

Next we look at $\chi = 1$ and $\mu\neq 1$. Evaluating the Whittaker function and the Gau\ss\  sum using \eqref{eq:whitt_on_diag} and \eqref{eq:evaluation_of_GS} yields a basic identity of the form
\begin{equation}
	\sum_{t=-\infty}^{\infty} q^{(t+a(\mu\pi))(\frac{1}{2}-s)} c_{t,l}(\mu) = \zeta_F(1) q^{-(l-a(\mu))(\frac{3}{2}-s)-\frac{a(\mu)}{2}}\epsilon(\frac{1}{2},\mu)^{-1}. \nonumber
\end{equation}
Note that $a(\mu\pi) = 2a(\mu)$. Also $a(\pi)=n=1$. Since we are assuming $\mu\neq 1$, we must have $l\geq 1$. Thus, we are done after comparing coefficients.

Now we consider  $\chi = \mu = 1$ and $l=0$. In this case the basic identity simplifies to
\begin{equation}
	\sum_{t=-\infty}^{\infty} q^{(t+1)(\frac{1}{2}-s)}c_{t,0}(1) =- L(s,\abs{\cdot}^{\frac{1}{2}}) = -\sum_{a=0}^{\infty} q^{-a(\frac{1}{2}+s)}. \nonumber	
\end{equation}
Again we can compare coefficients.

The last case to check is  $\chi = \mu = 1$ and $l\geq 1$. One checks that the basic identity becomes
\begin{equation}
	\sum_{t=-\infty}^{\infty} q^{(t+1)(\frac{1}{2}-s)}c_{t,l}(1) = -q^{-l(\frac{3}{2}-s)}\zeta_F(1) \frac{1-q^{\frac{1}{2}-s}}{1-q^{-\frac{1}{2}-s}}. \nonumber
\end{equation}
Expanding the fraction into geometric series and comparing coefficients concludes the proof.
\end{proof}

\subsection{Irreducible principal series}

In this section we assume $\pi = \chi_1\boxplus \chi_2$. Again all the invariants of $\pi$ can be described by $\chi_1$ and $\chi_2$. Recall that we are always assuming $\omega_{\pi} \in \mathfrak{X}_n$. In particular $\chi_1\chi_2(\varpi) = 1$. In the proofs of \cite[Proposition~2.39,2.40]{Sa15_2} some values for $c_{t,l}(\mu)$ have been computed. In this section we refine and complete these computations in order to list precise expressions for all possible $t$, $l$ and $\mu$.

\begin{lemma} \label{lm:general_prinz_ser}
Let $\pi =\chi_1\boxplus \chi_2$ with $a(\chi_i) >0$ for $i=1,2$. If $\chi_1\vert_{\mathcal{O}^{\times}} \neq \chi_2\vert_{\mathcal{O}^{\times}}$ then
\begin{equation}
	c_{t,l}(\mu) = \begin{cases}&\epsilon(\frac{1}{2},\mu^{-1}\chi_1^{-1})\epsilon(\frac{1}{2},\mu^{-1}\chi_2^{-1})G(\varpi^{-l},\mu^{-1}) \\
			&\qquad\qquad \text{ if $a(\mu\chi_1),a(\mu\chi_2)\neq 0$ and $t=-a(\mu\chi_1)-a(\mu\chi_2)$,} \\
		&-q^{-\frac{1}{2}}\chi_i(\varpi^{-1})\epsilon(\frac{1}{2},\mu^{-1}\chi_j^{-1})G(\varpi^{-l},\mu^{-1}) \\
			&\qquad\qquad\text{ if $a(\mu\chi_j) \neq a(\mu\chi_i)=0$ for $\{j,i\} = \{1,2\}$, and $t=-a(\mu\chi_j)-1$, } \\
		&\zeta_F(1)^{-2}q^{-\frac{t}{2}}\chi_i(\varpi^{t})G(\varpi^{-a(\mu\chi_j)},\mu\chi_j)G(\varpi^{-l},\mu^{-1})\\
			&\qquad\qquad \text{ if $a(\mu\chi_j) \neq a(\mu\chi_i)=0$ for $\{j,i\} = \{1,2\}$, and $t\geq-a(\mu\chi_j)$, } \\
		&0 \\
			&\qquad\qquad\text{ else.}
	\end{cases} \nonumber
\end{equation}
If $\chi_1\vert_{\mathcal{O}^{\times}} = \chi_2\vert_{\mathcal{O}^{\times}}$ then
\begin{equation}
	c_{t,l}(\mu) = \begin{cases}&\epsilon(\frac{1}{2},\mu^{-1}\chi_1^{-1})\epsilon(\frac{1}{2},\mu^{-1}\chi_2^{-1})G(\varpi^{-l},\mu^{-1}) \\
			&\qquad\qquad \text{ if $a(\mu\chi_1),a(\mu\chi_2)\neq 0$ and $t=-a(\mu\chi_1)-a(\mu\chi_2)$,} \\
		&  q^{-1}  G(\varpi^{-l},\mu^{-1})\\
			&\qquad\qquad\text{ if $a(\mu\chi_1)=a(\mu\chi_2)=0$, and $t=-2$, } \\
		& -q^{-\frac{1}{2}}\zeta_F(1)^{-1}  G(\varpi^{-l},\mu^{-1})(\chi_1(\varpi)+\chi_2(\varpi))\\
			&\qquad\qquad\text{  if $a(\mu\chi_1)=a(\mu\chi_2)=0$, and $t=-1$,}\\
		& q^{-\frac{t}{2}}G(\varpi^{-l},\mu^{-1})\bigg(-q^{-1}\zeta_F(1)^{-1}(\chi_1(\varpi^{t+2})+\chi_2(\varpi^{t+2}))+\zeta_F(1)^{-2}\sum_{k=0}^t 	\chi_1(\varpi^k)\chi_2(\varpi^{t-k}) \bigg)\\
			&\qquad\qquad\text{  if $a(\mu\chi_1)=a(\mu\chi_2)=0$, and $t\geq 0$,}\\
		&0 \\
			&\qquad\qquad\text{ else.}
	\end{cases} \nonumber
\end{equation}
\end{lemma}
\begin{proof}
The case $\mu\neq \chi_1,\chi_2$ is straightforward. So we start by considering $\chi_1 = \mu^{-1}\abs{\cdot}^c \neq \chi_2\abs{\cdot}^{2c}$. The same calculation will work when we interchange the role of $\chi_1$ and $\chi_2$ in this case. The basic identity reads
\begin{equation}
	\sum_{t=-\infty}^{\infty} q^{(t+a(\mu\pi))(\frac{1}{2}-s)}c_{t,l}(\mu)=G(\varpi^{-l},\mu^{-1})\epsilon(\frac{1}{2},\mu^{-1}\chi_2^{-1}) \frac{L(s,\abs{\cdot}^c)}{L(1-s,\abs{\cdot}^{-c})}. \nonumber
\end{equation}
Expanding the quotient of $L$-factors into a power series and recalling $\chi_1(\varpi) = q^{-c}$ yields
\begin{equation}
	\frac{L(s,\abs{\cdot}^c)}{L(1-s,\abs{\cdot}^{-c})} = -q^{-1}\chi_1(\varpi)^{-1} q^s+\zeta_F(1)^{-1}\sum_{a=0}^{\infty} \chi_1(\varpi^a)q^{-as}. \label{eq:expansion_of_LFrac}
\end{equation}
Inserting this into the basic identity enables us to compare coefficients and conclude this case.

In the end we consider the situation where both, $\chi_1$ and $\chi_2$, are unramified twists of $\mu$. Since the central character is trivial on the uniformizer we have $\chi_1(\varpi) = \chi_2^{-1}(\varpi)=\abs{\varpi}_{F}^c$ for some $c \in \C$. The basic identity becomes
\begin{equation}
	\sum_{t=-\infty}^{\infty} q^{t(\frac{1}{2}-s)}c_{t,l}(\mu) = \omega_{\pi}(-1) G(\varpi^{-l},\mu^{-1})\frac{L(s,\abs{\cdot}^c)L(s,\abs{\cdot}^{-c})}{L(1-s,\abs{\cdot}^c)L(1-s,\abs{\cdot}^{-c})}. \nonumber
\end{equation} 
We use \eqref{eq:expansion_of_LFrac} twice to obtain
\begin{eqnarray}
	&&\frac{L(s,\abs{\cdot}^c)L(s,\abs{\cdot}^{-c})}{L(1-s,\abs{\cdot}^c)L(1-s,\abs{\cdot}^{-c})} = q^{-2}q^{2s} - q^{-1}\zeta_F(1)^{-1}(\chi_1(\varpi)+\chi_2(\varpi))q^s \nonumber \\
	&& + \sum_{a=0}^{\infty} \bigg(-q^{-1}\zeta_F(1)^{-1}(\chi_1(\varpi^{a+2})+\chi_2(\varpi^{a+2}))+\zeta_F(1)^{-2}\sum_{l=0}^a \chi_1(\varpi^l)\chi_2(\varpi^{a-l}) \bigg) q^{-as}. \nonumber
\end{eqnarray}
Now we may compare coefficients and conclude the proof.
\end{proof}

\begin{lemma} \label{lm:degenerate_principal}
Let $\pi =\chi_1\boxplus \chi_2$ with $n=a(\chi_1) > a(\chi_2) = 0$. Then we have
\begin{equation}
	c_{t,l}(\mu) = \begin{cases} 
		& \mu(-1)\zeta_F(1)q^{-\frac{l}{2}}\epsilon(\frac{1}{2},\mu^{-1}\omega_{\pi}^{-1})\chi_2(\varpi^{a(\mu\omega_{\pi})-l}) \\
			&\qquad\qquad\text{ if $\mu\neq \omega_{\pi}^{-1}$, $l>0$, and $t=-a(\mu\omega_{\pi})-l$, } \\
		&  \chi_2(\varpi^{t+2n})q^{-\frac{1}{2}(t+n)}\epsilon(\frac{1}{2},\omega_{\pi}^{-1}) \\
			&\qquad\qquad\text{ if $l=0$, and $t\geq -n$,} \\
		&-\omega_{\pi}(-1)\zeta_F(1)q^{-\frac{1}{2}(l+1)}\chi_2(\varpi^{-l+1}) \\
			&\qquad\qquad\text{ if $\mu=\omega_{\pi}^{-1}$ and $t= -l-1$,} \\
		&\omega_{\pi}(-1) q^{-\frac{t}{2}-l}\chi_2(\varpi^{-t-2l}) \\
			&\qquad\qquad\text{ if $\mu=\omega_{\pi}^{-1}$ and $t\geq -l$,} \\
		&0 \qquad \text{ else.}
	\end{cases} \nonumber
\end{equation}
Note that this also covers the case $a(\chi_2)>a(\chi_1)=0$.
\end{lemma}
The proof uses the same ideas as previous proof. Thus, we leave the details to the reader.

\section{The shape of Whittaker new vectors on $g_{t,l,v}$} \label{se:Wh_supp}

In this section we use our description of $c_{t,l}(\mu)$ to evaluate Whittaker new vectors $W_{\pi}$ on the special matrices $g_{t,l,v}$. We will obtain expressions for the local Whittaker functions featuring several ($p$-adic) special functions. These functions are analogues of well known special functions that appear in the archimedean representation theory of $GL_2$ and are interesting in their own right. Note that 
\begin{equation}
	W_{\pi}\left(\left(\begin{matrix} z & 0 \\ 0 & z \end{matrix}\right)n(x)g_{t,l,v}k\right) = \omega_{\pi}(z)\psi(x)W_{\pi}(g_{t,l,v}) \text{ for } k\in K_{1}(n). \nonumber
\end{equation}
Thus, by \eqref{eq:basic_decomp} understanding $W_{\pi}$ on $g_{t,l,v}$ determines it on the whole of $GL_2$.

Probably the most prestigious function we will encounter is the Kloosterman sum and its twisted generalisation (generalised Sali\'e sums). We define
\begin{equation}
	S_{\chi}(A,B,m) = \int_{\mathcal{O}^{\times}} \chi(x) \psi\left(\frac{Ax+Bx^{-1}}{\varpi^{m}}\right)d^{\times} x \text{ for } A,B \in \mathcal{O} \text{ and } m \in \N_0. \label{eq:def_Salie_sum}
\end{equation}
These are indeed representations of the classical Kloosterman sums as p-adic oscillatory integrals. One notes that due to the normalisation of $d^{\times}x$ the trivial bound is 1. If $\chi=1$, we will drop it from the notation. A more general function is 
\begin{equation}
	K(\xi,A,B) = \int_{\mathfrak{O}^{\times}}\xi(x) \psi(\Tr(Ax)+B\Norm(x))d_Ex, \label{eq:def_hyper_geo_sum}
\end{equation}
which we associated to a multiplicative character $\xi \colon E^{\times} \to S^1$ on some quadratic space $E$ over $F$. Here $A\in E$ and $B\in F$.

However, the focus of this section is to describe the support of the Whittaker functions as precisely as possible. This will help us later on to exclude several choices for $t$ and $l$ for which $W_{\pi}(g_{t,l,v})$ vanishes.

We consider each type of representation on its own. The case of supercuspidal representations has already been considered in \cite[Proposition~2.30]{Sa15_2}. However, in many cases the sums of $\epsilon$-factors simplify considerably.

The results in this section hold for any non-archimedean field $F$ and any irreducible, admissible, unitary, infinite dimensional representation $\pi$ with central character $\omega_{\pi} \in \mathfrak{X}_n$.

\subsection{Dihedral supercuspidals}

Here we will derive an expression of the Whittaker new vector for dihedral supercuspidal representations which goes beyond the one given in \cite[Proposition~2.30]{Sa15_2}. The following results hold for any dihedral representation even those for $2\mid q$. However, in this case not every supercuspidal representation is dihedral.   

If $\pi$ is dihedral supercuspidal, then so is $\tilde{\pi}$. Thus we can find a quadratic extension $E/F$ and a multiplicative character $\xi$ such that $\tilde{\pi}  = \omega_{\xi}$. We now use the properties of dihedral supercuspidal representations summarised in Section~\ref{se:notation} to calculate the Whittaker function.

\begin{lemma} \label{lm:supp_wh_supercusp}
If $\pi$ is dihedral supercuspidal and $k=\max(n,2l)$, then we have
\begin{equation}
	W_{\pi}(g_{t,l,v}) = \begin{cases}
		\epsilon(\frac{1}{2},\tilde{\pi}) &\text{ if } l=0 \text{ and } t=-k,\\
		\gamma q^{-\frac{t}{2}}K(\xi^{-1},\Omega^{\frac{t}{f}},v\varpi^{-l}) &\text{ if }l=\frac{n}{2} \text{ and }-k \leq t<0, \text{ or }l\neq 0,\frac{n}{2} \text{ and } t= -k,\\	
		0 &\text{ else.}
	\end{cases}	 \nonumber
\end{equation}
\end{lemma}
\begin{proof}
First, we apply \cite[Lemma~1.1.1]{Sc02} in the setting of $E$ and obtain
\begin{eqnarray}
	&&\epsilon(\frac{1}{2},\xi\cdot(\mu^{-1}\circ \Norm)) \nonumber \\
	&& \quad = q^{\frac{f}{2}(n(\psi_E)-a(\xi\cdot(\mu^{-1}\circ \Norm))} \int_{\Omega^{n(\psi_E)-a(\xi\cdot(\mu^{-1}\circ\Norm))}\mathfrak{O}^{\times}}\xi^{-1}(x)\mu( \Norm(x)) \psi_E(x)d\mu_E(x). \nonumber
\end{eqnarray} 
Note that if $t=-a(\mu\pi)$ we compute
\begin{equation}
	n(\psi_E)-a(\xi\cdot(\mu^{-1}\circ \Norm)) = \frac{t}{f}. \nonumber
\end{equation}

With all this at hand we proceed calculating our Whittaker new vector. We obtain
\begin{eqnarray}
	W_{\pi}(g_{t,l,v}) &=& \sum_{t=-a(\mu\pi)} \epsilon(\frac{1}{2},\mu^{-1}\tilde{\pi})G(\varpi^{-l},\mu^{-1})\mu(v) \nonumber \\
	&=& \gamma q^{\frac{t}{2}}\int_{\Omega^{\frac{t}{f}}\mathfrak{O}^{\times}}\int_{\mathcal{O}^{\times}}\xi^{-1}(x)\psi(\Tr(x)+v\varpi^{-l}y)\sum_{\mu\in\mathfrak{X}_l}\mu(\Norm(x)y^{-1})d_Exd^{\times}y \nonumber \\
	&=&\gamma q^{-\frac{t}{2}}\int_{\mathfrak{O}^{\times}}\int_{\mathcal{O}^{\times}}\xi^{-1}(\Omega^{\frac{t}{f}}x)\psi(\Tr(\Omega^{\frac{t}{f}}x)+v\varpi^{-l}y)\sum_{\mu\in\mathfrak{X}_l}\mu(\Norm(x)y^{-1})d_Exd^{\times}y \nonumber \\
	&=&\gamma q^{-\frac{t}{2}}\int_{\mathfrak{O}^{\times}}\xi^{-1}(x)\psi(\Tr(\Omega^{\frac{t}{f}}x)+v\varpi^{-l}\Norm(x))d_Ex. \nonumber
\end{eqnarray}
\end{proof}

\subsection{Twists of Steinberg}

Throughout this subsection $E$ will denote the quadratic space $F\times F$. In this case any multiplicative character $\xi$ on $E^{\times}$ factors in $\xi = (\chi_1\circ \text{pr}_1)\cdot (\chi_2\circ\text{pr}_2)$ for two characters $\chi_1$ and $\chi_2$ of $F^{\times}$. However, at the moment we will only encounter the special situation where $\chi_1=\chi_2$. In other words, $\xi$ factors through the norm map. We will compute the local Whittaker functions in terms of $K(\xi,A,B)$ and other well known exponential sums.

We start of with the simplest case.

\begin{lemma} \label{lm:steinberg_exp}
For $\pi = St$ we have
\begin{equation}
	W_{\pi}(g_{t,l,v}) = \begin{cases} 
		-q^{-t-1} &\text{ if $t\geq -1$ and $l=0$,} \\
		q^{-t-2l} &\text{ if $t\geq -l$  and $l\geq 0$, }\\ 
		 q^{-t-2l} \psi(-\varpi^{l+t}v^{-1}) &\text{ if $l\geq 1$ and $-2l\leq t\leq -l-1$,} \\
		0 &\text{ else.}
	\end{cases} \nonumber
\end{equation}
\end{lemma}
\begin{proof}
By the definition of $c_{t,l}(\mu)$ we have
\begin{equation}
	W_{\pi}(g_{t,l,v}) =  \sum_{\mu\in \mathfrak{X}_l} c_{t,l}(\mu) \mu(v). \nonumber
\end{equation}
We will now insert the result from Lemma~\ref{lm:twis_of_ST}. The interesting cases are obviously $t\leq -l-1$ and $l\geq 1$. Inserting the values for $c_{t,l}(\mu)$ calculated above yields
\begin{eqnarray}
	W(g_{t,l,v}) &=& q^{-t-2l}\sum_{\mu \in \mathfrak{X}_{-l-t}'} G(-\varpi^{l+t},\mu)\mu(v) -  \delta_{t=-l-1}\zeta_{F}(1)q^{-l} \nonumber \\
	&=& q^{-t-2l}\sum_{\mu\in \mathfrak{X}_{-l-t}} G(-\varpi^{-l-t}v^{-1}, \mu) \nonumber \\
	&=&  q^{-t-2l}\sharp\mathfrak{X}_{-l-t} \int_{1+\varpi^{-l-t}\mathcal{O}} \psi(-\varpi^{l+t}v^{-1}y) d^{\times}y \nonumber\\
	&=&  q^{-t-2l}\psi(-\varpi^{l+t}v^{-1}). \nonumber
\end{eqnarray}
\end{proof}

We now move on the the slightly more complicated situation of $\pi = \chi St$ for a non trivial $\chi$.

\begin{lemma} \label{lm:twist_steinberg_exp}
Let $\pi = \chi St$ where $\chi$ is a character such that $a(\chi)>0$ and $\chi(\varpi) = 1$. If $a(\chi)\geq 1$ and $l\neq a(\chi) = \frac{n}{2}$, we have
\begin{equation}
	W(g_{t,l,v}) = \begin{cases}
		\epsilon(\frac{1}{2},\tilde{\pi}) &\text{ if } t = -n \text{ and } l=0, \\
		q^{-\frac{t}{2}}\zeta_F(1)^{-2} K(\chi\circ \Norm,(\varpi^{\frac{t}{2}},\varpi^{\frac{t}{2}}),v\varpi^{-l}) &\text{ if } t = -\max(n,2l) \text{ and } 0<l<n, \\ 
		\chi^2(-v^{-1})\psi(-v^{-1}\varpi^{-l}) &\text{ if } t = -2l \text{ and } l\geq n,, \\
		0 &\text{ else.}
	\end{cases} \nonumber
\end{equation}

Finally, if $a(\chi)>1$ and $l=a(\chi)=\frac{n}{2}$, we have the degenerate situation
\begin{equation}
	W(g_{t,l,v}) = \begin{cases}
		-\zeta_F(2)^{-1}\zeta_F(1)q^{-1-\frac{a(\chi)}{2}-t} \chi(v^{-1})\epsilon(\frac{1}{2},\chi^{-1}) &\text{ if }t>-2, \\
		q\zeta_F(1)^{-2}K(\chi\circ \Norm, (\varpi^{-1},\varpi^{-1}),v\varpi^{-l}) &\text{ if } t=-2 \text{ and } l=1, \\
		\chi(v^{-1})\epsilon(\frac{1}{2},\chi^{-1})\zeta_F(1)^{-1}q^{1-\frac{a(\chi)}{2}} S(1,-b_{\chi}v^{-1},1)  &\text{ if $t=-2$ and $l>1$,} \\
		q^{-\frac{t}{2}}\zeta_F(1)^{-2} K(\chi\circ \Norm,(\varpi^{\frac{t}{2}},\varpi^{\frac{t}{2}}),v\varpi^{-l})  &\text{ if }-2l \leq t <-2 \text{ even}, \\
		0 &\text{ else.}
	\end{cases} \nonumber
\end{equation}
\end{lemma}
\begin{proof}
We start by expanding 
\begin{equation}
	W_{\pi}(g_{t,l,v}) =  \sum_{\mu\in \mathfrak{X}_l} c_{t,l}(\mu) \mu(v). \nonumber
\end{equation}
Using Lemma~\ref{lm:twis_of_ST} we first observe that if $t>-2$, the only character $\mu\in \mathfrak{X}_l$ with nonzero $c_{t,l}(\mu)$ is $\mu = \chi^{-1}$. Thus we can move on to the other cases.

If $l=0$, the only character to consider is $\mu =1$ which contributes only if $t=-2a(\chi)=-a(\pi)$.  Thus from now on we will assume $l>0$.

If $t= -2$, we obtain
\begin{equation}
	W_{\pi}(g_{-2,l,v}) = q^{-1}G(\varpi^{-l},\chi)\chi^{-1}(v) + \sum_{\substack{\mu\in \mathfrak{X}_l, \\ a(\mu\chi)=1}} \epsilon(\frac{1}{2},\mu^{-1}\chi^{-1})^2G(\varpi^{-l},\mu^{-1})\mu(v). \nonumber
\end{equation}
Reversing the evaluation of the Gau\ss\  sum given in \eqref{eq:evaluation_of_GS} yields the compact form
\begin{eqnarray}
	W_{\pi}(g_{-2,l,v}) &=& q\zeta_F(1)^{-2} \sum_{\substack{\mu \in \mathfrak{X}_l, \\ a(\mu\chi) \leq 1}} G(\varpi^{-1},\mu\chi)^2 G(v\varpi^{-l},\mu^{-1}) \label{eq:twist_st_-2}\\
	&=&  q\zeta_F(1)^{-2}\sum_{\mu \in \mathfrak{X}_l} G(\varpi^{-1},\mu\chi)^2 G(v\varpi^{-l},\mu^{-1}). \nonumber
\end{eqnarray}
To exploit cancellation in the $\mu$-average we write the Gau\ss\  sums as an integral. This leads to
\begin{equation}
	W_{\pi}(g_{-2,l,v}) = q\zeta_F(1)^{-2} \int_{ (\mathcal{O}^{\times})^3}\chi(y_1y_2)\psi(y_1\varpi^{-1}+y_2\varpi^{-1}+y_3 v \varpi^{-l}) \sum_{\mu\in\mathfrak{X}_l} \mu(y_1y_2y_3^{-1})d^{\times}y_3d^{\times}y_2d^{\times}y_1. \nonumber
\end{equation}
We observe
\begin{equation}
	 \sum_{\mu\in\mathfrak{X}_l} \mu(y_1y_2y_3^{-1}) =\begin{cases}
		\sharp\mathfrak{X}_l &\text{ if } y_1y_2y_3^{-1}\in 1+\varpi^l \mathcal{O}, \\
		0 &\text{ else.}
	\end{cases} \nonumber
\end{equation}
Using this to simplify the integral we obtain
\begin{eqnarray}
	&&W_{\pi}(g_{-2,l,v}) \nonumber \\
	&& \quad = q\zeta_F(1)^{-2}\sharp\mathfrak{X}_l\vol(1+\varpi^l \mathcal{O},d^{\times}) \int_{(\mathcal{O}^{\times})^2} \chi(y_1y_2) \psi(y_1\varpi^{-1}+y_2\varpi^{-1}+y_1y_2 v \varpi^{-l}) d^{\times} y_1 d^{\times}y_2 \nonumber \\
	&& \quad = \zeta_F(1)^{-2}q K(\chi\circ \Norm, (\varpi^{-1},\varpi^{-1}),v\varpi^{-l}). \nonumber
\end{eqnarray}
If $l=1=a(\chi)$, we will leave this expression as it is. However, in the other cases we write
\begin{equation}
	W_{\pi}(g_{-2,l,v}) = q\zeta_F(1)^{-2} \int_{\mathcal{O}^{\times}} \chi(y_1)\psi(y_1\varpi^{-1})G(\varpi^{-1}+y_1v\varpi^{-l},\chi)d^{\times}y_1 \nonumber
\end{equation}
instead. Here we have to consider two different cases. First observe that if $l>1$, the Gau\ss\  sum vanishes unless $l=a(\chi)$ (which would also imply $a(\chi)>1$). Thus, if $l=a(\chi)$, we obtain
\begin{eqnarray}
	W_{\pi}(g_{-2,a(\chi),v}) &=& q^{1-\frac{a(\chi)}{2}}\zeta_F(1)^{-1}\epsilon(\frac{1}{2},\chi^{-1}) \int_{\mathcal{O}^{\times}} \chi(y_1)\chi^{-1}(y_1v+\varpi^{a(\chi)-1})\psi(y_1\varpi^{-1})d^{\times}y_1 \nonumber\\
	 &=& \chi(v^{-1})\epsilon(\frac{1}{2},\chi^{-1})\zeta_F(1)^{-1}q^{1-\frac{a(\chi)}{2}} \int_{\mathcal{O}^{\times}} \chi^{-1}(1+v^{-1}y_1^{-1}\varpi^{a(\chi)-1})\psi(y_1\varpi^{-1})d^{\times}y_1 \nonumber \\
	 &=& \chi(v^{-1})\epsilon(\frac{1}{2},\chi^{-1})\zeta_F(1)^{-1}q^{1-\frac{a(\chi)}{2}} S(1,-b_{\chi}v^{-1},1) \text{ for some } b_{\chi}\in \mathcal{O}^{\times}. \nonumber
\end{eqnarray}
In the last step we observed that, if $1<l=a(\chi)$, we have $a(\chi)-1 \geq \frac{a(\chi)}{2}$ and thus Lemma~\ref{lm:char_trick} below can be used to find the desired $b_{\chi} \in \mathcal{O}^{\times}$.

If $l=1$, the situation is completely different. In this case we have
\begin{equation}
	W_{\pi}(g_{-2,1,v}) = q^{1-\frac{a(\chi)}{2}}\zeta_F(1)^{-1}\epsilon(\frac{1}{2},\chi^{-1}) \int_{-v^{-1}+\varpi^{1-a(\chi)}\mathcal{O}^{\times}} \chi(y_1)\chi^{-1}(y_1v+\varpi^{a(\chi)-1})\psi(y_1\varpi^{-1})d^{\times}y_1. \nonumber
\end{equation}
If $a(\chi)>1$, we can rewrite this as follows.
\begin{eqnarray}
	&&\int_{-v^{-1}+\varpi^{1-a(\chi)}\mathcal{O}^{\times}} \chi(y_1)\chi^{-1}(y_1v+\varpi^{a(\chi)-1})\psi(y_1\varpi^{-1})d^{\times}y_1 \nonumber \\
	&=& \zeta_F(1)\psi(-v^{-1}\varpi^{-1})\int_{\varpi^{1-a(\chi)}\mathcal{O}^{\times}} \chi(y_1-v^{-1})\chi^{-1}(y_1v-1+\varpi^{a(\chi)-1})\psi(y_1\varpi^{-a(\chi)})\frac{dy_1}{\abs{y_1-v^{-1}}} \nonumber  \\
	&=&\psi(-v^{-1}\varpi^{-1})\int_{\mathcal{O}^{\times}}\underbrace{\chi(y_1\varpi^{1-a(\chi)}-v^{-1})\chi^{-1}(y_1\varpi^{1-a(\chi)}v-1+\varpi^{a(\chi)-1})}_{=1}\psi(y_1\varpi^{-a(\chi)})d^{\times}y_1 = 0. \nonumber
\end{eqnarray}
This implies that, if $a(\chi)>1$ we have
\begin{equation}
	W_{\pi}(g_{-2,1,v})=0. \nonumber
\end{equation}

Similarly, if $t<-2$, we get
\begin{eqnarray}
	W_{\pi}(g_{t,l,v}) &=& \sum_{\substack{\mu\in\mathfrak{X}_l, \\ t = -2a(\mu\chi)}} \epsilon(\frac{1}{2},\mu^{-1}\chi^{-1})^2G(\varpi^{-l},\mu^{-1}) \mu(v) \label{eq:tgeq2_original} \\
	&=&  q^{-\frac{t}{2}}\zeta_F(1)^{-2} \sum_{\mu\in \mathfrak{X}_l} G(\varpi^{\frac{t}{2}},\mu\chi)^2G(\varpi^{-l}v,\mu^{-1}). \nonumber
\end{eqnarray}
At this point we expand the Gau\ss\  sums into integrals and use cancellation between the characters $\mu\in\mathfrak{X}_l$. This yields
\begin{eqnarray}
	W_{\pi}(g_{t,l,v}) &=&  q^{-\frac{t}{2}}\zeta_F(1)^{-2} \int_{\mathcal{O}^{\times}}\int_{\mathcal{O}^{\times}} \chi(y_1y_2)\psi(y_1\varpi^{\frac{t}{2}}+y_2\varpi^{\frac{t}{2}}+y_1y_2v\varpi^{-l})d^{\times}y_1d^{\times}y_2\nonumber  \\
	&=& q^{-\frac{t}{2}}\zeta_F(1)^{-2} K(\chi\circ \Norm,(\varpi^{\frac{t}{2}},\varpi^{\frac{t}{2}}),v\varpi^{-l}). \label{eq:complete_sume_ts} 
\end{eqnarray}

In several cases we can obtain further simplification by writing
\begin{equation}
	W_{\pi}(g_{t,l,v})= q^{-\frac{t}{2}}\zeta_F(1)^{-2} \int_{\mathcal{O}^{\times}}\chi(y_1) \psi(y_1\varpi^{\frac{t}{2}})G(\varpi^{\frac{t}{2}}+y_1v\varpi^{-l},\chi)d^{\times}y_1. \nonumber
\end{equation}
Now let us assume for the moment that $-l\neq \frac{t}{2}$. Then obviously $\abs{y_1^{-1}\varpi^{\frac{t}{2}}+v\varpi^{-l}} = \max(q^l,q^{-\frac{t}{2}})$. And the Gau\ss\  sum vanishes whenever $\max(-\frac{t}{2},l) \neq a(\chi)$. This implies that if the local Whittaker function is non zero, we must have $a(\chi) = l > -\frac{t}{2}$ or $a(\chi) = -\frac{t}{2}>l$.

At last we consider $l=-\frac{t}{2}$. In this case we deduce from  \eqref{eq:tgeq2_original} that $a(\mu\chi)=l$. Since the support of the Gau\ss\  sum implies $\mu\in \mathfrak{X}_l'$ we can assume $l\geq a(\chi)$. Whenever $l<n$ we are happy with the expression given in \eqref{eq:complete_sume_ts}. But if $l\geq n$, we can evaluate the Gau\ss\  sum and calculate
\begin{eqnarray}
	&&W_{\pi}(g_{t,l,v}) \nonumber \\
	&& \quad = q^{\frac{-t-a(\chi)}{2}}\zeta_F(1)^{-1} \epsilon(\frac{1}{2},\chi^{-1}) \int_{-v^{-1}+\varpi^{l-a(\chi)}\mathcal{O}^{\times}} \chi(y_1)\chi^{-1}(1+y_1 v) \psi(y_1\varpi^{\frac{t}{2}}) d^{\times}y_1 \nonumber \\
	&& \quad = q^{\frac{a(\chi)}{2}}\zeta_{F}(1)^{-1}\epsilon(\frac{1}{2},\chi^{-1})\chi(v^{-1})\psi(-v^{-1}\varpi^{-l}) \int_{\mathcal{O}^{\times}} \underbrace{\chi(\varpi^{l-a(\chi)}-v^{-1}y_1^{-1})}_{=\chi^{-1}(-vy_1)}\psi(y_1\varpi^{-a(\chi)})d^{\times}y_1.\nonumber
\end{eqnarray} 
This reduces to a Gau\ss\  sum which can be evaluated and almost everything cancels out. 
\end{proof}

\subsection{Irreducible Principal Series}

In this subsection we will treat the Whittaker functions associated to irreducible principal series representations. In this case we work with the quadratic space $E=F\times F$. For two characters $\chi_1$ and $\chi_2$ on $F^{\times}$ we write $\chi_1\otimes \chi_2$ for the obvious character on $E^{\times}$. 

We start with the most degenerate case. Obviously we are talking about $\pi=\omega_{\pi}\abs{\cdot}^{s}\boxplus \abs{\cdot}^{-s}$. For notational simplicity we will sometimes write $\chi_1 = \omega_{\pi}\abs{\cdot}^{s}$ and $\chi_2=\abs{\cdot}^{-s}$. In this case we exploit that $K(\omega_{\pi}^{-1}\otimes 1,\cdot,\cdot)$ degenerates to a one dimensional object which is related to the incomplete Gau\ss\  sum
\begin{equation} 
	G_l(y, \omega_{\pi}^{-1})= \int_{1+\varpi^l\mathcal{O}}\omega_{\pi}^{-1}(x)\psi(yx) d^{\times} x, \nonumber
\end{equation}
which is a special case of $K(\omega_{\pi}^{-1}\otimes 1,\cdot,\cdot)$.

\begin{lemma} \label{lm:unbalanced_ps}
Let $\pi=\omega_{\pi}\abs{\cdot}^{s}\boxplus \abs{\cdot}^{-s}$. Then $n=a(\pi)=a(\omega_{\pi})$ and we have
\begin{equation}
	W_{\pi}(g_{t,l,v}) = \begin{cases}
		\chi_2(\varpi^{t+2n})q^{-\frac{t+n}{2}}\epsilon(\frac{1}{2},\omega_{\pi}^{-1}) &\text{ if $l=0$ and $t\geq -n$,}\\
		\omega_{\pi}(-v^{-1})\chi_2(\varpi^{-t-2l}) \zeta_F(1)^{-1}q^{-\frac{t}{2}}G_l(-v^{-1}\varpi^{t+l},\omega_{\pi}) &\text{ if $0<l<n$ and $t=-n-l$,} \\
		\omega_{\pi}(-v^{-1})q^{-\frac{t+2l}{2}}\chi_2(\varpi^{-t-2l})\psi(-v^{-1}\varpi^{t+l}) &\text{ if $l \geq n>0$ and $t\geq -2l$,} \\
		0 &\text{ else.}
	\end{cases} \nonumber
\end{equation}
\end{lemma}
\begin{proof}
The strategy is as usual to use Lemma~\ref{lm:degenerate_principal} in the finite Fourier expansion, \eqref{eq:fourier_exp}, of $W_{\pi}$. One sees directly that, if $l=0$, there is only one contribution. The same is true for $l>0$ and $t>-l$. Therefore, let us from now on assume $l>0$ and $t<-l$.
We obtain
\begin{eqnarray}
	W_{\pi}(g_{t,l,v}) &=& \sum_{\substack{\mu\in\mathfrak{X}_l\setminus \{\omega_{\pi}^{-1}\}, \\ t+l=-a(\mu\omega_{\pi})}}\zeta_F(1)q^{-\frac{l}{2}}\chi_2(\varpi^{-t-2l})\epsilon(\frac{1}{2},\mu^{-1}\omega_{\pi}^{-1})\mu(-v) \nonumber \\
	&&-\delta_{\substack{t=-l-1,\\ \mu=\omega_{\pi}^{-1}}}\omega_{\pi}(-v^{-1})\zeta_F(1)q^{-\frac{l+1}{2}}\chi_2(\varpi^{-l+1}) \nonumber\\
	&=& \omega_{\pi}(-v^{-1})\chi_2(\varpi^{-t-2l})q^{-\frac{t+2l}{2}}\sum_{\mu\in\mathfrak{X}_l}G(-v^{-1}\varpi^{t+l},\mu\omega_{\pi}) \label{eq:unbalanced_ps_as_epss}
\end{eqnarray}
Now we can write the Gau\ss\  sum as integral and take the character sum inside the integral. This leads to
\begin{equation}
	W_{\pi}(g_{t,l,v}) = \omega_{\pi}(-v^{-1})\chi_2(\varpi^{-t-2l}) \zeta_F(1)^{-1}q^{-\frac{t}{2}} \int_{1+\varpi^l\mathcal{O}}\omega_{\pi}(x)\psi(-xv^{-1}\varpi^{t+l})d^{\times}x. \nonumber
\end{equation}

If $l\geq n$, then the character is constant in the range of integration. We obtain
\begin{eqnarray}
	W_{\pi}(g_{t,l,v}) &=& \omega_{\pi}(-v^{-1})\chi_2(\varpi^{-t-2l}) \zeta_F(1)^{-1}q^{-\frac{t}{2}} \int_{1+\varpi^l\mathcal{O}}\psi(-xv^{-1}\varpi^{t+l})d^{\times}x \nonumber \\
	&=& \omega_{\pi}(-v^{-1})\chi_2(\varpi^{-t-2l})q^{-\frac{t}{2}} \int_{1+\varpi^l\mathcal{O}}\psi(-xv^{-1}\varpi^{t+l})dx \nonumber \\
	&=&  \omega_{\pi}(-v^{-1})\chi_2(\varpi^{-t-2l})q^{-\frac{t}{2}-l}\psi(-v^{-1}\varpi^{t+l}) \underbrace{\int_{\mathcal{O}}\psi(-xv^{-1}\varpi^{t+2l})dx}_{=\delta_{t\geq -2l}}. \nonumber 
\end{eqnarray}

\end{proof}

Next we will look at another degenerate situation.

\begin{lemma}\label{lm:equiv_ps_expwh}
Let $\pi = \chi\abs{\cdot}^s\boxplus \chi\abs{\cdot}^{-s}$ for a non trivial character $\chi$. If $a(\chi)\geq 1$ and $l\neq a(\chi) = \frac{n}{2}$, we have
\begin{equation}
	W(g_{t,l,v}) = \begin{cases}
		\epsilon(\frac{1}{2},\tilde{\pi}) &\text{ if } t = -n \text{ and } l=0, \\
		q^{-\frac{t}{2}}\zeta_F(1)^{-2} K(\chi\circ \Norm,(\varpi^{\frac{t}{2}},\varpi^{\frac{t}{2}}),v\varpi^{-l}) &\text{ if } t = -\max(n,2l) \text{ and } 0<l<n, \\ 
		\chi^2(-v^{-1})\psi(-v^{-1}\varpi^{-l}) &\text{ if } t = -2l \text{ and } l\geq n,, \\
		0 &\text{ else.}
	\end{cases} \nonumber
\end{equation}

Finally, if $a(\chi)>1$ and $l=a(\chi)=\frac{n}{2}$, we have the degenerate situation
\begin{equation}
	W(g_{t,l,v}) = \begin{cases}
		q^{-\frac{t}{2}}G(\varpi^{-l},\chi)\big(-q^{-1}\zeta_F(1)^{-1}(q^{s(t+2)}+q^{-s(t+2)})   +&\zeta_F(1)^{-2}\sum_{k=0}^t q^{s(t-2k)} \big) 	\\
		&\text{ if $t\geq 0$,} \\
		-q^{-\frac{1}{2}}\zeta_F(1)^{-1}G(\varpi^{-l}, \chi)(q^s+q^{-s}) &\text{ if $t = -1$,} \\
		q\zeta_F(1)^{-2}K(\chi\circ \Norm, (\varpi^{-1},\varpi^{-1}),v\varpi^{-l}) &\text{ if }t=-2 \text{ and } l=1, \\
		\chi(v^{-1})\epsilon(\frac{1}{2},\chi^{-1})\zeta_F(1)^{-1}q^{1-\frac{a(\chi)}{2}} S(1,-b_{\chi}v^{-1},1)  &\text{ if $t=-2$ and $l>1$,} \\
		q^{-\frac{t}{2}}\zeta_F(1)^{-2} K(\chi\circ \Norm,(\varpi^{\frac{t}{2}},\varpi^{\frac{t}{2}}),v\varpi^{-l})  &\text{ if }-2l \leq t <-2 \text{ even}, \\
		0 &\text{ else.}
	\end{cases} \nonumber
\end{equation}
\end{lemma}
Not surprisingly we will encounter many similarities to the case $\pi=\chi St$. Thus we will be quite brief in the proof.
\begin{proof}
Interesting situations occur only for $t\leq -2$. For those cases we have
\begin{eqnarray}
	W_{\pi}(g_{t,l,v})  &=& \sum_{\substack{\mu\in \mathfrak{X}_l, \\ t=-2a(\mu\chi)}} \epsilon(\frac{1}{2},\mu^{-1}\chi^{-1})^2 G(v\varpi^{-l},\mu^{-1}) +\delta_{t=-2} q^{-1} G(v\varpi^{-l},\chi) \nonumber \\
	&=& \zeta_F(1)^{-2}q^{-\frac{t}{2}}\sum_{\mu\in \mathfrak{X}_l} G(\varpi^{\frac{t}{2}},\mu\chi)^2 G(v\varpi^{-l},\mu^{-1}). \nonumber
\end{eqnarray}
We have seen this exact sum already in \eqref{eq:twist_st_-2} and \eqref{eq:tgeq2_original}. The rest of the proof is left to the reader.
\end{proof}

Finally, we treat the 'generic' irreducible principal series.
 
\begin{lemma} \label{lm:generic_ps_shape_Wf}
Let $\pi = \chi_1\abs{\cdot}^s\boxplus \chi_2\abs{\cdot}^{-s}$such that $a(\chi_1)\geq a(\chi_2)>0$. If $l\not\in \{a(\chi_1),a(\chi_2)\}$, then
\begin{equation}
	W_{\pi}(g_{t,l,v}) = \begin{cases}
		&\epsilon(\frac{1}{2}, \tilde{\pi})\\
			&\qquad\qquad\qquad\qquad\qquad\qquad\qquad\qquad\qquad\text{if $l=0$ and $t=-n$,} \\
		&\zeta_F(1)^{-2}q^{-\frac{t}{2}}q^{s(a(\chi_1)-a(\chi_2))}K(\chi_1\otimes \chi_2,(\varpi^{-a(\chi_1)},\varpi^{-a(\chi_2)}),v\varpi^{-l}) \\
			&\qquad\qquad\qquad\qquad\qquad\qquad\qquad\qquad\qquad\text{ if $t=-n$ and $0<l<a(\chi_2)$,}\\
		&\zeta_F(1)^{-2}q^{-\frac{t}{2}}q^{s(a(\chi_1)-l)}K(\chi_1\otimes \chi_2,(\varpi^{-a(\chi_1)},\varpi^{-l}),v\varpi^{-l})\\
			&\qquad\qquad\qquad\qquad\qquad\qquad\qquad\qquad\qquad\text{ if $t=-l-a(\chi_1)$ and $a(\chi_2)<l<a(\chi_1)$,}\\
		&\zeta_F(1)^{-2}q^{-\frac{t}{2}}K(\chi_1\otimes \chi_2,(\varpi^{-l},\varpi^{-l}),v\varpi^{-l}) \\
			&\qquad\qquad\qquad\qquad\qquad\qquad\qquad\qquad\qquad\text{ if $t=-2l$ and $a(\chi_1)<l<n$,}\\
		&\omega_{\pi}(-v^{-1})\psi(-v^{-1}\varpi^{-l}) \\
			&\qquad\qquad\qquad\qquad\qquad\qquad\qquad\qquad\qquad\text{ if $t=-2l$ and  $l\geq n$ ,}\\
		&0 \\
			&\qquad\qquad\qquad\qquad\qquad\qquad\qquad\qquad\qquad\text{ else.}
	\end{cases} \nonumber
\end{equation}

If $l=a(\chi_i) \neq a(\chi_j)$ with $\{i,j\} = \{1,2\}$, then
\begin{equation}
	W_{\pi}(g_{t,l,v}) = \begin{cases}
		&\zeta_F(1)^{-2}q^{-\frac{t}{2}}q^{s(2a(\chi_1)+t)}K(\chi_1\otimes \chi_2,(\varpi^{-a(\chi_1)},\varpi^{a(\chi_1)+t}),v\varpi^{-l})\\		
				&\qquad\qquad\qquad\qquad\qquad\qquad \qquad\qquad\text{ if $l=a(\chi_2)$ and $-n\leq t< -a(\chi_1)$,}\\
		&\zeta_F(1)^{-2}q^{-\frac{t}{2}}q^{st}G(\varpi^{-a(\chi_1)},\chi_2^{-1}\chi_1)G(v\varpi^{a(\chi_2)},\chi_2) \\
				&\qquad\qquad\qquad\qquad\qquad\qquad \qquad\qquad\text{ if $l=a(\chi_2)$ and $t\geq -a(\chi_1)$,} \\
		&\zeta_F(1)^{-2}q^{-\frac{t}{2}}q^{s(-t-2l)}K(\chi_1\otimes \chi_2,(\varpi^{l+t},\varpi^{-l}),v\varpi^{-l})\\		
				&\qquad\qquad\qquad\qquad\qquad\qquad \qquad\qquad\text{ if $l=a(\chi_1)\neq a(\chi_2)$ and $-2l\leq t<-a(\chi_1)$,}\\
		&\zeta_F(1)^{-2}q^{-\frac{t}{2}}q^{-st}G(\varpi^{-a(\chi_1)},\chi_2\chi_1^{-1})G(v\varpi^{a(\chi_1)},\chi_1) \\
				&\qquad\qquad\qquad\qquad\qquad\qquad \qquad\qquad\text{ if $l=a(\chi_1)$ and $t\geq -a(\chi_1)$,} \\
		&0 \\
				&\qquad\qquad\qquad\qquad\qquad\qquad \qquad\qquad\text{ else.}
	\end{cases} \nonumber
\end{equation}

And if $l=a(\chi_1) = a(\chi_2)$, then
\begin{equation}
	W_{\pi}(g_{t,l,v}) = \begin{cases}
		&\zeta_F(1)^{-2}q^{-\frac{t}{2}} \chi_1(v^{-1})\epsilon(\frac{1}{2},\chi_2^{-1})\sum_{l_2=1}^{l} q^{s(-t-2l_2)}K(\chi_1\otimes \chi_2, (\varpi^{t+l_2},\varpi^{-l_2}), v\varpi^{-l})  \\
			&\qquad+ \delta_{t\geq -a(\chi_1^{-1}\chi_2)} \zeta_F(1)^{-2}q^{-\frac{t}{2}}\bigg[ G(v\varpi^{-l},\chi_1)G(\varpi^{-a(\chi_1^{-1}\chi_2)},\chi_1^{-1}\chi_2)q^{-st}  \\
			&\qquad\qquad\qquad\qquad\qquad\qquad\qquad+ G(v\varpi^{-l},\chi_2)G(\varpi^{-a(\chi_2^{-1}\chi_1)},\chi_2^{-1}\chi_1)q^{st}\bigg] \\
				&\qquad\qquad\qquad\qquad\qquad\qquad \qquad\qquad\qquad\qquad \qquad\qquad \text{ if $-n \leq t\leq -2$},\\
		&\sum_{\{i,j\}=\{1,2\}}\chi_i(v^{-1}\varpi^{t+l})q^{-\frac{t+a(\chi_i^{-1}\chi_j)+l}{2}}\epsilon(\frac{1}{2},\chi_i)\epsilon(\frac{1}{2},\chi_i^{-1}\chi_j), \\
				&\qquad\qquad\qquad\qquad\qquad\qquad \qquad\qquad\qquad\qquad \qquad\qquad\text{ if }t>-2\\
		&0 \\
				&\qquad\qquad\qquad\qquad\qquad\qquad \qquad\qquad\qquad\qquad \qquad\qquad\text{ else.}
	\end{cases} \nonumber
\end{equation}
\end{lemma}
\begin{proof}
Let us consider the interesting situation $l>0$. For $t>-2$ the only contribution comes from the characters $\mu\in \{\chi_1^{-1},\chi_2^{-1}\}$ which is easily written down. We thus assume $t\leq -2$. Applying the usual tricks we end up with
\begin{eqnarray}
	W_{\pi}(g_{t,l,v})  &=&  \zeta_F(1)^{-2}q^{-\frac{t}{2}} \sum_{\substack{t=-l_1-l_2, \\ l_1,l_2>0}}q^{s(l_1-l_2)}\sum_{\mu\in \mathfrak{X}_l}G(\varpi^{-l_1},\mu\chi_1)G(\varpi^{-l_2},\mu\chi_2)G(v\varpi^{-l},\mu^{-1}) \nonumber \\
	&&+ \delta_{t\geq -a(\chi_1^{-1}\chi_2)} \zeta_F(1)^{-2}q^{-\frac{t}{2}}\bigg[ G(v\varpi^{-l},\chi_1)G(\varpi^{-a(\chi_1^{-1}\chi_2)},\chi_1^{-1}\chi_2)q^{-st} \nonumber \\
	&&\qquad\qquad\qquad\qquad+ G(v\varpi^{-l},\chi_2)G(\varpi^{-a(\chi_2^{-1}\chi_1)},\chi_2^{-1}\chi_1)q^{st}\bigg]. \nonumber
\end{eqnarray}
As earlier we can compress the $\mu$-sum to $K(\chi_1\otimes \chi_2, (\varpi^{-l_1},\varpi^{-l_2}), v\varpi^{-l})$. This gives
\begin{eqnarray}
	W_{\pi}(g_{t,l,v})  &=&  \zeta_F(1)^{-2}q^{-\frac{t}{2}} \sum_{\substack{t=-l_1-l_2, \\ l_1,l_2>0}}q^{s(l_1-l_2)} K(\chi_1\otimes \chi_2, (\varpi^{-l_1},\varpi^{-l_2}), v\varpi^{-l})\nonumber \\
	&&+ \delta_{t\geq a(\chi_1^{-1}\chi_2)} \zeta_F(1)^{-2}q^{-\frac{t}{2}}\bigg[ G(v\varpi^{-l},\chi_1)G(\varpi^{-a(\chi_1^{-1}\chi_2)},\chi_1^{-1}\chi_2)q^{-st} \nonumber \\
	&&\qquad\qquad\qquad\qquad+ G(v\varpi^{-l},\chi_2)G(\varpi^{-a(\chi_2^{-1}\chi_1)},\chi_2^{-1}\chi_1)q^{st}\bigg].  \label{eq:orig_generic_contribution} 
\end{eqnarray}
The statement follows from standard computation which we leave to the reader. 
\end{proof}

\begin{rem}
If $a(\chi_1)=a(\chi_2)=l$, one can see that Lemma~\ref{lm:generic_ps_shape_Wf} fails to provide a simple integral representation of $W_{\pi}(g_{t,l,v})$. Instead we end up having a sum of several integrals. However, later it will turn out that all but one are zero. 
\end{rem}

\section{Evaluation of oscillatory integrals}

In this Intermezzo we provide the necessary tools which are needed in the next section to estimate the numerous oscillatory integrals. Basically, this relies on rephrasing well known methods used to evaluate complete exponential sums with large moduli. In our framework these methods will be closely related to the classical method of stationary phase over the real numbers. The attentive reader will notice many similarities to well established results from \cite{BM15} and \cite{IK04}. In some results we will restrict ourselves to odd $q$. This simplifies many calculations avoiding problems coming from finite fields of characteristic 2.

We start by recalling how the $p$-adic logarithm can be used to transform multiplicative oscillations in additive ones.

\begin{lemma} \label{lm:char_trick}
Let $e=e(F/\Q_p)$ be the absolute ramification index. We define $\kappa_F = \lceil\frac{e}{p-1}\rceil$. Let $\chi$ be an multiplicative character such that $a(\chi) \geq \kappa_F$. There is $b_{\chi} \in \mathcal{O}^{\times}$ uniquely determined modulo  $\p^{a(\chi)-\kappa_F}$ such that
\begin{equation}
	\chi(1+z\varpi^{\kappa_F}) = \psi\left(\frac{b_{\chi}}{\varpi^{a(\chi)}}\log_F(1+z\varpi^{\kappa_F})\right) \text{ for all } z\in \mathcal{O}. \nonumber
\end{equation}
Furthermore, if $\frac{a(\chi)}{3} \leq \alpha \in \N$, then there is $b_{\xi} \in \mathcal{O}^{\times}$ such that
\begin{equation}
	\chi(1+z\varpi^{\alpha}) = \psi\left( \frac{b_{\chi}}{\varpi^{a(\chi)}}\left(z\varpi^{\alpha}-\frac{z^2}{2}\varpi^{2\alpha}\right)\right) \text{ for all } z\in \mathcal{O}. \nonumber
\end{equation}
In particular, if $\frac{a(\chi)}{2}\leq \alpha \in \N$, then
\begin{equation}
	\chi(1+z\varpi^{\alpha}) = \psi(z b_{\xi}\varpi^{\alpha-a(\chi)}) \text{ for all } z\in \mathcal{O}. \nonumber
\end{equation}
Note that in the last two cases we do not make any assumption on $\kappa_F$.
\end{lemma}
The proof of this result can be found in \cite{BM15}.

Next, we will recall some results on quadratic congruences. The following lemma is basically \cite[Lemma~9.6]{KL13}.

\begin{lemma} \label{lm:quad_cong}
Let $a,b,c \in \mathcal{O}$. We set
\begin{eqnarray}
	S &=& \{ x\in \mathcal{O}/\p^n\colon ax^2+bx+c \in \p^n \}, \nonumber\\
	\Delta &=& b^2-4ac = \Delta'\varpi^{\delta_0} \text{ for } \Delta'\in \mathcal{O}^{\times}. \nonumber
\end{eqnarray}
If $v(a)=0$, we have
\begin{equation}
	S= \left\{ -\frac{b}{2a}\pm \frac{Y}{2a} \varpi^{\delta} + \alpha \varpi^{n-\delta} \colon \alpha \in \mathcal{O}/\p^{\delta}\right\} \label{eq:para_quad_cong}
\end{equation}
with
\begin{equation}
	Y = \begin{cases} 0 &\text{ if } \delta_0 \geq n,\\  Y_0 &\text{ if $Y_0^2 = \Delta'$ and $\delta_0<n$ is even,} \end{cases} \text{ and } \delta = \begin{cases} \lfloor \frac{n}{2} \rfloor &\text{ if $\delta_0\geq n$,}  \\ \delta' &\text{ if } \delta_0 = 2\delta' <n. \end{cases} \nonumber
\end{equation}
In particular,
\begin{equation}
	 \sharp S \leq 2 q^{\delta}. \nonumber
\end{equation}
If $v(a)>0$ and $v(b)=0$, we have $\sharp S = 1$. Furthermore, the solution $x_0\in S$ has valuation $v(c)$.
\end{lemma}

Finally, we will briefly introduce the method of stationary phase and try to gain some intuition for the size of oscillatory integrals. Let us start by introducing some notation. For a Schwartz function $\Phi$ with support in $\mathcal{O}^n$ and a function $f\colon \text{supp}(\Phi)\to \mathcal{O}$ we define
\begin{equation}
	S_{\Phi}(f,m) = \int_{\mathcal{O}^{n}} \Phi(\du{x})\psi(\varpi^{-m}f(\du{x}))d\du{x}, \label{eq:def_of_S(fm)}
\end{equation}
where $d\du{x}$ denotes integration with respect to the product measure on $\mathcal{O}^{n}$. If $\Phi = \mathbbm{1}_{\mathcal{O}^n}$, we drop it from the notation. The method of stationary phase provides a very general tool to understand the size of such integrals.

For $m=1$ the basic estimates rely on algebraic methods and are highly non trivial. Indeed, $S(f,1)$ reduces to a sum over the finite field $\mathcal{O}/\p$. In the one dimensional situation we have the following very strong bound due to Weil, \cite{We48}. Let $g(x) = \prod_{i=1}^d (x-\xi_i)^{a_i}$ be a rational function with coefficients in $\mathcal{O}/\p$. Furthermore let $\chi$ be a multiplicative character of $(\mathcal{O}/\p)^{\times}$ such that no $a_i$ is a multiple of the order of $\chi$. Then we have 
\begin{equation}
	\abs{\sum_{\substack{x\in \mathcal{O}/\p, \\ x\neq \xi_i \text{ for } 1\leq i \leq d}} \chi(g(x))\psi(\varpi^{-1}f(x))} \leq (N+d-1)\sqrt{q}, \label{eq:Weils_bound}  
\end{equation}
for each polynomial $f\in (\mathcal{O}/\p)[x]$ of degree $N$ satisfying $f(0)=0$. For $n>1$ we have to involve some heavy machinery and we need some additional assumptions. Let us assume that $f\in \mathcal{O}[X_1,\dots,X_n]$ has degree $d_f$ co-prime to $p$. We write $d_{f,\p}$ for the degree of the reduced polynomial $\tilde{f} \in (\mathcal{O}/\p)[X_1,\dots,X_n]$ and assume that the homogeneous part of degree $d_{f,\p}$ defines a smooth projective hypersurface. Then \cite[Example 19.(5)]{Ko10} yields
\begin{equation}
	\abs{S(f,1)} \leq (d_{f,\p}-1)^nq^{-\frac{n}{2}}. \label{eq:Deligns_bound}
\end{equation}

If $m>1$, the situation is completely different because there are several elementary methods that can be used for the evaluation. The basic idea is parallel to the classical method of stationary phase. More precisely one will split the integral in suitable pieces each of which can be expressed in terms of the  Gau\ss\  sum
\begin{equation}
	G(A\varpi^{-s},B) = \int_{\mathcal{O}^n}\psi({}^{t}x A x \varpi^{-s}+B\cdot x)dx \text{ for }A\in GL_n(\mathcal{O}) \text{ and }B\in F^n, \nonumber
\end{equation}
which can be evaluated in great generality. A basic result in this direction is the following.

\begin{lemma} \label{lm:eval_quad_GS}
Let $\rho \geq 0$,  $A\in \mathcal{O}^{\times}$, and $B\in F$. Assuming $q$ to be odd we have
\begin{equation}
	G(A\varpi^{-\rho},B) = \begin{cases} 
		q^{-\frac{\rho}{2}}\gamma_F(A,\rho) \psi(-\frac{\varpi^{\rho}B^2}{4A}) &\text{ if } B\in \p^{-\rho},\\
		0 &\text{ else} 
	\end{cases} \nonumber
\end{equation}
where
\begin{equation}
	\gamma_F(A,\rho) =\begin{cases} \chi_F(A)\epsilon(\frac{1}{2},\chi_F) &\text{ if }\rho > 0  \text{ is odd}, \\ 1 &\text{ if }\rho = 0\text{ or } \rho \text{ is even}.\end{cases} \nonumber	
\end{equation}

Let $A\in\text{Mat}_{2\times 2}(\mathcal{O})$ be a symmetric matrix and let $B\in F^2$. Then
\begin{equation}
	\abs{G(\frac{\varpi^{-\rho}}{2}A, B)} \leq \begin{cases}
		q^{-\frac{\text{rk}(A_{\p})}{2}} &\text{ if $\rho = 1$ and } \varpi B\in \mathcal{O}^{2}, \\
		1 &\text{ if $\rho = 0$ and }B\in \mathcal{O}^2, \\
		0 &\text{ else} 
	\end{cases} \nonumber
\end{equation}
where $A_{\p}$ is the image of $A$ in $A\in\text{Mat}_{2\times 2}(\mathcal{O}/\p)$.
\end{lemma}
The one-dimensional statement is essentially \cite[Lemma~6]{BM15}.  In \cite{DF97} one can find a very general evaluation of multi-dimensional Gau\ss\  sums over $\Q_p$ which carries over to general $F$. But since we will make heavy use of the two dimensional case we provide a quick proof.
\begin{proof}
Here we only deal with the two dimensional case. First we write $A_{\p}$ for the image of $A$ in $\text{Mat}_{2\times 2}(\mathcal{O}/\p)$. Since $\rho \in \{0,1\}$ the quadratic Gau\ss\  sum depends only on $A_{\p}$ and not on $A$. Next we observe that if $\rho = 0$ or $A_{\p} = 0$ then we are simply dealing with a linear sum and the statement is obvious. Thus, assume $\rho =1$ and $A_{\p} \neq 0$. First we write
\begin{equation}
	A_{\p} =  \left( \begin{matrix} a & b \\ b & c 	\end{matrix}\right) \text{ for } a,b,c\in \mathcal{O}/\p. \nonumber
\end{equation}
If $a\neq 0$, we have
\begin{equation}
	A_{\p} = \left( \begin{matrix} 1 & 0 \\ ba^{-1} & 1 \end{matrix}\right)\left( \begin{matrix} a & 0 \\ 0 & a^{-1}\det(A_{\p}) 	\end{matrix}\right)\left( \begin{matrix} 1 & a^{-1}b \\ 0 & 1 	\end{matrix}\right). \nonumber
\end{equation}  
Therefore, by making a linear change of variables yields
\begin{eqnarray}
	G\left(\frac{\varpi^{-1}}{2}A,B\right) &=& G\left(\frac{\varpi^{-1}}{2}\left(\begin{matrix} a& 0 \\ 0 &a^{-1}\det(A_{\p})	\end{matrix} \right), \left( \begin{matrix} 1& 0 \\ -a^{-1}b & 1\end{matrix}\right) B \right) \nonumber \\
	&=& G\left(\frac{\varpi^{-1}a}{2}, B_1\right)G\left(\frac{\varpi^{-1}\det(A_{\p})}{2a}, B_2-\frac{b}{a}B_1 \right). \nonumber
\end{eqnarray}
In particular, we can evaluate the remaining one dimensional Gau\ss\  sums and obtain
\begin{equation}
	G\left(\frac{\varpi^{-1}}{2}A,B\right) = \begin{cases}
		\gamma(A_{\p}) \psi\left( -\frac{\omega B_1^2}{2a}\right)q^{-\frac{1}{2}} \\
		\qquad\qquad\text{ if } \det(A_{\p}) = 0,\text{ } B_1,B_2 \in \p^{-1} \text{ and } B_2-\frac{b}{a}B_1 \in \mathcal{O}, \\
		\gamma(A_{\p}) \psi\left(\frac{-\varpi}{2\det(A_{\p})}\left(aB_1^2-2bB_1B_2+cB_2^2\right)\right)q^{-1} \\
		\qquad\qquad\text{ if } \det(A_{\p})\in (\mathcal{O}/\p)^{\times},\text{ and } B_1,B_2 \in \p^{-1} , \\
		0 \\
		\qquad\qquad\text{ else}
	\end{cases} \nonumber
\end{equation}  
with 
\begin{equation}
	\gamma(A_{\p}) = \begin{cases}
		\chi_F(\frac{a}{2})\epsilon(\frac{1}{2},\chi_F) &\text{ if } \text{rk}(A_{\p}) = 1,\\
		\chi_F(\det(A_{\p}))\epsilon(\frac{1}{2},\chi_F)^2 &\text{ if } \text{rk}(A_{\p}) = 2.
	\end{cases} \nonumber
\end{equation} 

If $a=0$ but $c\neq 0$, then the argument is essentially the same, one simply exchanges the roles of $a$ and $c$ as well as $B_1$ and $B_2$.

If $a=c=0$, then we must have $b\neq 0$. Observing
\begin{equation}
	\left(\begin{matrix} 0 & b\\ b & 0 \end{matrix}\right) = 	\left(\begin{matrix} -1 & 1\\ 1 & 1 \end{matrix}\right)	\left(\begin{matrix} \frac{b}{2} & 0\\ 0 & \frac{b}{2} \end{matrix}\right)	\left(\begin{matrix} -1 & 1\\ 1 & 1 \end{matrix}\right)	\nonumber
\end{equation}
and making a linear change of variables yields
\begin{eqnarray}
	G\left(\frac{\varpi^{-1}}{2}A,B\right) &=& G(\varpi^{-1}b,-B_1+B_2)G(\varpi^{-1}b,B_1+B_2) \nonumber \\
	&=& \gamma(A_{\p})\psi(\frac{-\varpi}{2b}(B_1^2+B_2^{2}))q^{-1}. \nonumber
\end{eqnarray}
The bounds for Gau\ss\  sums are special cases of these explicit evaluations.
\end{proof}

In the one dimensional situation over $F=\Q_p$ an estimate for $S(f,m)$ allowing very general phase functions $f$ is given in \cite[(5.3)]{CZ02}. We now translate this result in our setting.

\begin{lemma} \label{lm:quoted_qp_bound}
Let $F=\Q_p$ for $p>2$. Furthermore, let $f$ be a polynomial, with degree $d_{\p}>0$ modulo $\p$. If $\tau=v(f')$ and every $\alpha$ solving the critical point congruence 
\begin{equation}
	\varpi^{-\tau}f'(\alpha) \in \p \label{eq:critical_point_cong}
\end{equation}
has multiplicity less then $M$, then we have
\begin{equation}
	\abs{S(f,m)} \leq (d_{\p}-1) q^{-\frac{1}{M+1}(m-\tau)} \nonumber 
\end{equation}
for all $m\geq \tau+2$.
\end{lemma}
\begin{proof}
This follows immediately from \cite[(5.3)]{CZ02} after normalizing correctly. 
\end{proof}

Points satisfying the congruence \eqref{eq:critical_point_cong} are called \textit{critical points}. If they have multiplicity one, they are referred to as \textit{non-degenerate critical points}. The contribution of non-degenerate critical points is well behaved. Indeed, if there are only such critical points one can evaluate the corresponding oscillatory integral explicitly. On the other hand, if there are critical points with multiplicity bigger than one, the situation becomes more complicated. Such critical points are called \textit{degenerate critical points}, their existence usually destroys square root cancellation in $S(f,m)$.

We now provide some concrete examples for complete exponential sums. A first example, which we frequently encounter, is the twisted Kloosterman sum $S_{\chi}(A,B, m)$. If $A,B\in \mathcal{O}^{\times}$, $m=1$, and $a(\chi)\leq 1$ we have the strong bound
\begin{equation}
	\abs{S_{\chi}(A,B,1)} \leq 2\zeta_F(1)q^{-\frac{1}{2}}, \label{eq:finite_field_KS}
\end{equation}
which is essentially due to Weil. See \cite[Chapter~11, Exercise~1]{IK04}. If $m>1$ and $a(\chi)<\frac{m}{2}$, one can evaluate this sum explicitly and obtain
\begin{equation}
	S_{\chi}(1,u,m) = \begin{cases}
		2\zeta_F(1) q^{-\frac{m}{2}}\Re(\alpha(u)) &\text{ if }u\in \mathcal{O}^{\times 2},  \\
		0 &\text{ else,}
	\end{cases} \nonumber
\end{equation}
where $\alpha(u)$ is some root of unity depending on $F$, $\chi$, and $u$. For arbitrary characters $\chi$ square root cancellation might fail. However, one can still derive estimates for the sum. See for example \cite{KL13}. Over $\Q_p$ a general sharp bound for $S_{\chi}(A,B,m)$ is given in \cite[Section~5]{CZ00}. One can prove a slightly more general estimate.
 
\begin{lemma} \label{lm:est_TKS}
Let $q$ be odd. Let $\chi$ be a multiplicative character, $a\in \mathcal{O}^{\times}$ a unit and $l,m\in \N_0$. If $m>a(\chi)$, then 
\begin{equation}
	\abs{S_{\chi}(1,a\varpi^l,m)} \leq \begin{cases}
		2\zeta_F(1)q^{-\frac{m}{2}} &\text{ if } l=0 \text{ and } a\in \mathcal{O}^{\times 2},\\
		\zeta_F(1)q^{-1}  &\text{ if } m=1 \text{ and } l \geq 1,\\
		0 &\text{else.}
	\end{cases} \nonumber
\end{equation}
In the opposite situation, $m<a(\chi)$, we have
\begin{equation}
	S_{\chi}(1,a\varpi^l,m) = 0.
\end{equation}
Finally, if $m=a(\chi)$, then
\begin{equation}
	\abs{S_{\chi}(1,a\varpi^l,m)} \leq \begin{cases}
		2\zeta_F(1)q^{-\frac{m}{2}} &\text{ if } l>0 \text{ or } l=0 \text{ and }  a\not \in \frac{b_{\chi}^2}{4}+\p,\\
		2\zeta_F(1)q^{-\frac{m}{4}} &\text{ else.}
	\end{cases} \nonumber
\end{equation}
If we assume additionally that $F= \Q_p$, then we have the stronger bound
\begin{equation}
	\abs{S_{\chi}(1,a\varpi^l,m)} \leq 2q^{-\frac{m}{3}}. \nonumber
\end{equation}
\end{lemma}
\begin{proof}
First, if $l\geq m$, this reduces to the Gau\ss\  sum evaluation \eqref{eq:evaluation_of_GS}. If $m=0$, we have a pure (multiplicative) character sum and the claim follows by orthogonality. If $m=1$ and $a(\chi)\leq 1$, then this is a slight extension of Weil's bound for Kloosterman sums. In all the other cases we can apply the method of stationary phase and conclude by estimating the remaining sum trivially as in \cite[Section~9]{KL13}. 

The last bound can be found in \cite[Section~5]{CZ00} and turns out to be sharp.
\end{proof}

Another example is an exponential sum involving a cubic polynomial.  In analogy to the real place we define
\begin{equation}
	\text{Ai}_{\psi} (a;b) = q^{\frac{v(a)}{3}}\int_{\mathcal{O}} \psi(ax^3+bx)dx. \nonumber
\end{equation}	
Note that $\text{Ai}_{\psi}(a;b)= 0$ if $v(b)<\min(0,v(a))$. Thus, by Lemma~\ref{lm:est_TKS} we have the bound 
\begin{equation}
	\abs{\text{Ai}_{\psi}(a,b)}\leq 2. \nonumber
\end{equation}

An important two dimensional example that lies at the heart of the next section is the integral $K(\chi_1\otimes \chi_2,(\varpi^{-l_1}, \varpi^{-l_2}), v\varpi^{-l})$ attached to the quadratic space $E=F\times F$.

\begin{lemma} \label{lm:prototype_K_stat}
Suppose $\chi_1$ and $\chi_2$ are characters on $F^{\times}$ such that $a(\chi_1)\geq a(\chi_2)\geq 1$. Put $k=\max(a(\chi_1),l)=2r+\rho$ for some $r\in \N_0$ and $\rho\in \{0,1\}$. Then for $0<l_1,l_2 \leq l$ and $r>0$ we have
\begin{eqnarray}
	&&K(\chi_1\otimes \chi_2,(\varpi^{-l_1}, \varpi^{-l_2}), v\varpi^{-l}) \nonumber \\
	&&\qquad\qquad  = \zeta_F(1)^2 q^{-2r} \sum_{(x_1,x_2)\in S} \chi_1(x_1)\chi_2(x_2) \psi(x_1\varpi^{-l_1}+x_2\varpi^{-l_2}+vx_1x_2\varpi^{-l}) \nonumber \\
	&&\qquad\qquad \qquad\qquad \qquad \cdot G\left( \frac{\varpi^{-\rho}}{2}A_{x_1,x_2}, \varpi^{-r-\rho}B_{x_1,x_2}\right) \nonumber
\end{eqnarray}	
for 
\begin{eqnarray}
	A_{x_1,x_2} &=& \left( \begin{matrix} -b_1\varpi^{k-a(\chi_1)} & vx_1x_2\varpi^{k-l} \\ vx_1x_2\varpi^{k-l} & -b_2\varpi^{k-a(\chi_2)} \end{matrix} \right), B_{x_1,x_2} = \left(\begin{matrix} b_1\varpi^{k-a(\chi_1)} + x_1\varpi^{k-l_1}+vx_1x_2 \varpi^{k-l} \\  b_2\varpi^{k-a(\chi_2)} + x_2\varpi^{k-l_2}+vx_1x_2 \varpi^{k-l} \end{matrix}\right),  \nonumber\\
	S &=& \{ x_1,x_2 \in (\mathcal{O}/\p^r)^{\times} \colon B_{x_1,x_2 } \in (\p^r)^2  \} \nonumber
\end{eqnarray}
where $b_1$ and $b_2$ are the constants associated to the characters $\chi_1$ and $\chi_2$ using Lemma~\ref{lm:char_trick}. In particular we have
\begin{equation}
	\abs{K(\chi_1\otimes \chi_2,(\varpi^{-l_1}, \varpi^{-l_2}), v\varpi^{-l}) } \leq \zeta_F(1)^2 q^{-2r} \sharp S \sup_{x_1,x_2\in S} \abs{G\left( \frac{\varpi^{-\rho}}{2}A_{x_1,x_2}, \varpi^{-r-\rho}B_{x_1,x_2}\right)}. \nonumber
\end{equation}
\end{lemma}
\begin{proof}
First we use Lemma~\ref{lm:char_trick} to rewrite the integral as
\begin{eqnarray}
	&&K(\chi_1\otimes \chi_2,(\varpi^{-l_1}, \varpi^{-l_2}), v\varpi^{-l}) \nonumber \\
		&&\qquad\qquad  = \zeta_F(1)^2 q^{-2r} \sum_{x_1,x_2\in (\mathcal{O}/\p^r)^{\times}} \chi_1(x_1)\chi_2(x_2) \psi(x_1\varpi^{-l_1}+x_2\varpi^{-l_2}+vx_1x_2\varpi^{-l}) \nonumber \\
		&&\qquad\qquad \qquad\qquad \qquad \cdot G\left( \frac{\varpi^{-\rho}}{2}A_{x_1,x_2}, \varpi^{-r-\rho}B_{x_1,x_2}\right) \nonumber
\end{eqnarray}
Using the support properties of the Gau\ss\  sum contained in Lemma~\ref{lm:eval_quad_GS} it is clear that we can restrict the summation to $S$.
\end{proof}

Whenever the set $S$ is small, this usually means $\sharp S \leq 2$, then the estimate given above is good enough and almost sharp. Otherwise it is possible to find a suitable parametrization for $S$ which reduces everything to a one-dimensional integral to which we can apply classical results. Over the next few sections we will investigate the behaviour of $K$ in several situations and deduce the size of the local Whittaker functions $W_{\pi}$.

\section{The size of Whittaker new vectors} \label{se:case_by_case}

In this section we estimate the size of Whittaker new vectors. To do so we will built on Section~\ref{se:Wh_supp} in which we basically reduced the problem to estimate $K(\xi,A,B)$ in several situations. Due to the generality of this sum there are many cases which seem quite different in nature and the estimation turns out to be Sisyphus work. This upcoming case study relies heavily on repeated use of the method of stationary phase as described in the previous section. Throughout this section we assume that $F$ has odd residual characteristic.

\subsection{Dihedral supercuspidals}

There are two slightly different types of dihedral supercuspidal representations. We start with representations associated to unramified quadratic extensions of $E/F$.

\begin{lemma} \label{lm:unramified_sc_bound}
Let $\pi$ be a dihedral supercuspidal representation associated to an unramified quadratic extension $E/\Q_p$ and a character $\xi\colon E^{\times}\to S^1$. Then we have
\begin{equation}
	\abs{W_{\pi}(g)} \leq 2\max(q^{\frac{n}{12}},\sqrt{q}). \nonumber
\end{equation}
\end{lemma}
\begin{proof}
We start by recalling some facts concerning the extension $E/F$. Since it is unramified we have $e=1$, $f=2$, and $d=0$. Thus $n=2a(\xi)$, $a(\psi_E)=0$, and $\vol(\mathfrak{O},d_E) =1$. Furthermore, we find an element $\zeta\in \mathcal{O}^{\times} \setminus \mathcal{O}^{\times 2}$ such that $E=F(\sqrt{\zeta})$. Note that $\mathfrak{O} = \mathcal{O}\oplus \mathcal{O}\sqrt{\zeta}$ and $\mathfrak{O}^{\times} = ( \mathcal{O}^{\times}\oplus \mathcal{O}\sqrt{\zeta})\cup( \mathcal{O}\oplus \mathcal{O}^{\times}\sqrt{\zeta})$. We choose uniformizers such that $\Omega=\varpi$. For $x=a+b\sqrt{\zeta}$ we compute
\begin{equation}
	\Norm(1+x) = (1+a+b\sqrt{\zeta})(1+a-b\sqrt{\zeta}) = 1+\Tr(x)+\Norm(x).\label{eq:important_norm_formula}
\end{equation}

We put $k=\max(2l,n)$. Then Lemma~\ref{lm:supp_wh_supercusp} tells us that we only need to consider $t=-k$ if $l\neq \frac{n}{2}$ and $0>t\geq -k$ if $l=\frac{n}{2}$, since otherwise $W_{\pi}(g_{t,l,v})$ vanishes. In these cases we have 
\begin{equation}
	W_{\pi}(g_{t,l,v}) = \gamma q^{-\frac{t}{2}} \int_{\mathfrak{O}^{\times}} \xi^{-1}(x) \psi(\varpi^{\frac{t}{2}}\Tr(x)+\varpi^{-l}v\Norm(x))d_Ex. \label{eq:orig_whi_sc}
\end{equation}

We write $\frac{k}{2}=2r+\rho$ for some $r\in \N_0$ and $\rho \in \{0,1\}$. First, we note that if $r=0$, then we must have $\rho=1$ and thus $a(\xi) = l = -\frac{t}{2}=1$. By \cite[Corollary~2.35]{Sa15_2} we have
\begin{equation}
	W_{\pi}(g_{t,l,v}) \leq \sqrt{2q}. \nonumber
\end{equation} 

From now on we assume $r\geq 1$. In this case we can calculate
\begin{eqnarray}
	W_{\pi}(g_{t,l,v}) &=& \gamma q^{-\frac{t}{2}} \int_{\mathfrak{O}^{\times}} \xi^{-1}(x) \psi\left(\varpi^{\frac{t}{2}}\Tr(x)+\varpi^{-l}v\Norm(x)\right)d_Ex \nonumber \\
	&=& \gamma q^{-2r-\frac{t}{2}} \sum_{x\in (\mathfrak{O}/\mathfrak{P}^r)^{\times}} \xi^{-1}(x) \psi\left(\varpi^{\frac{t}{2}}\Tr(x)\right) \nonumber \\
	&& \qquad\cdot \int_{\mathfrak{O}} \xi^{-1}\left(1+\frac{y}{x}\Omega^r\right)\psi_E\left(\Omega^{r+\frac{t}{2}}y+\frac{v\Omega^{-l}}{2}\Norm(x+y\Omega^r)\right)d_E y\nonumber\\
	&=& \gamma q^{-2r-\frac{t}{2}} \sum_{x\in (\mathfrak{O}/\mathfrak{P}^r)^{\times}} \xi^{-1}(x) \psi\left(\varpi^{\frac{t}{2}}\Tr(x)+v\varpi^{-l}\Norm(x)\right) \nonumber \\
	&&\qquad\cdot\int_{\mathfrak{O}} \psi_E\bigg(\left(-\frac{b_{\xi}}{x} \Omega^{\frac{k}{2}-a(\xi)}+\Omega^{\frac{k}{2}+\frac{t}{2}}+\frac{v \Norm(x)}{x}\Omega^{\frac{k}{2}-l}\right)y\Omega^{-r-\rho} \nonumber \\
	&&\qquad\qquad\qquad\qquad\qquad +\left(\frac{v\Omega^{\frac{k}{2}-l}}{2}\Norm(y)+\frac{b_{\xi}\Omega^{\frac{k}{2}-a(\xi)}}{2x^2} y^2 \right) \varpi^{-\rho} \bigg) d_E y. \nonumber
\end{eqnarray}
Next we will rewrite the remaining integral as a Gau\ss\  sum. We recall that $\mathfrak{O}=\mathcal{O}\oplus\mathcal{O}\sqrt{\zeta}$ then we can view the integral as an two dimensional integral. The quadratic term is
\begin{equation}
	\Tr\left(\frac{v\varpi^{\frac{k}{2}-l}}{2}\Norm(y)+\frac{b_{\xi}\varpi^{\frac{k}{2}-a(\xi)}}{2x^2}y^2\right) = {}^t y\bigg(v\varpi^{\frac{k}{2}-l}A_1+\varpi^{\frac{k}{2}-a(\xi)}A_2\bigg)y \nonumber
\end{equation}
for
\begin{equation}
	A_1 = \left(\begin{matrix}1 & 0 \\ 0  & -\zeta \end{matrix}\right),  A_2 = \Norm(x)^{-2}\left(\begin{matrix} (b_1x_1^2+b_1 \zeta x_2^2-2b_2\zeta x_1x_2)& \zeta (b_2 x_1^2+b_2\zeta x_2^2-2b_1 x_1x_2)\\ \zeta (b_2 x_1^2+b_2\zeta x_2^2-2b_1 x_1x_2) & \zeta(b_1x_1^2+b_1 \zeta x_2^2-2b_2\zeta x_1x_2) \end{matrix}\right),  \nonumber
\end{equation}
and $y\in \mathcal{O}^2$. This can be checked by a brute force calculation. If we write $\Im(a_1+a_2\sqrt{\zeta}) = a_2$ and $\Re(a_1+a_2\sqrt{\zeta}) = a_1$, we obtain the more compact form
\begin{equation}
	A_2 = \left( \begin{matrix} \Re(\frac{b}{x^2}) & \zeta \Im(\frac{b}{x^2}) \\ \zeta\Im(\frac{b}{x^2}) & \zeta\Re(\frac{b}{x^2}) \end{matrix} \right). \nonumber
\end{equation} 
In particular $\det(A_2) = \zeta \Norm(\frac{b}{x^2})$. Since $\frac{b}{x^2} \in \mathfrak{O}^{\times}$, we have $\Re(\frac{b}{x^2})\in \mathcal{O}^{\times}$ or $\Im(\frac{b}{x^2})\in \mathcal{O}^{\times}$. Using this we can argue that at least one entry of $vA_1+A_2$ is a unit. Indeed, if $\Im(\frac{b}{x^2})$ is a unit, we are done. Otherwise, assume $v+\Re(\frac{b}{x^2}) \in \p$ then $\zeta(-v+\Re(\frac{b}{x^2}))\in \mathcal{O}^{\times}$ which establishes what was claimed.  

Similarly we can write the linear term as
\begin{equation}
	\Tr\left(\left(-\frac{b_{\xi}}{x}\varpi^{\frac{k}{2}-a(\xi)} + \varpi^{\frac{k}{2}+\frac{t}{2}}+\frac{v\Norm(x)}{x}\varpi^{\frac{k}{2}-l}\right)y\right) = 2 { }^tB y\nonumber
\end{equation}
for
\begin{equation}
	B= \left(\begin{matrix} (-b_1x_1+\zeta b_2x_2)\Norm(x)^{-1}+v\varpi^{\frac{k}{2}-l}x_1+\varpi^{\frac{k}{2}+\frac{t}{2}} \\ (-b_2x_1+b_1x_2)\zeta\Norm(x)^{-1}-v\zeta\varpi^{\frac{k}{2}-l} x_2 \end{matrix}\right). \nonumber
\end{equation}
In this notation we obtain
\begin{eqnarray}
	W_{\pi}(g_{t,l,v}) &=& \gamma q^{-2r-\frac{t}{2}} \sum_{x\in S^{\times}}\xi^{-1}(x)\psi(\varpi^{\frac{t}{2}}\Tr(x)+v\varpi^{-l}\Norm(x)) \nonumber \\
	&&\qquad \cdot G\left( \varpi^{-\rho}\Big(v\varpi^{\frac{k}{2}-l}A_1+\varpi^{\frac{k}{2}-a(\xi)}A_2\Big),2\varpi^{-r-\rho}B\right). \label{eq:unram_sc_reduced_S}
\end{eqnarray}
Here we restricted the sum to 
\begin{equation}
	x\in S^{\times} =  \left\{x \in (\mathfrak{O}/\p_E^r)^{\times}\colon - \frac{b_{\xi}}{x} \Omega^{\frac{k}{2}-a(\xi)}+\Omega^{\frac{k}{2}+\frac{t}{2}}+\frac{v \Norm(x)}{x}\Omega^{\frac{k}{2}-l} \in\mathfrak{P}^r \right\}, \nonumber
\end{equation} 
since otherwise the Gau\ss\  sum vanishes due to Lemma~\ref{lm:eval_quad_GS}. Writing $x= x_1+x_2 \sqrt{\zeta}$ and $b_{\xi} = b_1+b_2\sqrt{\zeta}$ we restate the congruences defining $S^{\times}$ as
\begin{eqnarray}
	-b_1\varpi^{\frac{k}{2}-a(\chi)}+x_1\varpi^{\frac{k}{2}+\frac{t}{2}}+v(x_1^2-\zeta x_2^2)\varpi^{\frac{k}{2}-l} \in \p^r,\nonumber\\
	-b_2\varpi^{\frac{k}{2}-a(\chi)} + x_2\varpi^{\frac{k}{2}+\frac{t}{2}} \in \p^r, \nonumber
\end{eqnarray}
for $x_1$ or $x_2$ is in $\mathcal{O}^{\times}$.

We will compute the set $S^{\times}$ in several cases and deduce the size of $W_{\pi}(g_{t,l,v})$ using \eqref{eq:unram_sc_reduced_S}.

\textbf{Case I: $0<l<\frac{n}{2}$.} In this situation we have $t=-k=-n$ and the structure of $S^{\times}$ is very simple. Indeed, we have
\begin{equation}
	x_2 \in b_2 +\p^r. \nonumber
\end{equation}
This leads to the quadratic congruence
\begin{equation}
	v\varpi^{\frac{k}{2}-l}x_1^2+x_1-(b_1+\zeta b_2^2\varpi^{\frac{k}{2}-l}) \in \p^r.\nonumber
\end{equation}
In the notation of Lemma~\ref{lm:quad_cong} this puts us in the situation where $v(b)=0$ and $v(a)=\frac{k}{2}-l>0$. Thus there is one solution. Even more, if $b_2 \not\in \mathcal{O}^{\times}$, then $x_2$ is not. But this forces $b_1 \in \mathcal{O}^{\times}$. Therefore, the unique solution satisfies $z_0=x_1+x_2\sqrt{\zeta}\in \mathfrak{D}^{\times}$. We  obtain
\begin{equation}
	W_{\pi}(g_{t,l,v}) = \gamma q^{\rho}\xi^{-1}(z_0)\psi(\varpi^{\frac{t}{2}}\Tr(z_0)+v\varpi^{-l}\Norm(z_0)) G(\varpi^{-\rho}A_2, 2\varpi^{-r-\rho}B). \nonumber
\end{equation}
Furthermore, since $\det(A_2) = \zeta\Norm(\frac{b}{z_0^2}) \in \mathcal{O}^{\times}$ and $A_2$ has entries in $\mathcal{O}$  we use Lemma~\ref{lm:eval_quad_GS} to see that
\begin{equation}
	\abs{W_{\pi}(g_{-n,l,v})} \leq 1. \nonumber
\end{equation}

\textbf{Case II: $l=\frac{n}{2}$.} In this case $k = n=2l$ and $-n\leq t<0$. We call $x_2$ admissible if it satisfies 
\begin{equation}
	x_2\varpi^{\frac{k}{2}+\frac{t}{2}} \in -b_2+\p^r. \nonumber
\end{equation}
In order to determine the structure of $S^{\times}$ we have to solve the quadratic congruence
\begin{equation}
	vx_1^2+x_1\varpi^{\frac{k}{2}+\frac{t}{2}}-(b_1+v\zeta x_2^2) \in \p^r, \nonumber
\end{equation}
for each admissible $x_2$. 

\textbf{Case II.1: $b_2 \in \mathcal{O}^{\times}$.} One sees that admissible $x_2$ exists only if $t=-k$. In this situation we  have exactly one admissible $x_2$. Namely
\begin{equation}
	x_2 = b_2 \in (\mathcal{O}/\p^r)^{\times}. \nonumber
\end{equation}
The quadratic equation for $x_1$ then has discriminant $\Delta=\Delta(v)=1+4v^2b_2^2\zeta +4vb_1$. If $\Delta\in \mathcal{O}^{\times}$, we have up to two possibilities for $x_1$, so that $\sharp S^{\times} \leq 2$. We now turn towards the matrix $vA_1+A_2$. We can compute 
\begin{equation}
	\det(vA_1+A_2) = \zeta \Norm(\frac{b}{x^2})-\zeta v^2 \in \zeta\frac{\Norm(b_{\xi})+(x_1-b_1)^2}{\Norm(x)^2} + \p. \nonumber
\end{equation}
Note that for certain compositions of $v$, $b_1$, and $b_2$ the case $\det(vA_1+A_2) \in \p$ can not be excluded. Therefore, we use the estimate
\begin{equation}
	\abs{G\left( \varpi^{-1}\Big(vA_1+A_2\Big),2\varpi^{-r-1}B\right)} \leq q^{-\frac{1}{2}}, \nonumber
\end{equation}
Thus in this case we obtain the bound
\begin{equation}
	\abs{W_{\pi}(g_{t,\frac{n}{2},v})} \leq  \begin{cases} 2q^{\frac{\rho}{2}} &\text{ if } t=-n,\\ 0 &\text{ else.} \end{cases} \nonumber
\end{equation}

Unfortunately, viewing $\Delta$ as an quadratic equation in $v$, it turns out that if $\Norm(b) \in \mathcal{O}^{\times 2}$ there are possibilities for $v$ such that $\Delta \in \p$. If this happens we use Lemma~\ref{lm:quad_cong} to parametrize the set $S^{\times}$ and define
\begin{equation}
	A_{\pm} =  -\frac{1}{2v}+b_2\sqrt{\zeta}\pm \frac{Y}{2v}\varpi^{\delta} \in \mathfrak{O}^{\times}. \nonumber
\end{equation}
Inserting the so obtained parametrization in \eqref{eq:unram_sc_reduced_S} yields
\begin{eqnarray}
	W_{\pi}(g_{t,l,v}) &=& q^{\rho}\sum_{\pm} \gamma_{\pm} \psi(\varpi^{\frac{t}{2}}\Tr(A_{\pm})+v\Norm(A_{\pm})\varpi^{-l}) \nonumber\\
	&&\qquad \cdot \sum_{x \in \mathcal{O}/\p^{\delta}}\xi^{-1}(A_{\pm}+x\Omega^{r-\delta}) \psi((1\pm  Y\varpi^{\delta})x\varpi^{-\rho-r-\delta}+vx^2\varpi^{-\rho-2\delta}) \nonumber \\
	&&\qquad\qquad \cdot G\left( \varpi^{-\rho}\Big(vA_1+A_2\Big),2\varpi^{-r-\rho}B_x\right). \nonumber
\end{eqnarray}
using the convention that $\gamma_{\pm} = \frac{\gamma}{2}$  if $v(\Delta)>0$ and $\gamma_{\pm} = \gamma$ otherwise.

We use
\begin{equation}
	x_2 = b_2 \text{ and } x_1 \in -\frac{1}{2v}+\p \nonumber
\end{equation}
to compute
\begin{eqnarray}
	A_{\p} &=& vA_1+\frac{1}{1- 4v^2b_2^2\zeta}\left( \begin{matrix} -4v\Norm(b) & \zeta b_2 \Delta \\ \zeta b_2 \Delta & -\zeta 4v\Norm(b) \end{matrix} \right) \nonumber \\
	&=& \frac{v}{1-\zeta 4 v^2b_2^2} \left( \begin{matrix} 1-\zeta 4 v^2b_2^2-4\Norm(b) & 0 \\ 0 & \zeta(-1+\zeta 4 v^2b_2^2-4\Norm(b))  \end{matrix} \right). \nonumber
\end{eqnarray}
Thus, $A_{\p}$ is independent of $x$ and diagonal. Furthermore, at least one of the diagonal entries is in $\mathcal{O}^{\times}$. The Gau\ss\  sum can be evaluated using Lemma~\ref{lm:eval_quad_GS}. For brevity we only deal with the case $\rho=0$. The details for the other case are very similar and left to the reader. Since $E/F$ is an unramified we have $\kappa_E =1$. Therefore, we can make use of the $p$-adic logarithm over $E$ without convergence issues and apply Lemma~\ref{lm:char_trick} to write
\begin{eqnarray}
	&&W_{\pi}(g_{t,l,v}) \nonumber \\
	&&\quad = q^{\delta}\sum_{\pm} \gamma_{\pm} \xi^{-1}(A_{\pm})\psi(\varpi^{\frac{t}{2}}\Tr(A_{\pm})+v\Norm(A_{\pm})\varpi^{-l}) \nonumber \\
	&&\qquad \cdot \int_{\mathcal{O}}\psi\left((1\pm  Y\varpi^{\delta})x\varpi^{-r-\delta}+vx^2\varpi^{-2\delta}-\Tr\left(\frac{b_{\xi}}{\Omega^{2r}}\log_E\left(1+\frac{x}{A_{\pm}}\Omega^{r-\delta}\right)\right)\right)  dx. \nonumber 
\end{eqnarray}
The remaining task is to show that the integral actually vanishes. To do so we open the Taylor expansion of the logarithm and obtain
\begin{eqnarray}
	I &=& \int_{\mathcal{O}}\psi\left((1\pm  Y\varpi^{\delta})x\varpi^{-r-\delta}+vx^2\varpi^{-2\delta}-\Tr\left(\frac{b_{\xi}}{\varpi^{2r}}\log_E\left(1+\frac{x}{A_{\pm}}\Omega^{r-\delta}\right)\right)\right)  dx \nonumber \\
	&=& \int_{\mathcal{O}}\psi(a_1x\varpi^{-r-\delta}+a_2x^2\varpi^{-2\delta}+a_3x^3 \varpi^{r-3\delta})dx \nonumber
\end{eqnarray}
for
\begin{eqnarray}
	a_1 &=& 1\pm Y\varpi^{\delta}-\Tr\left(\frac{b_{\xi}}{A_{\pm}}\right), \nonumber \\
	a_2 &=& v+\frac{1}{2}\Tr\left(\frac{b_{\xi}}{A_{\pm}^2}\right), \nonumber \\
	a_3 &=& -\frac{1}{3}\Tr\left(\frac{b_{\xi}}{A_{\pm}^3}\right). \nonumber
\end{eqnarray}
Our remaining task is to find further cancellation in $I$. We compute 
\begin{equation}
	a_1 \in 1+\frac{2b_2^2\zeta+\frac{b_1}{v}}{\frac{1}{4v^2}-\zeta b_2^2} \mod \p. \nonumber
\end{equation}
From this we can deduce that $a_1\in \p$ is equivalent to $\Delta\in\p$, which is assumed at the moment. Using $\Delta \in \p$ we can also check that
\begin{equation}
	a_2 \in v-\frac{\Norm(b)}{v\Norm(A_{\pm})^2}+\p. \nonumber
\end{equation}
However, this implies that in order for $a_2 \in \p$ we must have\footnote{One can actually check that $\Delta\in\p$ implies $a_1,a_2\in\p$. Even more, $\Delta=0$ implies $a_1=a_2=0$.}
\begin{equation}
	\Norm(b_{\xi})\in v^2\Norm(A_{\pm})^2+\p. \nonumber
\end{equation}

We further compute
\begin{eqnarray}
	a_3 \in -\frac{\Norm(b)}{3v^2 \Norm(A_{\pm})^3}+\p.\nonumber
\end{eqnarray}
Which implies that $a_3 \in 3^{-1}\mathcal{O}^{\times}$. Thus in the worst case scenario we obtain $I\leq q^{\frac{r}{3}-\delta}$. We conclude
\begin{equation}
	\abs{W_{\pi}(g_{-n,\frac{n}{2},v})} \leq 2q^{\frac{n}{12}}. \nonumber
\end{equation}

\textbf{Case II.2: $b_2\not\in \mathcal{O}^{\times}$, $ t = 2v(b_2)-n$, and $t<-\frac{n}{2}-\rho$.}  Thus, for each $\alpha \in \mathcal{O}/\p^{v(b_2)}$
\begin{equation}
	x_2 = (b_2)_0+\alpha \varpi^{r-v(b_2)} \nonumber
\end{equation}
is admissible. The equation for $x_1$ reads
\begin{equation}
	vx_1^2+\varpi^{v(b_2)}x_1-(b_1+v\zeta x_2^2)\in \p^r \nonumber
\end{equation}
and has discriminant
\begin{equation}
	\Delta = \varpi^{2v(b_2)}+4v(b_1+v x_2^2 \zeta). \nonumber
\end{equation}
For a fixed $x_2$ we write $S_{x_2}$ for the set of possible $x_1$. Since there are $q^{v(b_2)}$ admissible $x_2$ we have the estimate
\begin{equation}
	\abs{W_{\pi}(g_{2v(b_2)-n,\frac{n}{2},v})} \leq  \sup_{x_2 \text{ admissible}}\abs{\sum_{x_1\in S_{x_2}}\xi^{-1}(x_1+x_2\sqrt{\zeta})\psi(2\varpi^{\frac{t}{2}}x_1+v\varpi^{-l}x_1^2)}. \nonumber
\end{equation}
If $\Delta\in \mathcal{O}^{\times}$, then it is so for any $x_2$ and $\sharp S_{x_2} \leq2$. Thus, we are done after estimating trivial. However, if $\Delta \in \p$, the $S_{x_2}$-sum is potentially large. We can deal with this using another stationary phase estimate. The computations turn out to be very similar to ones in Case~II.1. The outcome is
\begin{equation}
	\abs{W_{\pi}(g_{2v(b_2)-n,\frac{n}{2},v})} \leq 2q^{\frac{n}{12}} \text{ for $v(b_2)>0$}. \nonumber
\end{equation}

\textbf{Case II.3: $b_2\not\in \mathcal{O}^{\times}$, $t\neq v(b_2)-n$ or $t\geq -\frac{n}{2}-\rho$.} In this case we use \eqref{eq:unram_sc_reduced_S} to estimate
\begin{equation}
	\abs{W_{\pi}(g_{t,l,v})}\leq q^{-2r-\frac{t}{2}-\frac{\rho}{2}}\sharp S. \nonumber
\end{equation}
Due to the assumptions of this case it is easy to compute $\sharp S$. We obtain
\begin{equation}
	\abs{W_{\pi}(g_{t,l,v})}\leq 2q^{\frac{\rho}{2}}. \nonumber
\end{equation}

\textbf{Case III: $\frac{n}{2}<l$.} Here we have $k=2l=-t$. Thus, towards the structure of $S^{\times}$  we find
\begin{equation}
	x_2 \in b_2\varpi^{l-a(\xi)}+\p^r \nonumber
\end{equation}
is no unit. We therefore have to find $x_1\in (\mathcal{O}/\p^r)^{\times} $ satisfying
\begin{equation}
	vx_1^2+x_1+(b_1\varpi^{l-a(\xi)}-\zeta b_2^2\varpi^{2l-n}) \in \p^r. \nonumber
\end{equation}
In particular we use Lemma~\ref{lm:quad_cong} with $v(a)=v(b)=v(\Delta)=0$ and find that $\sharp S^{\times} \leq 2$. Estimating the $S^{\times}$-sum in \eqref{eq:unram_sc_reduced_S} trivially yields
\begin{eqnarray}
	\abs{W_{\pi}(g_{-2l,l,v})} &\leq& 2q^{\rho}\abs{ G(\varpi^{-\rho}v,2\varpi^{-r-\rho}((b_1x_1-b_2x_2)\Norm(x)+v(x_1)+1))}\nonumber \\
	&&\qquad\cdot \abs{G(-\varpi^{-\rho}v\zeta,2\varpi^{-r-\rho}((b_2x_1-b_1x_2)\zeta \Norm(x)+v\zeta x_2))} \nonumber \\
	&\leq& 2. \nonumber
\end{eqnarray}

This was the last case to consider and the proof is complete.
\end{proof}

We now turn to representations associated to ramified extensions of $F$.

\begin{lemma} \label{lm:ramified_sc_bound}
Let $\pi$ be a dihedral supercuspidal representation associated to a ramified quadratic extension $E/\Q_p$ and a multiplicative character $\xi$. Then we have
\begin{equation}
	\abs{W_{\pi}(g)}\leq 2q^{\frac{n}{12}+\frac{1}{2}}. \nonumber
\end{equation}
\end{lemma}

\begin{proof}
Since $E/F$ is a ramified extension we have $f=1$, $e=2$ and $d=1$. In particular $n(\psi_E)=-1$ and the additive measure on $E$ is normalized so that
\begin{equation}
	\vol(\mathfrak{D},\mu_E) = q^{-\frac{1}{2}}. \nonumber
\end{equation} 
By changing $\varpi$ if necessary we can assume that $E=F(\sqrt{\varpi})$ and $\Omega = \sqrt{\varpi}$. The identity \eqref{eq:important_norm_formula} still holds.

The log-conductor of $\pi$ is given by $n=a(\pi) = a(\xi)+1$. We observe that
\begin{equation}
	\abs{W_{\pi}(g_{-n,0,v})} = \abs{\epsilon(\frac{1}{2},\tilde{\pi})} = 1. \nonumber
\end{equation}
Thus, we can assume $l>0$ and define $k=\max(a(\xi),2l)=2r+\rho$. Using Lemma~\ref{lm:supp_wh_supercusp} and Lemma~\ref{lm:char_trick} we compute
\begin{eqnarray}
	W_{\pi}(g_{t,l,v}) &=& \gamma q^{-\frac{t}{2}}K(\xi^{-1}, \Omega^t,v\varpi^{-l}) \nonumber \\
	&=& \gamma q^{-\frac{t}{2}-r} \sum_{x\in (\mathfrak{O}/\mathfrak{P}^r)^{\times}} \xi^{-1}(x)\psi(\Tr(x\Omega^{t})+v\varpi^{-l}\Norm(x)) \nonumber \\
	&&\qquad \cdot  \int_{\mathfrak{O}} \psi_E\bigg(v\Norm(y) \Omega^{k-l-\rho}+\frac{b_{\xi}}{x^2}y^2\Omega^{k-a(\chi)-1-\rho} \nonumber \\
	&&\qquad\qquad\qquad +\left(-\frac{b_{\xi}}{x}\Omega^{k-a(\xi)}+\Omega^{k+1+t}+\frac{v\Norm(x)}{x}\Omega^{k-2l+1}\right) y\Omega^{-r-\rho-1}\bigg)d_Ey. \label{eq:prototype_ram_su_w}
\end{eqnarray}
We need to estimate this for $t\geq -k$ if $l \geq \frac{n}{2}$ and $t=-k-1$ otherwise. In all these cases the remaining integral reduces to a quadratic Gau\ss\  sum over $E$. Thus we can restrict the $x$-sum to
\begin{equation}
	S= \{ x\in (\mathfrak{O}/\mathfrak{P}^r)^{\times} \colon  -b_{\xi} \Omega^{k-n}+x\Omega^{k+t}+v\Norm(x)\Omega^{k-2l}\in \mathfrak{P}^{r-1} \}.\nonumber
\end{equation}

\textbf{Case I: $0<l<\frac{n}{2}$.} Due to the support properties of $W_{\pi}$ we can assume that $t=-n=-a(\xi)-1$. Obviously $k=a(\xi)$. The set $S$ is determined by the congruence
\begin{equation}
	-b_{\xi} + x +v\Norm(x)\Omega^{n-2l}\in \mathfrak{P}^r. \nonumber
\end{equation}
If $n-2l\geq r$, there is exactly one solution. Namely $x = b_{\xi}$ modulo $\mathfrak{P}^r$. Otherwise we write $x = x_1+x_2\Omega$. This, leads to the two congruences
\begin{eqnarray}
	-b_1+x_1 &\in& \p^{\lceil \frac{r}{2} \rceil},\nonumber\\
	-b_2+x_2+v(x_1^2-\varpi x_2^2)\varpi^{r-l} &\in& \p^{\lfloor\frac{r}{2}\rfloor}. \nonumber
\end{eqnarray}
Using Lemma~\ref{lm:quad_cong} we observe that there is a unique solution $(x_1,x_2) \in \mathcal{O}/\p^{\lceil \frac{r}{2} \rceil} \times \mathcal{O}/\p^{\lfloor\frac{r}{2}\rfloor}$. Furthermore, one quickly checks that $x_1$ is a unit. We recall $n(\psi_E)=-1$,  $\vol(\mathfrak{D},\mu_E) = q^{-\frac{1}{2}}$ and $k-l-\rho>0$. With this in mind we evaluate the quadratic Gau\ss\  sum and obtain
\begin{equation}
	\abs{W_{\pi}(g_{-n,l,v})}= 1. \nonumber
\end{equation}

\textbf{Case II: $l=\frac{n}{2}$.} In this case $2l=a(\xi)+1$, $k=2l$ and $\rho=0$. We write $x=x_1+x_2\Omega$ and $b_{\xi} = b_1+b_2\Omega$, to transform the congruence condition in $S$ into a system of congruences over $F$. For simplicity we consider further subcases.

\textbf{Case II.1: $2l+t>0$ even.} In this case we obtain
\begin{eqnarray}
	-b_1+x_1 \varpi^{\frac{2l+t}{2}}+v x_1^2-v\varpi x_2^2 &\in& \p^{\lceil \frac{r-1}{2}\rceil}, \nonumber\\
	-b_2+x_2 \varpi^{\frac{2l+t}{2}} &\in& \p^{\lfloor \frac{r-1}{2}\rfloor}. \nonumber
\end{eqnarray}
The first equation is quadratic in $x_1$ with discriminant $\Delta(v,x_2) = \varpi^{2l+t}+4v(b_1+\varpi x_2^2)$. Since $b_{\xi}\in \mathfrak{O}^{\times}$ we have $b_1 \in \mathcal{O}^{\times}$ and therefore $\Delta\in \mathcal{O}^{\times}$. There are up to two solutions modulo $\p^{\lceil \frac{r-1}{2}\rceil}$ for $x_1$ for each given $x_2$. Thus the two congruences define a set $S'$ modulo $\mathfrak{P}^{r-1}$ with 
\begin{equation}
	\sharp S' \leq \begin{cases}
		2q^{\frac{2l+t}{2}} &\text{ if } \frac{2l+t}{2}< \lfloor\frac{r-1}{2}\rfloor,\\
		2q^{\lfloor\frac{r-1}{2}\rfloor} &\text{ if } \frac{2l+t}{2}\geq \lfloor\frac{r-1}{2}\rfloor, \\
		0 &\text{ else},
	\end{cases} \nonumber
\end{equation}
and 
\begin{equation}
	S = \{ x+\alpha \Omega^{r-1}\colon x\in S', \alpha\in \mathfrak{O}/\mathfrak{P}\}. \nonumber
\end{equation}
We obtain
\begin{equation} 
	\abs{W_{\pi}(g_{-2l,l,v})} \leq q^{-\frac{1}{2}-\frac{2l+t}{2}}\sharp S' \sup_{x\in S'} \abs{\sum_{\alpha \in \mathfrak{O}/\mathfrak{P}} \psi_E(\Omega^{-2}(z_x \alpha +\frac{v}{2}\Norm(\alpha)))}. \nonumber
\end{equation}
The remaining sum is a sum over the finite field $\mathcal{O}/\p$ and can be estimated using $\eqref{eq:Weils_bound}$. We obtain
\begin{equation}
	\abs{W_{\pi}(g_{-2l,l,v})} \leq 2. \nonumber
\end{equation}

\textbf{Case II.2: $2l+t>0$ odd. } We write $2l+t = 2\beta-1$ for some $\beta \in \N$. Then we have to solve the congruences
\begin{eqnarray}
	-b_2+x_1\varpi^{\beta-1} &\in& \p^{\lfloor \frac{r-1}{2}\rfloor} , \label{eq:ram_cong_case_122a} \\
	-v\varpi x_2^2+\varpi^{\beta}x_2-b_1+vx_1^2 &\in& \p^{\lceil\frac{r-1}{2}\rceil}. \label{eq:ram_cong_case_122b}
\end{eqnarray}
Looking at the first congruence reveals that, if $\beta-1 \neq v(b_2) < \lfloor \frac{r-1}{2} \rfloor$, there are no solutions for $x_1\in \mathcal{O}^{\times}$. Thus, $S=\emptyset$ and $W_{\pi}(g_{t, \frac{n}{2},v})=0$. 

The other extreme is $\beta-1, v(b_2)\geq \lfloor \frac{r-1}{2}\rfloor$. Here we can go back to the original equation determining $x$ and find that we need to count solutions to
\begin{equation}
	\Norm(x) \in \frac{b_{\xi}}{v}+\mathfrak{P}^{r-1}. \nonumber
\end{equation}
Therefore,
\begin{equation}
	\sharp S \leq  q \sharp\{ y \colon \Norm(y) \in 1+\p^{\lceil \frac{r-1}{2} \rceil}\} \leq  q^{\lceil\frac{r-1}{2}\rceil}.\nonumber
\end{equation}
We conclude that
\begin{equation}
	\abs{W_{\pi}(g_{t,\frac{n}{2},v})} \leq q^{-\beta}\sharp S \leq 1. \nonumber 
\end{equation}

The generic situation appears whenever $\beta-1 = v(b_2) <\lfloor \frac{r-1}{2}\rfloor$. In this case there are $q^{\beta-1}$ solutions for $x_1$. Let $S_{x_1}$ be the set of possible $x_2$ for each $x_1$. Further, we note that $-\frac{t}{2}-r-\frac{1}{2} = -\beta$ and bound
\begin{equation}
	\abs{W_{\pi}(g_{t,l,v})} \leq q^{-1} \sup_{x_1} \abs{\mathcal{S}(x_1) }, \nonumber
\end{equation}
for 
\begin{equation}
	\mathcal{S}(x_1) = \sum_{\substack{\beta\in \mathfrak{O}/\mathfrak{P},\\ x_2 \in S_{x_2} }}\xi^{-1}(x_1+\Omega x_2 + \beta \Omega^{r-1})\psi_E(x_2\Omega+\beta \Omega^{r-1}+\frac{v\varpi^{-l}}{2}\Norm(x_1+x_2\Omega+\beta \Omega^{r-1})). \nonumber
\end{equation}
After considering $r-1$ even or odd separately one can estimate $\mathcal{S}(x_1)$ and obtain
\begin{equation}
	\abs{W_{\pi}(g_{t,l,v})} \leq 2q^{\frac{n}{12}}. \nonumber
\end{equation}
We leave out the details at this point since we will perform a very similar calculation in the next case.

\textbf{Case II.3: $t=-2l$.} In this situation we have $x_2 = b_2 +\p^{\lfloor \frac{r-1}{2} \rfloor}$ so that everything comes down to solving the quadratic congruence
\begin{equation}
	vx_1^2+x_1-\varpi vb_2^2-b_1 \in \p^{\lceil \frac{r-1}{2}\rceil} \nonumber
\end{equation}
with discriminant
\begin{equation}
	\Delta = 1+4vb_1+4v^2b_2^2\varpi. \nonumber
\end{equation}
Depending on the $p$-adic size of $\Delta$ we have to examine different cases.

First, assume $v(\Delta)\geq 1$. Then $x_1$ is of the form 
\begin{equation}
	x_{\pm} =\begin{cases}
		-\frac{1}{2v} \pm \frac{Y}{2v} \varpi^{\delta}+\alpha \varpi^{\lceil \frac{r-1}{2}\rceil -\delta} &\text{ if }\Delta = Y^2\varpi^{2\delta} \text{ for some } \delta<\frac{1}{2}\lceil\frac{r-1}{2}\rceil \text{ for } \alpha \in \mathcal{O}/\p^{ \delta}, \\
	 	-\frac{1}{2v}+\alpha \varpi^{\lceil \frac{r-1}{2}\rceil -\delta} &\text{ if } v(\Delta) \geq \lceil \frac{r-1}{2}\rceil \text{ with } \delta=\lfloor\frac{1}{2}\lceil \frac{r-1}{2} \rceil\rfloor \text{ and } \alpha \in \mathcal{O}/\p^{ \delta}.
	\end{cases} \nonumber
\end{equation}
We define $B_{\pm}$ to be the $\alpha$ independent part of $x_{\pm}$. This determines $x\in S$ up to $\mathfrak{P}^{r-1}$. We obtain
\begin{eqnarray}
	S &=& \{ B_{\pm}+\alpha \varpi^{\lceil \frac{r-1}{2}\rceil -\delta}+b_2\Omega + \beta\Omega^{r-1} \colon \alpha \in \mathcal{O}/\p^{\delta}, \beta \in \mathfrak{O}/\mathfrak{P} \} \nonumber \\
	&=& \begin{cases} 
		\{B_{\pm}+\alpha \varpi^{\lceil \frac{r-1}{2}\rceil -\delta}+b_2\Omega \colon \alpha \in \mathcal{O}/\p^{\delta +1} \} &\text{ if $r-1$ is even,}\\
		\{ B_{\pm}+\alpha \varpi^{\lceil \frac{r-1}{2}\rceil -\delta}+(b_2+\beta \varpi^{\frac{r}{2}-1})\Omega  \colon \alpha \in \mathcal{O}/\p^{ \delta }, \beta \in \mathcal{O}/\p \} &\text{ if $r-1$ is odd.}
	\end{cases} \nonumber
\end{eqnarray} 

Next, we reinsert this parametrization in \eqref{eq:prototype_ram_su_w}. First, we deal with $r-1$ odd. Each element of $S$ is of the shape $A_{\pm} + \beta \Omega^{r-1} + \alpha \Omega^{r-2\delta}$, for $A_{\pm}= B_{\pm}+b_2\Omega$. We find
\begin{eqnarray}
	W_{\pi}(g_{-n, \frac{n}{2},v}) &=& \sum_{\pm}\gamma_{\pm} q^{-\frac{1}{2}}\xi^{-1}(A_{\pm}) \psi(\Tr(A_{\pm})+v\varpi^{-l}\Norm(A_{\pm})) \nonumber \\
	&&  \quad \sum_{\substack{\alpha \in \mathcal{O}/\p^{\delta },\\ \beta\in \mathcal{O}/\p}}\xi^{-1}(1+\frac{\alpha}{A_{\pm}}\Omega^{r-2\delta}+\frac{\beta}{A_{\pm}}\Omega^{r-1}) \nonumber \\
	&&\qquad \qquad \cdot \psi(2(1+v\Re(\overline{A_{\pm}}))\alpha\varpi^{-\frac{r}{2}-\delta}-2vb_2\beta \varpi^{-\frac{r}{2}}+v\alpha^2 \varpi^{-2\delta}+v\beta^2 \varpi^{-1}). \nonumber
\end{eqnarray}	
We use Lemma~\ref{lm:char_trick} to transform $\xi$ into an additive oscillation. The Taylor expansion of $\log_E$ is given by
\begin{eqnarray}
	&&-\Tr\left( \frac{b_{\xi}}{\Omega^{a(\xi)+1}}\log_E(1+(\frac{\alpha+\beta\Omega^{2\delta -1 }}{A_{\pm}})\Omega^{r-2\delta})\right)   \nonumber \\
	&&\qquad\qquad\qquad\qquad = \sum_{j\geq 1} \frac{(-1)^j}{j}\Tr\left[ \frac{b_{\xi}}{A_{\pm}^j}(\alpha+\beta\varpi^{\delta-1}\Omega)^j \right] \varpi^{(j-2)\frac{r}{2}-j\delta}. \nonumber
\end{eqnarray}
Observe $r-4\delta \geq 0$, so that we can truncate after the third term. We write $\frac{b_{\xi}}{A_{\pm}^j} = a_j'+a_j''\Omega$ and have a closer look at the coefficients.
\begin{eqnarray}
	\Tr\left[ \frac{b_{\xi}}{A_{\pm}}(\alpha+\beta\varpi^{\delta-1}\Omega)\right] &=& 2(a_1'\alpha+a_1''\beta \varpi^{\delta}), \nonumber \\
	\Tr\left[ \frac{b_{\xi}}{A_{\pm}^2}(\alpha+\beta\varpi^{\delta-1}\Omega)^2 \right] &=& 2(a_2'\alpha^2+2a_2''\alpha \beta \varpi^{\delta}+a_2'\beta^2\varpi^{2\delta-1})\nonumber \\
	\Tr\left[ \frac{b_{\xi}}{A_{\pm}^3}(\alpha+\beta\varpi^{\delta-1}\Omega)^3\right] &=& 2(a_3'\alpha^3+3a_3'\alpha \beta^2\varpi^{2\delta-1}+3a_3''\alpha^2\beta\varpi^{\delta}+a_3''\beta^3\varpi^{3\delta-1})\nonumber
\end{eqnarray}
This shows
\begin{eqnarray}
	&&W_{\pi}(g_{-n, \frac{n}{2},v})  \nonumber \\
	&&\quad =  \sum_{\pm}\gamma_{\pm} q^{-\frac{1}{2}}\xi^{-1}(A_{\pm}) \psi(\Tr(A_{\pm})+v\varpi^{-l}\Norm(A_{\pm})) \nonumber \\
	&& \qquad\cdot \sum_{\alpha \in \mathcal{O}/\p^{\delta }}\psi(-\frac{2}{3}a_3'\alpha^3\varpi^{\frac{r}{2}-3\delta}+(v+a_2')\alpha^2\varpi^{-2\delta}+2(1+v\Re(\overline{A_{\pm}})-a_1')\alpha\varpi^{-\frac{r}{2}-\delta}) \nonumber \\
	&& \quad \qquad \cdot \sum_{\beta\in \mathcal{O}/\p} \psi(\beta^2\varpi^{-1}(v+a_2')+\beta\varpi^{-\frac{r}{2}}2(-vb_2-a_1''+a_2''\alpha\varpi^{\frac{r}{2}-\delta}). \nonumber
\end{eqnarray}	

We first look at the $\beta$-sum. It turns out that
\begin{equation}
	v+a_2' \in v(1+4vb_1) + \p \subset \p. \nonumber
\end{equation}
Thus, the inner sum is actually linear and introduces the congruence condition
\begin{equation}
	\alpha a_2'' \varpi^{\frac{r}{2}-\delta} -vb_2-a_1'' \in \p^{\frac{r}{2}}. \nonumber
\end{equation}
A short computation shows that $a_2''\in \p^{\delta}$ which makes this congruence independent of $\alpha$. We may as well assume that it is satisfied since the remaining sum vanishes otherwise. We are left with the $\alpha$-sum. To deal with it we check that $a_3' = -8b_1v^3 \in\mathcal{O}^{\times}$. This implies at worst cubic cancellation. Indeed one can show, assuming $\Delta \in\p$, that the linear and the quadratic term are in $\p$. We obtain
\begin{equation}
	\abs{W_{\pi}(g_{-n, \frac{n}{2},v})} \leq 2q^{\frac{r}{6}+\frac{1}{2}} = 2q^{\frac{n}{12}+\frac{1}{2}}.\nonumber
\end{equation}

If $r-1$ even or $\Delta\in\mathcal{O}^{\times}$, one obtains
\begin{equation}
	W_{\pi}(g_{-n,\frac{n}{2},v}) \leq 2q^{\frac{n}{12}}. \nonumber
\end{equation}
The details are left to the reader.

\textbf{Case III: $\frac{n}{2}<l$.} In this case we have $k=2l$, $\rho=0$ and $t= -k =-2l$. In order to compute $S$ we write $x=x_1+x_2\Omega$ and obtain the system of equations
\begin{eqnarray}
	-b_2'+x_2 &\in &\p^{\lfloor \frac{r-1}{2} \rfloor}, \nonumber \\
	vx_1^2+x_1 -v\varpi x_2^2-b_1' &\in & \p^{\lceil \frac{r-1}{2} \rceil}, \nonumber
\end{eqnarray}
where $b_1'+b_2'\Omega = b_{\xi} \Omega^{2l-n}$. Because $2l>n$ we obtain a quadratic equation for $x_1$ with discriminant in $\mathcal{O}^{\times}$. Thus, we obtain maximal two solutions $\{x_+,x_-\}\subset (\mathfrak{O}/\mathfrak{P}^{r-1})^{\times}$. We write $S'$ for this set of solutions. Then we have
\begin{equation}
	S= \{ x+\alpha\Omega^{r-1} \colon x\in S', \alpha \in \mathfrak{O}/\mathfrak{P}\}\subset (\mathfrak{O}/\mathfrak{P}^r)^{\times}. \nonumber
\end{equation}
Inserting this parametrization of $S$ in \eqref{eq:prototype_ram_su_w} yields
\begin{equation}
	W_{\pi}(g_{-2l,l,v}) = \gamma q^{-\frac{1}{2}} \sum_{x\in S'} \xi^{-1}(x) \psi(\Tr(\Omega^t x)+v\varpi^{-l}\Norm(x)) \sum_{\alpha \in \mathfrak{O}/\mathfrak{P}} \psi_E(\Omega^{-2}(z_x \alpha +\frac{v}{2}\Norm(\alpha))). \nonumber
\end{equation}
Since $n(\psi_E)=-1$, the $\alpha$-sum is a complete exponential sum over a finite field. Therefore, we can use \eqref{eq:Weils_bound} and obtain
\begin{equation}
	\abs{W_{\pi}(g_{-2l,l,v})} \leq 2. \nonumber
\end{equation}

Finally, we observe that since $\xi$ does not factor through the norm we have $a(\xi)\geq 2$. So that $k\geq 2$ and we have covered all possible cases.
\end{proof}

\subsection{Twists of Steinberg}

In this section we analyse the behaviour of $W_{\pi}$ when $\pi$ is a twist of Steinberg. Before treating the most general case we look at two examples of small ramification.

\begin{lemma} \label{lm:St_bound}
If $\pi  = St$ then
\begin{equation}
	\abs{W_{\pi}(g_{t,l,v})} = \begin{cases} 
		q^{-(t+k)} &\text{ if } t\geq -k, \\
		0 &\text{ else,}
	\end{cases}\nonumber
\end{equation}
for $k=\max(2l,n)$. In particular
\begin{equation}
	\abs{W_{\pi}(g)} \leq 1. \nonumber 
\end{equation}
\end{lemma}
\begin{proof}
The claim is a direct consequence of  Lemma~\ref{lm:steinberg_exp}.
\end{proof}

\begin{lemma} \label{lm:St_twist_1_bound}
Let $q$ be odd and let $\pi = \chi St$ for $\chi\in \mathfrak{X}_1'$. Then we have
\begin{equation}
	W_{\pi}(g) \leq 2. \nonumber
\end{equation}
\end{lemma}
\begin{proof}
We consider all the cases appearing in Lemma~\ref{lm:twist_steinberg_exp}.  If $l=1$ and $t=-2$ we have
\begin{eqnarray}
	W_{\pi}(g_{-2,1,v}) &=& q\zeta_F(1)^{-2}  K(\chi\circ \Norm, (\varpi^{-1},\varpi^{-1}),v\varpi^{-1})  \nonumber \\	
	&=& q^{\frac{1}{2}}\zeta_F(1)^{-1}\epsilon(\frac{1}{2},\chi^{-1})\int_{\mathcal{O}^{\times}\setminus(-v^{-1}+\p)} \chi\left(\frac{x}{1+vx}\right)\psi(\varpi^{-1}x)d^{\times}x.\nonumber
\end{eqnarray}
This is an exponential sum over a finite field and \eqref{eq:Weils_bound} implies
\begin{equation}
	\abs{W_{\pi}(g_{t,l,v})} \leq 2. \nonumber
\end{equation}
The other cases are rather easy. One sees
\begin{equation}
	\abs{W_{\pi}(g_{t,l,v})} \leq 1 \nonumber
\end{equation}
for the remaining combinations of $l$ and $t$.
\end{proof}

\begin{lemma} \label{lm:St_twist_2_bound}
Let $\pi = \chi St$ for some unitary character $\chi$ with $\lceil \frac{a(\chi)}4 \rceil \geq \kappa_F$ and $a(\chi)>1$. Then we can evaluate the Whittaker function explicitly as follows.

If $l\neq \frac{n}{2},0$, then
\begin{equation}
	W_{\pi}(g_{t,l,v}) = \begin{cases} 
		\chi^2(x_0)\psi((x_0-b)\varpi^{-\frac{k}{2}}) &\text{ if } t=-k,\\
		0 &\text{ else,}
	\end{cases} \nonumber
\end{equation}
where $x_0$ is the unique solution to $v\varpi^{\frac{k}{2}-l}x^2+x+b\varpi^{\frac{k}{2}-a(\chi)}$ satisfying $v(x_0)=0$.

If $l=\frac{n}{2}$ and $-2> t>-n$, then $W_{\pi}(g_{t,l,v})=0$ unless $-bv\in \mathcal{O}^{\times 2}$. In the latter case we observe that
\begin{equation}
	W_{\pi}(g_{t,l,v}) = q^{-\frac{n+t}{4}}\gamma_F(-\frac{b}{2},l)\psi(-b\varpi^{-l})\sum_{\pm}\begin{cases}
		\gamma_F(\pm Y,-\frac{t}{2})\chi(-\frac{b}{v})\psi(\pm 2Y\varpi^{\frac{t}{2}}) \\
		\text{ if } Y^2=-\frac{b}{v}\in \mathcal{O}^{\times 2} \text{ and } t\geq -2\lceil\frac{l}{2}\rceil, \\
		\gamma_F(2v,-\frac{t}{2})\chi\left(b\frac{A_{\pm}+x_0\varpi^{-\frac{t}{2}-\lfloor\frac{l}{2}}\rfloor}{vA_{\mp}-x_0\varpi^{-\frac{t}{2}-\lfloor\frac{l}{2}}\rfloor}\right)\psi\left(A_{\pm}\varpi^{\frac{t}{2}}+x_0\varpi^{\lfloor\frac{l}{2}\rfloor}\right) \\
		\text{ if } Y^2=-4bv+\varpi^{n+t} = \Delta \in \mathcal{O}^{\times 2} \text{ and } t< -2\lceil\frac{l}{2}\rceil, \\
		\text{ where $x_0\in\mathcal{O}$ solves \eqref{eq:def_eq_for_x_0_tst} below and $A_{\pm}=-\frac{\varpi^{l+\frac{t}{2}}}{2v}\pm\frac{Y}{2v}$.}
	\end{cases} \nonumber
\end{equation}

If $-\frac{t}{2} = l=\frac{n}{2}$, we define $\Delta=1-4vb$. One has
\begin{equation}
	W_{\pi}(g_{t,l,v}) = q^{\frac{n}{12}}\gamma_F(-2v,l)\chi^{-1}(4v^2) \psi(\frac{3}{4v}\varpi^{-\frac{n}{2}})\text{Ai}_{\psi}(-16bv^3 \varpi^{\lceil \frac{l}{2}\rceil+2\{\frac{l}{2}\}-3\lfloor \frac{1}{2}\lceil\frac{l}{2}\rceil \rfloor};\Delta\varpi^{-\lfloor \frac{l}{2}\rfloor-\lfloor \frac{1}{2}\lceil\frac{l}{2}\rceil \rfloor}), \nonumber
\end{equation}
for $v(\Delta)\geq \lceil\frac{l}{2}\rceil$ and
\begin{equation}
	W_{\pi}(g_{t,l,v}) = \delta_{\Delta \in F^{2\times}}q^{\frac{\Delta}{4}}\sum_{\pm}\gamma_F(-1\pm \sqrt{\Delta},\rho)\gamma_F(\Delta \pm \sqrt{\Delta},l-\frac{v(\Delta)}{2})\chi^2(-\frac{1}{2v}\pm\frac{\sqrt{\Delta}}{2v}) \psi(\varpi^{-\frac{n}{2}}(\frac{\Delta-3}{4v})), \nonumber
\end{equation}
for $v(\Delta)<\lceil\frac{l}{2}\rceil$.

In particular,
\begin{equation}
	\abs{W_{\pi}(g)} \leq 2q^{\frac{n}{12}} \text{ and } \sup_g \abs{W_{\pi}(g)} \gg_q q^{\frac{n}{12}}. \nonumber
\end{equation}
\end{lemma}

\begin{proof}
Define $k=\max(n,2l)$. For $l=0$, $l\geq n$ or $t\geq -2$ the statement follows immediately from the expressions given in Lemma~\ref{lm:twist_steinberg_exp}. Thus, from now on we assume $0<l<n$ and $t<-2$. By Lemma~\ref{lm:twist_steinberg_exp} we have
\begin{eqnarray}
	W_{\pi}(g_{t,l,v}) &=& \zeta_F(1)^{-2} q^{-\frac{t}{2}} K(\chi\circ \Norm, (\varpi^{\frac{t}{2}},\varpi^{\frac{t}{2}}), v\varpi^{-l}), \nonumber
\end{eqnarray}
for $E=F\times F$. We will evaluate the oscillatory integral $K$ starting from the prototype given in Lemma~\ref{lm:prototype_K_stat}. To do so we need to investigate the structure of the critical set $S$ in several cases.

\textbf{Case I: $0<l<\frac{n}{2}$.} In this situation we have $k=n=-t$. The matrix $A_{x_1,x_2}$ given in Lemma~\ref{lm:prototype_K_stat} is diagonal modulo $\p$ and independent of $x_1$ and $x_2$. Furthermore, the congruence conditions for $(x_1,x_2) \in S$ read
\begin{eqnarray}
	x_1 -x_2 &\in& \p^r, \nonumber \\
	v\varpi^{\frac{k}{2}-l}x_1^2+x_1+b&\in& \p^r. \nonumber
\end{eqnarray}
By Lemma~\ref{lm:quad_cong} we conclude that $\sharp S =1$. Therefore, we have
\begin{eqnarray}
	W_{\pi}(g_{t,l,v})&=&  q^{\rho}\sum_{(x_0,x_0) \in S}  \chi^2(x_0) \psi(2x_0 \varpi^{-\frac{k}{2}}+v\varpi^{-l} x_0^2) \nonumber  \\
	&&\qquad\qquad G(-\frac{b}{2}\varpi^{-\rho},  (b+x_0+vx_0^2\varpi^{\frac{k}{2}-l})\varpi^{-r-\rho})^2.\nonumber
\end{eqnarray}
By Lemma~\ref{lm:eval_quad_GS} we arrive at 
\begin{equation}
	W_{\pi}(g_{t,l,v}) = \chi(x_0)^2 \psi((x_0-b) \varpi^{-\frac{k}{2}}) \nonumber
\end{equation}
where $x_0$ is the unique solution of $v\varpi^{\frac{k}{2}-l}x_1^2+x_1+b=0$ satisfying $v(x_0)=0$.

\textbf{Case II: $l=\frac{n}{2}$.} This is the transition region where the Whittaker function can be nonzero for several $t$. Recall that the congruences defining $S$ are
\begin{eqnarray}
	b+\varpi^{l+\frac{t}{2}}x_1+vx_1x_2 &\in&\p^r, \label{eq:first_cong_crit_TST} \\
	b+\varpi^{l+\frac{t}{2}}x_2+vx_1x_2 &\in&\p^r. \label{eq:second_cong_crit_TST} 
\end{eqnarray}

\textbf{Case II.1: $-a(\chi)-\rho\leq t<-2$.} The congruences simplify to $x_1x_2 \in -\frac{b}{v} +\p^r$. Which has a unique solution $x_2$ for each $x_1\in (\mathcal{O}/\p^r)^{\times}$. Using the (non obvious) fact that the $S$-sum in Lemma~\ref{lm:prototype_K_stat} is well defined modulo $\p^r$ we obtain
\begin{eqnarray}
	W_{\pi}(g_{t,a(\chi),v}) &=& q^{-\frac{t}{2}-2r} \chi\left(-\frac{b}{v}\right) \psi\left(-\frac{b}{v}\varpi^{-a(\chi)}\right)\sum_{x\in (\mathcal{O}/\p^r)^{\times}}   \psi\left(x \varpi^{\frac{t}{2}}-x^{-1}\frac{b}{v} \varpi^{\frac{t}{2}}\right) \nonumber \\
	&&\qquad\qquad G\left( \frac{\varpi^{-\rho}}{2} \left(\begin{matrix} -b &  -b \\ -b & -b\end{matrix}\right), \varpi^{r +\frac{t}{2}}\left(\begin{matrix} x \\ -\frac{b}{vx} \end{matrix}\right) \right). \nonumber
\end{eqnarray}
Evaluating the Gau\ss\  sum using Lemma~\ref{lm:eval_quad_GS} yields
\begin{equation}
	G\left( \frac{\varpi^{-\rho}}{2} \left(\begin{matrix} -b &  -b \\ -b & -b\end{matrix}\right), \varpi^{r +\frac{t}{2}}\left(\begin{matrix} x \\ -\frac{b}{vx} \end{matrix}\right) \right) = 		\gamma_F(-\frac{b}{2},\rho)q^{-\frac{\rho}{2}} \nonumber
\end{equation}
if $t\geq -a(\chi)$ or $t=-a(\chi)-1$ and $x^2 \in -\frac{b}{v}+\p$. Otherwise, the Gau\ss\  sum vanishes. Thus, for $t\geq -a(\chi)$ we get
\begin{eqnarray}
	W_{\pi}(g_{t,a(\chi),v}) &=& \gamma_F\left(-\frac{b}{2},\rho\right) \chi\left(-\frac{b}{v}\right)\psi(-b\varpi^{-l})  \zeta_F(1)^{-1}q^{-\frac{t}{2}-r-\frac{\rho}{2}} S_1\left(1,-\frac{b}{v},-\frac{t}{2}\right). \nonumber
\end{eqnarray}
Evaluating the Kloosterman sum reveals
\begin{equation}
	W_{\pi}(g_{t,a(\chi),v}) = 	q^{-\frac{t}{4}-\frac{n}{4}}  \sum_{\pm}\gamma_F\left(-\frac{b}{2},\rho\right)\gamma_F\left(\pm Y,2\{-\frac{t}{4}\}\right) \chi\left(-\frac{b}{v}\right)\psi((\pm Y\varpi^{l+\frac{t}{2}}-b)\varpi^{-l})\label{eq:twist_steinberg_big_t}
\end{equation}
if $Y^2 = -\frac{b}{v} \in \mathcal{O}^{\times 2}$. Otherwise, the Whittaker function vanishes. 

For $t=-a(\chi)-1$ we observe that the critical points of the Kloosterman sum are congruent to $Y$ modulo $\p$. Thus, we also arrive at \eqref{eq:twist_steinberg_big_t}.

\textbf{Case II.2: $-n< t<-a(\chi)-\rho$.} The first congruence, \eqref{eq:first_cong_crit_TST}, can be rewritten as
\begin{equation}
	x_1 \in  -b(\varpi^{l+\frac{t}{2}}+vx_2)^{-1}  +\p^r. \nonumber
\end{equation}
Substituting this in the second congruence, \eqref{eq:second_cong_crit_TST}, yields
\begin{equation}
	vx_2^2 +\varpi^{l+\frac{t}{2}}x_2 + b \in \p^{-\frac{t}{2}-r-\rho}. \nonumber
\end{equation}
It is easy to see that the discriminant of this equation satisfies $v(\Delta)=0$. We can parametrize $x_2$ using Lemma~\ref{lm:quad_cong}. We compute
\begin{eqnarray}
	W_{\pi}(g_{t,l,v}) &=& q^{-\frac{t}{2}-l+\rho}\chi(-b)\psi(-b\varpi^{-l})\sum_{x_2}\chi(x_2)\chi^{-1}(vx_2+\varpi^{l+\frac{t}{2}})\psi(x_2\varpi^{\frac{t}{2}})\nonumber \\
	&&\qquad\qquad \cdot G\left(\frac{v^2x_2^2}{2}\left(\begin{matrix} -b & -b \\ -b &-b\end{matrix}\right)\varpi^{-\rho},\varpi^{r+\frac{t}{2}}\left(\begin{matrix} 0 \\ vx^2_2+x_2\varpi^{l+\frac{t}{2}}+b\end{matrix} \right)\right). \nonumber
\end{eqnarray} 
We use Lemma~\ref{lm:quad_cong} to parametrise the family $x_2$ and set $A_{\pm} = -\frac{\varpi^{l+\frac{t}{2}}}{2v}\pm \frac{\sqrt{\Delta}}{2v}\in\mathcal{O}^{\times}$ to shorten notation. Observe that $vA_{\pm}+\varpi^{\frac{t}{2}+l} = -vA_{\mp}$. For $-t> a(\chi)+\rho$ we can use Lemma~\ref{lm:char_trick} and get
\begin{eqnarray}
	W_{\pi}(g_{t,l,v}) &=& q^{\frac{\rho}{2}}\gamma_F(-\frac{b}{2},\rho)\sum_{\pm}\chi\left(\frac{bA_{\pm}}{vA_{\mp}}\right)\psi(A_{\pm}\varpi^{\frac{t}{2}}-b\varpi^{-l})\nonumber \\
	&&\quad \cdot \int_{\mathcal{O}} \psi\left(\sum_{j\geq 2}\frac{b}{j}((-1)^jA_{\pm}^j-A_{\mp}^j)\left(\frac{-vx}{b}\right)^j\varpi^{-(\frac{t}{2}+r)j-l}\right)dx. \nonumber
\end{eqnarray}
One checks that $A_{\pm}^2-A_{\mp}^2 = \pm\frac{\sqrt{\Delta}}{v^2}\varpi^{l+\frac{t}{2}}$. Furthermore, the binomial expansion shows that $((-1)^jA_{\pm}^j-A_{\mp}^j)\in\p^{l+\frac{t}{2}}$. Evaluating the remaining oscillatory integral yields
\begin{eqnarray}
	W_{\pi}(g_{t,l,v}) &=& q^{-\frac{t}{4}-\frac{n}{4}}\psi(-b\varpi^{-l})\gamma_F(-\frac{b}{2},\rho)\nonumber \\
	&&\qquad \cdot \sum_{\pm}\gamma_F\left(\frac{\sqrt{\Delta}}{2b},-\frac{t}{2}\right) \chi\left(b\frac{A_{\pm}+x_0\varpi^{-\frac{t}{2}-\lfloor\frac{l}{2}}\rfloor}{vA_{\mp}-x_0\varpi^{-\frac{t}{2}-\lfloor\frac{l}{2}}\rfloor}\right)\psi\left(A_{\pm}\varpi^{\frac{t}{2}}+x_0\varpi^{\lfloor\frac{l}{2}\rfloor}-b\varpi^{-l}\right), \nonumber
\end{eqnarray}
where $x_0\in\mathcal{O}$ is the unique solution to
\begin{equation}
	\sum_{j\geq 2} -v\left(-\frac{v}{b}x\right)^{j-1}\frac{(-1)^jA_{\pm}^j-A_{\mp}^j}{\varpi^{l+\frac{t}{2}}}\varpi^{-rj-(j-1)\frac{t}{2}} =0.\label{eq:def_eq_for_x_0_tst}
\end{equation}

\textbf{Case II.3: $t=-n$.}
In this case we have to solve the congruences
\begin{eqnarray}
	x_1-x_2  &\in&  \p^r, \nonumber\\
	vx_1^2+x_1 +b &\in& \p^r. \nonumber
\end{eqnarray}
The quadratic congruence has discriminant
\begin{equation}
	\Delta  = \Delta(v) = 1-4bv. \nonumber
\end{equation}
For some $v$ the discriminant might be ($p$-adically) small, so that there are many solutions for $x_1$. In this case we have to argue slightly more carefully.

Using Lemma~\ref{lm:quad_cong} we can parametrize
\begin{equation}
	S = \left\{ \left(-\frac{1}{2v}\pm \frac{Y}{2v}\varpi^{\delta}+\alpha \varpi^{r-\delta}, -\frac{1}{2v}\pm \frac{Y}{2v}\varpi^{\delta}+\alpha \varpi^{r-\delta}\right) \in ((\mathcal{O}/\p^r)^{\times})^2 \colon \alpha \in \mathcal{O}/\p^{\delta} \right\}. \label{eq:shape_S_very_deg}
\end{equation}
We set 
\begin{equation}
	A_{\pm} =-\frac{1}{2v}\pm \frac{Y}{2v}\varpi^{\delta} \nonumber
\end{equation}
and $\gamma_{\pm} = 1$ if $v(\Delta)<r$ and $\gamma{\pm} =\frac{1}{2}$ otherwise. We can rewrite the Gau\ss\  sum from Lemma~\ref{lm:prototype_K_stat} as
\begin{eqnarray}
	&G\left(\frac{\varpi^{-\rho}}{4}\left(\begin{matrix}-1&1\\ 1 & 1 \end{matrix}\right)\left(\begin{matrix} -b-vx_1^2& 0 \\ 0 & -b+vx_1^2 \end{matrix}\right)\left(\begin{matrix} -1&1\\1&1\end{matrix}\right) ,\varpi^{-r-\rho} B_{x_1,x_2}\right) \nonumber \\
	&= G\left(\frac{\varpi^{-\rho}x_1}{4}, 0\right) G\left( \frac{-\varpi^{-\rho}}{4}(-2b-x_1), \varpi^{-r-\rho}(x_1+vx_1^2+b)\right). \nonumber
\end{eqnarray}
First, we consider the degenerate case, $v(\Delta)>0$. In particular, we have $-2b-x_1 \in\p$ and $v(\Delta)\geq 2 > \rho$. Therefore, we obtain the stronger congruence $vx_1^2+x_1+b\in \p^{r+\rho}$. Using our parametrisation of $S$ in Lemma~\ref{lm:prototype_K_stat} yields
\begin{eqnarray}
	W_{\pi}(g_{t,l,v}) &=& q^{\frac{\rho}{2}}\sum_{\pm}\gamma_F(A_{\pm},\rho)\chi^2(A_{\pm}) \psi(\varpi^{-\frac{n}{2}}(2A_{\pm}+vA_{\pm}^2)) \nonumber \\
	&& \qquad \cdot \sum_{\alpha\in\mathcal{O}/\p^{\delta-\rho}} \chi^2(1+\frac{\alpha}{A_{\pm}}\varpi^{r+\rho-\delta})\psi((1\pm Y\varpi^{\delta})\alpha \varpi^{-r-\delta}+v\alpha^2\varpi^{\rho-2\delta}) \nonumber \\
	&=&q^{\delta-\frac{\rho}{2}}\sum_{\pm}\gamma_F(A_{\pm},\rho)\chi^2(A_{\pm}) \psi(\varpi^{-\frac{n}{2}}(2A_{\pm}+vA_{\pm}^2)) \nonumber \\
	&& \quad \cdot \int_{\mathcal{O}} \psi(\underbrace{(1\pm Y\varpi^{\delta}+\frac{2b}{A_{\pm}})}_{=\frac{-\Delta}{2vA_{\pm}}+\frac{Y^2\varpi^{2\delta}}{2vA_{\pm}}\in \p^{r+\rho}}\alpha\varpi^{-r-\delta}+\underbrace{(v-\frac{b}{A_{\pm}^2})}_{\in\frac{\Delta \mp Y\varpi^{\delta}}{2vA_{\pm}^2}+\p^{r+\rho}}\alpha^2\varpi^{\rho-2\delta}+\frac{2b}{3A_{\pm}^3}\alpha^3\varpi^{r+2\rho-3\delta})d\alpha. \nonumber
\end{eqnarray}
Here we used the fact that $r+\rho-\delta\geq \frac{r}{2}+\rho \geq \kappa_F$ to insert the Taylor expansion of the $p$-adic logarithm. Since $2r-4\delta \geq 0$ we could truncate after the third term. Since the cubic coefficient is a unit we have the bound
\begin{equation}
	\abs{W_{\pi}(g_{t,l,v})} \leq 2 q^{\frac{r}{3}+\frac{\rho}{6}} = 2 q^{\frac{n}{12}}. \nonumber
\end{equation}  
Even more, if $v(\Delta)\geq r+\rho$, we have
\begin{equation}
	W_{\pi}(g_{t,l,v}) = q^{\frac{n}{12}}\gamma_F(-2v,\rho)\chi^{-1}(4v^2) \psi(\frac{3}{4v}\varpi^{-\frac{n}{2}})\text{Ai}_{\psi}(-16bv^3 \varpi^{r+2\rho-3\delta};\Delta\varpi^{-r-\delta}). \nonumber
\end{equation}
If $\rho<\Delta<r+\rho$, we arrive at
\begin{equation}
	W_{\pi}(g_{t,l,v}) = q^{\frac{\delta}{2}}\sum_{\pm}\gamma_F(-1\pm \sqrt{\Delta},\rho)\gamma_F(\Delta \pm \sqrt{\Delta},\abs{2\{\frac{\delta-\rho}{2}\}})\chi^2(-\frac{1}{2v}\pm\frac{\sqrt{\Delta}}{2v}) \psi(\varpi^{-\frac{n}{2}}(\frac{\Delta-3}{4v})). \label{eq:exp_eval_mild_disc_st}
\end{equation}
Note that in this case $W_{\pi}(g_{t,l,v})$ vanishes when $\Delta\not \in F^{\times 2}$. Finally, if $0= v(\Delta)$, one easily checks that \eqref{eq:exp_eval_mild_disc_st} holds as well.

\textbf{Case III: $\frac{n}{2}<l<n$.} Here we have $k=2l$ and the Whittaker function is non zero only for $t=-k$. The congruence conditions defining $S$ yield the system of equations 
\begin{eqnarray}
	x_1 -x_2 &\in& \p^r, \nonumber \\
	(x_1+\frac{1}{2v})^2 &\in& \frac{1}{4v^2}-\frac{b}{v}\varpi^{\frac{k}{2}-a(\chi)} + \p^r \subset \mathcal{O}^{\times}. \nonumber
\end{eqnarray}
The quadratic equation has a unique solution modulo $\p^r$ which is in $\mathcal{O}^{\times}$. Thus, $\sharp S=1$. From Lemma~\ref{lm:prototype_K_stat} and Lemma~\ref{lm:eval_quad_GS} we obtain
\begin{equation}
	W_{\pi}(g_{-2l,l,v}) = \chi(x_0)^2 \psi((x_0-b\varpi^{\frac{k}{2}-a(\chi)}) \varpi^{-\frac{k}{2}}) \nonumber
\end{equation}
where $x_0$ is the unique solution of $vx_1^2+x_1+b\varpi^{\frac{k}{2}-a(\chi)}=0$ satisfying $v(x_0)=0$.
	
This was the last case to be considered and the proof is complete.
\end{proof}

\begin{rem}
There are several other ways to evaluate the integral $K(\chi\circ\Norm,\cdot,\cdot)$. For example one may compute that
\begin{equation}
	K(\chi\circ\Norm,(\varpi^{-k},\varpi^{-k}),v\varpi^{-l}) = \int_{\mathcal{O}^{\times}}\chi(x)\psi(xv\varpi^{-l})S_1(1,x,k)d^{\times}x. \nonumber
\end{equation}
For $k>1$ the Kloosterman sum may be evaluated and one is left with a twisted quadratic Gau\ss\  sum. To treat the remaining sum one can apply the (1-dimensional) method of stationary phase. This turns out to be similar in spirit to the calculation in \cite[Lemma~10]{BM15}.
\end{rem}

\subsection{Irreducible principal series}

The last category of representations to deal with are principal series representations. It is already known from \cite{Sa15_2} that in certain degenerate situations the Whittaker function can be as large as the local bound \eqref{eq:loc_bound_Sa} predicts. We consider several different situations starting with one that is similar to the twisted Steinberg representations.

\begin{lemma} \label{lm:balanced_ps_bound}
Let $\pi = \chi \abs{\cdot}^s \boxplus \chi \abs{\cdot}^{-s}$ for some $s\in i \R$ and $\lceil\frac{a(\chi)}{4}\rceil \geq \kappa_F$. Then we have
\begin{equation}
	\abs{W_{\pi}(g)}\leq 2q^{\frac{n}{12}} \text{ and } \sup_{g}\abs{W_{\pi}(g)}\gg_{q} q^{\frac{n}{12}}. \nonumber
\end{equation}
Even more, for $t<-2$ we can evaluate the $W_{\pi}(g_{t,l,v})$ explicitly and obtain the same expressions as in Lemma~\ref{lm:St_twist_2_bound}.
\end{lemma}
\begin{proof}
 From the proof of Lemma~\ref{lm:equiv_ps_expwh} we see that for $t\leq -2$ we are in the same situation as for $\pi = \chi St$. The remaining cases can be estimated trivially using Lemma~\ref{lm:equiv_ps_expwh}.
\end{proof}

We move on to the unbalanced principal series $\pi =\chi_1\boxplus\chi_2$ with unramified $\chi_2$.

\begin{lemma}  \label{lm:unbalanced_ps_bound_1}
Let $\pi = \chi \abs{\cdot}^s\boxplus \abs{\cdot}^{-s}$ for $s\in i\R$ and put $n=a(\chi)>0$. Then
\begin{equation}
	\abs{W_{\pi}(g)} \leq q^{\frac{1}{2}\lfloor\frac{n}{2} \rfloor}  \nonumber
\end{equation}
\end{lemma}
This follows from \cite[Corollary~2.35]{Sa15_2} and is sharp by the lower bound given in \cite[Proposition~2.39]{Sa15_2}. We will give a proof based on properties of epsilon factors. In particular their stability under twists  by characters with small log-conductor. 
\begin{proof}
First we observe from Lemma~\ref{lm:unbalanced_ps} that
\begin{equation}
	\abs{W_{\pi}(g_{t,l,v})} \leq 1  \nonumber
\end{equation}
if $t\geq -n$ and $l=0$ or $t\geq -2l$ and $l\geq n$. The only other non-zero situations are $0<l<n$ and $t=-n-l$. In these cases the delta term in \eqref{eq:unbalanced_ps_as_epss} does not contribute so that we obtain
\begin{equation}
	\abs{W_{\pi}(g_{t,l,v})} = \zeta_F(1) q^{-\frac{l}{2}}\abs{\sum_{\mu\in \mathfrak{X}_l}\epsilon(\frac{1}{2},\mu^{-1} \chi^{-1})\mu(-v)}. \nonumber
\end{equation}	
If $l\leq \lfloor \frac{n}{2}\rfloor$, we use \cite[Lemma~2.37]{Sa15_2} to obtain\footnote{This is basically the argument used in the proof of \cite[Proposition~2.39]{Sa15_2}.}
\begin{eqnarray}
	\abs{W_{\pi}(g_{t,l,v})} &=& \zeta_F(1) q^{-\frac{l}{2}}\abs{\sum_{\mu\in \mathfrak{X}_l}\epsilon(\frac{1}{2},\chi^{-1})\mu(b_{\chi}v)} 
	= \begin{cases}
		q^{\frac{l}{2}} &\text{ if } v \in b_{\chi}^{-1}+\p^l, \\
		0 &\text{ else.}
	\end{cases} \label{eq:size_unbalanced_whitt_smalll}
\end{eqnarray}
On the other hand, if $l > \lfloor\frac{n}{2}\rfloor$, Lemma~\ref{lm:unbalanced_ps} shows
\begin{equation}
	\abs{W_{\pi}(g_{t,l,v})} = \zeta_F(1)^{-1}q^{-\frac{t}{2}}\abs{G_l(-v^{-1}\varpi^{t+l},\chi^{-1})} = q^{\frac{n-l}{2}}\abs{ I_{\chi}(l,v)}. \nonumber
\end{equation}
for
\begin{eqnarray}
	I_{\chi}(l,v) &=& \int_{\mathcal{O}} \chi(1+\varpi^{l}x)\psi(-v^{-1}\varpi^{l-n}x)dx 
	= \int_{\mathcal{O}} \psi((b_{\chi}-v^{-1})\varpi^{l-n}x) dx. \nonumber
\end{eqnarray}
In the last step we used Lemma~\ref{lm:char_trick}. The remaining integral can be calculated using Gau\ss\  sums. Indeed,
\begin{equation}
	\abs{W_{\pi}(g_{t,l,v})} = \begin{cases}
		q^{\frac{n-l}{2}} &\text{ if }v(b_{\chi}-v^{-1})\geq n-l, \\
		0 &\text{ else.}
	\end{cases} \nonumber
\end{equation}
This concludes the proof.
\end{proof}

\begin{rem} \label{rem:bound_inc_GS}
A byproduct of the previous lemma is the bound
\begin{equation}
	\abs{G_l(v\varpi^{-a(\chi)},\chi)}\leq \zeta_F(1)q^{-\frac{a(\chi)}{2}}, \nonumber
\end{equation}
for $l\leq \frac{a(\chi)}{2}$.
\end{rem}

We move towards more generic situations.

\begin{lemma} \label{lm:unbalanced_ps_bound_2}
Let $\pi = \chi_1\abs{\cdot}^{s} \boxplus \chi_2\abs{\cdot}^{-s}$ be a irreducible principle series representation. Also assume $a(\chi_1)>a(\chi_2)$ and $s\in i\R$. Define $m=\max(2l,n)$. One has
\begin{equation}
	\abs{W_{\pi}(g_{t,l,v})} \leq 2q^{-\frac{t+m}{2}} \nonumber
\end{equation}
as long as $n$ is odd or $a_2\leq 3$ or $b_1b_2\not \in \mathcal{O}^{\times 2}$. In general we might encounter a degenerate critical point for $l=\frac{n}{2}$ which leads to the weaker bound
\begin{equation}
	\abs{W_{\pi}(g_{-l-a(\chi_1),l,v})} \leq 2q^{\frac{n}{4}-\frac{a_2}{3}}. \nonumber
\end{equation}
\end{lemma}
\begin{proof} 
To simplify notation we write $a_1=a(\chi_1)$ and $a_2=a(\chi_2)$ and put $k=\max(a_1,l)=2r+\rho$. We also write $b_i = b_{\chi_i}\in \mathcal{O}^{\times}$ for the constant associated to $\chi_i$ in Lemma~\ref{lm:char_trick} for $i=1,2$.

From Lemma~\ref{lm:generic_ps_shape_Wf} it is clear that
\begin{equation}
	\abs{W_{\pi}(g_{t,l,v})}=1 \text{ if } l=0\text{ and } t=-n \text{ or } l\geq n \text{ and } t=-2l. \nonumber
\end{equation}
We also see that 
\begin{equation}
	\abs{W_{\pi}(g_{t,a_i,v})} = q^{-\frac{t+n}{2}} \text{ if } t\geq -a_1 \text{ and } i=1,2. \nonumber
\end{equation}
To evaluate the remaining non zero cases  we have to understand 
\begin{equation}
	K(\chi_1\otimes \chi_2, (\varpi^{-l_1},\varpi^{-l_2}),v\varpi^{-l}) \nonumber
\end{equation}
for suitable $0<l_1,l_2\leq l$. In most cases we will do this as before using the method of stationary phase. However, if $l$ is small, we find it easier to exploit the stability of $\epsilon$-factors directly similarly to the approach in \cite[Proposition~2.40]{Sa15_2}.

\textbf{Case I: $l\leq \frac{a_1}{2}$ and $t<-a_1$.} In this case we have $l_1=a_1$, and $l_2=-t-a_1$. Since we assume $t<-a_1$ the delta term in \eqref{eq:orig_generic_contribution} does not contribute. We have
\begin{equation}
	W_{\pi}(g_{t,l,v}) = \zeta_F(1)^{-2} q^{-\frac{t}{2}} q^{s(l_1-l_2)} \sum_{\mu\in \mathfrak{X}_l} G(\varpi^{-a_1},\mu\chi_1)G(\varpi^{-l_2},\mu\chi_2)G(v\varpi^{-l}, \mu^{-1}). \nonumber
\end{equation}
Recall from \cite[Lemma~2.37]{Sa15_2} that
\begin{equation}
	\epsilon(\frac{1}{2},\mu^{-1}\chi_1^{-1}) = \epsilon(\frac{1}{2},\chi_1^{-1})\mu(-b_1). \nonumber
\end{equation}
This implies
\begin{equation}
	G(\varpi^{-a_1},\mu\chi_1) = \zeta_F(1)q^{-\frac{a_1}{2}}\mu(-b_1)\epsilon(\frac{1}{2},\chi_1^{-1}). \nonumber
\end{equation}
Inserting this expression above yields
\begin{equation} 
	W_{\pi}(g_{t,l,v}) = \zeta_F(1)^{-1} q^{-\frac{a_1+t}{2}} q^{s(l_1-l_2)} \epsilon(\frac{1}{2},\chi_1^{-1})\sum_{\mu\in \mathfrak{X}_l}G(\varpi^{-l_2},\mu\chi_2)G(-vb_1\varpi^{-l}, \mu^{-1}). \nonumber
\end{equation}
We can now evaluate the $\mu$-sum by writing the Gau\ss\  sum as an integral, taking the $\mu$-sum inside, and exploiting full cancellation. One arrives at
\begin{equation}
	W_{\pi}(g_{t,l,v}) = \zeta_F(1)^{-1} q^{-\frac{a_1+t}{2}} q^{s(l_1-l_2)} \epsilon(\frac{1}{2},\chi_1^{-1}) G(\varpi^{-t-a_1}-\varpi^{-l}vb_1, \chi_2). \nonumber
\end{equation}
By evaluating the remaining Gau\ss\  sum we obtain the sharp bounds
\begin{equation}
	\abs{W_{\pi}(g_{t,l,v})} \leq \begin{cases}
		1 &\text{ if } l<a_2 \text{ and }t=-n, \\
		q^{-\frac{t+n}{2}} &\text{ if }	l=a_2 \text{ and } -n\leq t<-a_1, \\
		q^{\frac{l-a_2}{2}} &\text{ if } l>a_2 \text{ and } t=-a_1-l, \\
		0 &\text{ else.}
	\end{cases} \nonumber
\end{equation}

\textbf{Case II: $\frac{a_1}{2}< l < a_2$.} In this case we have $l_1=a_1$, $l_2=a_2$ and $t=-n$. Note that $a_1-a_2\geq r$ implies $a_2\leq \frac{a_1}{2}$. However, this situation was covered in Case~I. Thus, we assume $a_1-a_2<r$. Our starting point is Lemma~\ref{lm:prototype_K_stat} together with Lemma~\ref{lm:generic_ps_shape_Wf}. Using Lemma~\ref{lm:eval_quad_GS} we compute the Gau\ss\  sum to be
\begin{eqnarray}
	G\left( \frac{\varpi^{-\rho}}{2}A_{x_1,x_2}, \varpi^{-r-\rho}B_{x_1,x_2}\right)
	&=& q^{-\frac{\rho}{2}} \gamma_F(-\frac{b_1}{2},\rho), \nonumber
\end{eqnarray}
whenever $x_1$ and $x_2$ satisfy
\begin{equation}
	b_1+x_1+vx_1x_2\varpi^{a_1-l}\in \p^r \text{ and } b_2\varpi^{a_1-a_2}+x_2\varpi^{a_1-a_2}+vx_1x_2\varpi^{a_1-l}\in \p^{r+\rho}. \nonumber
\end{equation}
This is reformulated to the congruences
\begin{eqnarray}
	&&x_1 \in -b_1+x_2\varpi^{a_1-a_2}+b_2\varpi^{a_1-a_2}+\p^r, \nonumber \\
	&&x_2^2 v \varpi^{a_1-l}+x_2(1+vb_1\varpi^{a_2-l}+vb_2\varpi^{a_1-l}) + b_2 \in \p^{r+\rho+a_2-a_1}. \nonumber 	
\end{eqnarray}
By Lemma~\ref{lm:quad_cong} $x_2\in\mathcal{O}^{\times}$ is uniquely determined modulo $\p^{r+a_2-a_1}$. Say $x_0$ is the unique solution in $(\mathcal{O}/\p^{r+a_2-a_1})^{\times}$ then we have
\begin{equation}
	S= \{ (x_0\varpi^{a_1-a_2}-b_1+b_2\varpi^{a_1-a_2}, x_0+\alpha\varpi^{r+\rho+a_2-a_1}) \colon \alpha \in \mathcal{O}/\p^{a_1-a_2}\}.\nonumber
\end{equation}
We insert this parametrization in the $S$-sum from Lemma~\ref{lm:prototype_K_stat}. Elementary rearrangements yield $\tilde{\gamma}\in S^1$ and $c\in \mathcal{O}^{\times}$ such that
\begin{eqnarray}
	W_{\pi}(g_{t,l,v}) &=& \tilde{\gamma} \zeta_F(1)^{-1} q^{\frac{a_2}{2}} G_{r+\rho+a_2-a_1}(cx_0 \varpi^{-a_2}, \chi_2). \nonumber
\end{eqnarray}
After checking that $r+a_2-a_1 \leq\frac{a_2}{2}$ we apply Remark~\ref{rem:bound_inc_GS} and obtain
\begin{equation}
	\abs{W_{\pi}(g_{t,l,v})} \leq 1. \nonumber
\end{equation}

\textbf{Case III: $\frac{a_1}{2}< l  =  a_2$.} This leads to $k=2r=a_1$, $l_1 = a_1$, and $l_2=-t-a_1$ for $-a_1>t\geq -n$. The situation turns out to very similar to the one in Case~II. One arrives at
\begin{eqnarray}
	W_{\pi}(g_{t,a_2,v})&=& \tilde{\gamma} \zeta_F(1)^{-1}q^{-\frac{t+a_1}{2}}G_{r+\rho+a_2-a_1}(c\varpi^{-a_2},\chi_2) \nonumber
\end{eqnarray} 
for some $\tilde{\gamma}\in S^1$ and some $c\in \mathcal{O}^{\times}$. Estimating the incomplete Gau\ss\  sum using Remark~\ref{rem:bound_inc_GS} yields
\begin{equation}
	\abs{W_{\pi}(g_{t,a_2,v})}\leq q^{-\frac{t+n}{2}}. \nonumber
\end{equation}

\textbf{Case IV: $l = a_1$.} We have $l_1 = -t-l$ and $-a_1>t\geq -2l$. We can rewrite the congruences defining $S$ and get
\begin{eqnarray}
	&&x_2\in x_1\varpi^{t+2l}+b_1-b_2\varpi^{a_1-a_2}+\p^r, \nonumber \\
	&&x_1^2v\varpi^{t+2l} +x_1(vb_1+\varpi^{t+2l}-vb_2\varpi^{a_1-a_2})+b_1\in \p^r. \nonumber
\end{eqnarray}
Depending on $t$ the behaviour can be quite different. The matrix $A_{\p}$ turns out to be
\begin{equation}
	A_{\p}= \left( \begin{matrix} -b_1 & vx_1x_2 \\ vx_1x_2 & 0 \end{matrix} \right) . \nonumber
\end{equation}

\textbf{Case IV.1: $t>-2l$.} In this case we have $(vb_1+\varpi^{t+2l}-vb_2\varpi^{a_1-a_2})\in \mathcal{O}^{\times}$. Thus, Lemma~\ref{lm:quad_cong} implies $\sharp S =1$. Estimating trivially yields
\begin{equation}
	\abs{W_{\pi}(g_{t,a_1,v})} = q^{-\frac{t+2l}{2}}\leq 1. \nonumber
\end{equation}

\textbf{Case IV.2: $t=-2l$.} Each $x_1$ determines a unique $x_2$ modulo $\p^r$. The quadratic congruence determining $x_1$ reads
\begin{equation}
	vx_1^2 +x_1(vb_1+1-vb_2\varpi^{a_1-a_2})+b_1\in \p^r \nonumber
\end{equation}
and has discriminant 
\begin{equation}
	\Delta = (1+vb_1-vb_2\varpi^{a_1-a_2})^2-4vb_1 = (1-vb)^2-4vb_2\varpi^{a_1-a_2}. \nonumber
\end{equation}
We can rewrite the quadratic equation as
\begin{equation}
	(2vx_1+1+vb)^2 \in \Delta + \p^r,\nonumber
\end{equation}
for $b=b_1-b_2\varpi^{a_1-a_2}$. Suppose that $v(\Delta)>0$. In particular, $1-vb\in \p$.  So that any admissible $x_1$ satisfies
\begin{equation}
	x_1 \in -\frac{1}{2v}-\frac{b_1}{2} +\p. \nonumber
\end{equation}
Then each admissible $x_1$ determines a unique $x_2$ by
\begin{equation}
	x_2 \in x_1 + b+\p^r \subset -\frac{1}{2v}+\frac{b_1}{2}  + \p \subset \p. \nonumber
\end{equation}
But we need $x_2$ to be a unit. We conclude that if $v(\Delta)>0$, then $S=\emptyset$ and therefore $W_{\pi}(g_{-2a_1,a_1,v})=0$.

If $v(\Delta) = 0$, we have $\sharp S \leq 2$ so that trivial estimates reveal
\begin{equation}
	\abs{W_{\pi}(g_{-2a_1,a_1,v})} \leq 2. \nonumber
\end{equation}

\textbf{Case V: $a_1 < l < n$.} Here we have $t=-2l$, and $l_1=l_2=l$. The set $S$ from Lemma\ref{lm:prototype_K_stat} is given by the system of congruences
\begin{eqnarray}
	&&x_1 \in x_2-b_1\varpi^{l-a_1}+b_2\varpi^{l-a_2} + \p^r, \nonumber \\
	&&vx_2^2+x_2(1-vb_1\varpi^{l-a_1}+vb_2\varpi^{l-a_2}) + b_2\varpi^{l-a_2} \in \p^r. \nonumber	
\end{eqnarray}
One can check that the discriminant of the quadratic equation for $x_2$ satisfies $\Delta \in 1+\p$. Therefore $\sharp S \leq 2$. Furthermore, 
\begin{equation}
	A_{\p}= \left( \begin{matrix} 0 & vx_1x_2 \\ vx_1x_2 & 0 \end{matrix} \right). \nonumber
\end{equation}
Thus, estimating trivially yields
\begin{equation}
	\abs{W_{\pi}(g_{-2l,l,v})} \leq 2. \nonumber
\end{equation}

\textbf{Case VI: $\max(\frac{a_1}{2},a_2) < l < a_1$.} This constitutes the transition region and it turns out that the approach we took before mutates into very messy calculations. We therefore choose to take a different approach. We calculate
\begin{eqnarray}
	&&K(\chi_1\otimes \chi_2, (\varpi^{-a_1},\varpi^{-l}), v\varpi^{-l})\nonumber \\
	&& \qquad = \int_{\mathcal{O}^{\times}} \chi_2(x_2)\psi(\varpi^{-l}x_2)G(\varpi^{-a_1}(1+\varpi^{a_1-l}vx_2), \chi_1)d^{\times} x_2 \nonumber \\
	&& \qquad = \epsilon(\frac{1}{2}, \chi_1^{-1})\zeta_F(1) q^{-\frac{a_1}{2}} \int_{\mathcal{O}^{\times}} \chi_1^{-1}(1+\varpi^{a_1-l}vx) \chi_2(x)\psi(x\varpi^{-l}) d^{\times}x. \nonumber
\end{eqnarray}
Using the same trick in the $x_2$ integral yields
\begin{eqnarray}
	&& K(\chi_1\otimes \chi_2, (\varpi^{-a_1},\varpi^{-l}), v\varpi^{-l}) =  \int_{\mathcal{O}^{\times}} \chi_1(x_1)\psi(\varpi^{-a_1}x_1)G(\varpi^{-l}(1+vx), \chi_2)d^{\times} x_1 \nonumber \\
	&&\quad = \epsilon(\frac{1}{2}, \chi_2^{-1})\chi_1(-\frac{1}{v})\psi(-v^{-1}\varpi^{-a_1})\zeta_F(1) q^{\frac{a_2}{2}-l} \int_{\mathcal{O}^{\times}} \chi_1(1-v\varpi^{l-a_2}x)\chi_2^{-1}(x)\psi(x\varpi^{l-a_1-a_2})d^{\times} x. \nonumber
\end{eqnarray}
We insert these expressions in Lemma~\ref{lm:generic_ps_shape_Wf} and obtain 
\begin{eqnarray}
	W_{\pi}(g_{-l-a_1,l,v}) &=& \chi_2(v^{-1})\epsilon(\frac{1}{2}, \chi_1^{-1}) \zeta_F(1)^{-1} q^{\frac{l}{2}}  \nonumber \\
	&&\qquad \cdot \int_{\mathcal{O}^{\times}} \chi_1^{-1}(1+\varpi^{a_1-l}x) \chi_2(x)\psi(v^{-1}x\varpi^{-l}) d^{\times}x,\label{eq:2_one_dim_integrals_(a)}
\end{eqnarray}
and
\begin{eqnarray}
	W_{\pi}(g_{-l-a_1,l,v}) &=& \gamma  \epsilon(\frac{1}{2}, \chi_2^{-1})[\chi_1\chi_2^{-1}](-\frac{1}{v})\psi(-v^{-1}\varpi^{-a_1})\zeta_F(1)^{-1} q^{\frac{a_1+a_2-l}{2}} \nonumber \\
	&&\qquad \cdot\int_{\mathcal{O}^{\times}} \chi_1(1+\varpi^{l-a_2}x)\chi_2^{-1}(x)\psi(-v^{-1}x\varpi^{l-a_1-a_2})d^{\times} x. \label{eq:2_one_dim_integrals_(b)}
\end{eqnarray}

Note that estimating the integrals trivially recovers the local bound of \cite[Corollary~2.35]{Sa15_2}. We will now reduce these integrals to a situation where Lemma~\ref{lm:quoted_qp_bound} becomes applicable. We treat different cases.

\textbf{Case VI.1: $a_2<l<\frac{a_1+a_2}{2}$.} We will show further cancellation in the integral appearing in \eqref{eq:2_one_dim_integrals_(a)}. Suppose $a_2=1$ then the current situation implies $l< \frac{a_1+1}{2}$ but obviously this yields $l\leq \frac{a_1}{2}$ which is excluded from Case~IV. Thus we assume $a_2>1$ and write $a_2= 2r+\rho$ for $\rho\in \{0,1\}$ and $r\in \N$. For any $r\leq \kappa \leq a_2$ we calculate
\begin{eqnarray}
	&&W_{\pi}(g_{-l-a_1,l,v})\nonumber \\
 	&&\quad = \gamma\chi_2(v^{-1})\epsilon(\frac{1}{2}, \chi_1^{-1}) q^{\frac{l}{2}-\kappa} \sum_{y\in (\mathcal{O}/\p^{\kappa})^{\times}}\chi_2(y)\psi(v^{-1}y\varpi^{-l}) \nonumber \\
	&&\qquad \cdot \int_{\mathcal{O}}\psi\left(-\frac{b_1}{\varpi^{a_1}}\log_F(1+(y+\varpi^{\kappa} x)\varpi^{a_1-l})+\frac{b_2}{\varpi^{a_2}}(\frac{x}{y}\varpi^{\kappa}-\frac{x^2}{2y^2}\varpi^{2\kappa})+v^{-1}x\varpi^{\kappa-l}\right)dx. \nonumber
\end{eqnarray}
The Taylor expansion of the logarithm reads
\begin{equation}
	-\frac{b_1}{\varpi^{a_1}}\log_F(1+(y+\varpi^{\kappa} x)\varpi^{a_1-l}) = \sum_{j\geq 1} \frac{(-1)^{j} b_1}{j}(y+\varpi^{\kappa}x)^j\varpi^{(j-1)a_1-jl}. \nonumber
\end{equation}
We first observe that, since $l<\frac{a_1+a_2}{2}$, we have
\begin{equation}
	(j-1)a_1-jl+2\kappa \geq (j-1)a_1-jl+2r \geq (j-2)(a_1-l) \geq j-2. \nonumber
\end{equation}
Opening up $(y+t\varpi^r)^j$ using the binomial expansion and dropping all the terms with coefficients in $\mathcal{O}$ yields
\begin{eqnarray}
	W_{\pi}(g_{-l-a_1,l,v})  &=& \chi_2(v^{-1})\epsilon(\frac{1}{2}, \chi_1^{-1}) q^{\frac{l}{2}-\kappa} \sum_{y\in (\mathcal{O}/\p^{\kappa})^{\times}}\chi_2(y)\psi\left(v^{-1}y\varpi^{-l}-\frac{b_1}{\varpi^{a_1}}\log_F(1+y\varpi^{a_1-l})\right) \nonumber \\
	&&\qquad \cdot G(-\frac{b_2}{2y^2}\varpi^{2\kappa-a_2},g(y)), \nonumber
\end{eqnarray}
for
\begin{equation}
	g(y) = \frac{b_2}{y}\varpi^{\kappa-a_2}+(v^{-1}-b_1)\varpi^{\kappa-l}+b_1\sum_{j\geq 2} (-1)^{j} y^{j-1}\varpi^{(j-1)a_1-jl+\kappa}. \nonumber
\end{equation}
If we set $\kappa = a_2$, the Gau\ss\  sum degenerates to the characteristic function on $1-vb_1\in \p^{l-a_2}$ which is independent of $y$. Thus, we can write
\begin{equation}
	1-vb_1 = \beta \varpi^{l-a_2}. \nonumber
\end{equation} 
This suggests that the true period of the integrals is $\p^a_2$ so that we choose $\kappa= r$. In this case a quick investigation of $g(y)$ reveals
\begin{equation}
	g(y) \in y^{-1}v^{-1}\varpi^{-r-\rho}(vb_2+y\beta+y^2\p[y]). \nonumber 
\end{equation}
We define $h(y) = vyg(y)\varpi^{r+\rho}$ and conclude
\begin{eqnarray}
	W_{\pi}(g_{-l-a_1,l,v})  &=& \chi_2(v^{-1})\epsilon(\frac{1}{2}, \chi_1^{-1}) q^{\frac{l}{2}-r} \sum_{y\in (\mathcal{O}/\p^r)^{\times}}\chi_2(y)\psi\left(v^{-1}y\varpi^{-l}-\frac{b_1}{\varpi^{a_1}}\log_F(1+y\varpi^{a_1-l})\right) \nonumber \\
	&&\qquad \cdot G(-\frac{b_2}{2y^2}\varpi^{-\rho},\frac{\varpi^{-r-\rho}}{vy}h(y)). \nonumber
\end{eqnarray}
The congruence $h(y)\in \p^r$ has one solution in $\mathcal{O}/\p^{r}$ and therefore, the evaluation of the Gau\ss\  sum yields
\begin{equation}
	\abs{W_{\pi}(g_{-l-a_1,l,v})} \leq q^{\frac{l-a_2}{2}}. \nonumber 
\end{equation}

\textbf{Case IV.2: $l=\frac{a_1+a_2}{2}$.} Note that this implies that $a_1$ and $a_2$ have the same parity so that $n$ must be even. We write $a_2 = 2r+\rho$ for some $\rho\in \{0,1\}$ and some $r\in \N_0$. Similar to Case~IV.1 we can deduce from \eqref{eq:2_one_dim_integrals_(a)} that 
\begin{equation}
	1-vb_1 = \beta \varpi^{\frac{a_1-a_2}{2}}. \nonumber
\end{equation}

\textbf{Case IV.2.1: $r=0$.} Since we assume $a_2>0$ this implies $\rho=1$. We compute
\begin{eqnarray}
	W_{\pi}(g_{-l-a_1,l,v}) &=& \chi_2(v^{-1})\epsilon(\frac{1}{2}, \chi_1^{-1}) q^{\frac{l}{2}-1} \sum_{y\in (\mathcal{O}/\p)^{\times}} \chi_2(y)\chi_1^{-1}(1+\varpi^{\frac{a_1-1}{2}}y) \psi(v^{-1}y\varpi^{-\frac{a_1+1}{2}}) \nonumber \\
	&=& \chi_2(v^{-1})\epsilon(\frac{1}{2}, \chi_1^{-1}) q^{\frac{l}{2}-1} \sum_{y\in (\mathcal{O}/\p)^{\times}} \chi_2(y) \psi\left( \varpi^{-1}\left(\frac{y\beta}{v}+\frac{b_1y^2}{2}\right)\right). \nonumber
\end{eqnarray}
Using Weil's bound \eqref{eq:Weils_bound} we obtain square-root cancellation in the remaining sum. We have shown
\begin{equation}
	\abs{W_{\pi}(g_{-l-a_1,l,v})} \leq 2q^{\frac{l-a_2}{2}}. \nonumber
\end{equation}

\textbf{Case IV.2.2: $r>0$.} In this case we split up the integral from \eqref{eq:2_one_dim_integrals_(a)} in $q$-pieces. Using suitable Taylor expansions we can write
\begin{equation}
	W_{\pi}(g_{t,l,v}) = \chi_2(v^{-1})\epsilon(\frac{1}{2}, \chi_1^{-1}) q^{\frac{l}{2}-1} \sum_{y\in (\mathcal{O}/\p)^{\times}} \chi_2(y)\chi_1^{-1}(1+\varpi^{\frac{a_1-a_2}{2}}y) \psi(\frac{y}{v}\varpi^{-\frac{a_1+a_2}{2}}) S(f_y,a_2-1)\nonumber
\end{equation}	
 where $S(f_y,a_1-1)$ is the integral defined in \eqref{eq:def_of_S(fm)} with
\begin{eqnarray}
	f_y(t) &=& \frac{t}{y}\left[ b_2+\frac{\beta}{v} y+b_1y^2+\sum_{j\geq 3}(-1)^jb_1y^j \varpi^{(j-2)\frac{a_1-a_2}{2}}\right] \nonumber \\
	&&\quad + \sum_{k\geq 2} \frac{t^k}{y^k} \left[ \frac{(-1)^{k+1}}{k}b_2\varpi^{k-1}+\sum_{j\geq k} \frac{(-1)^j}{j}\binom{j}{k} b_1y^j\varpi^{(k-1)+(j-2)\frac{a_1-a_2}{2}} \right]. \nonumber 
\end{eqnarray}
The size of $S(f_y,a_2-1)$ is controlled by the discriminant $\Delta= \frac{\beta^2}{v^2}-4b_1b_2$. If $\Delta\in \p$, we see straight away that
\begin{equation}
	S(f_y,a_2-1) = 0 \text{ for all } y \in (\mathcal{O}/\p)^{\times}\setminus \{ -\frac{\beta}{2vb_1}\}. \nonumber
\end{equation}
Putting $y_0  = -\frac{\beta}{2vb_1}$ we obtain
\begin{equation}
	W_{\pi}(g_{t,l,v}) = \chi_2(v^{-1})\epsilon(\frac{1}{2}, \chi_1^{-1}) q^{\frac{l}{2}-1}\chi_2(y_0)\chi_1^{-1}(1+\varpi^{\frac{a_1-a_2}{2}}y_0) \psi(\frac{y_0}{v}\varpi^{-\frac{a_1+a_2}{2}}) S(f_{y_0},a_2-1). \nonumber
\end{equation}
Having a closer look at $f_{y_0}$ reveals
\begin{eqnarray}
	f_{y_0}(t) &=& t\left[-\frac{\Delta v}{2\beta}+  \sum_{j\geq 3} (-1)^j b_1y_0^{j-1}\varpi^{(j-2)\frac{a_1-a_2}{2}}\right]\nonumber  \\
	&&\quad +  \varpi t^2\left[\frac{\Delta v^2 b_1}{2\beta^2}+\varpi^{\frac{a_1-a_2}{2}}(\cdots) \right]+ \varpi^2 t^3\left[ \frac{b_2}{3y_0^3}+\varpi( \cdots)\right] + \varpi^3 t^4 \left[ \cdots \right]. \nonumber
\end{eqnarray}
Thus we can estimate $S(f_{y_0},a_2-1)$ using Lemma~\ref{lm:quoted_qp_bound} with $d_{\p} =3$, $\tau = 2$ and $M=2$. This yields
\begin{equation}
	\abs{W_{\pi}(g_{-l-a_1,l,v})} \leq 2q^{\frac{l}{2}-\frac{a_2}{3}} \text{ for } a_2\geq 4. \nonumber
\end{equation}
The remaining cases $a_2=2,3$ can be checked by hand.

If $\Delta\in \mathcal{O}^{\times}$ a similar analysis leads to the stronger bound
\begin{equation}
	\abs{W_{\pi}(g_{-l-a_1,l,v})} \leq 2q^{\frac{l}{2}-\frac{a_2}{2}} \text{ for } a_2\geq 4. \nonumber
\end{equation}

\textbf{Case VI.3: $l>\frac{a_1+a_2}{2}$.} This case is very similar to Case~IV.1.  One uses \eqref{eq:2_one_dim_integrals_(b)} instead of \eqref{eq:2_one_dim_integrals_(a)}. Slightly adjusting the argument leads to
\begin{equation}
	\abs{W_{\pi}(g_{-l-a_1,l,v})} \leq q^{\frac{a_1-l}{2}}. \nonumber 
\end{equation}

This completes the last case and so the proof.
\end{proof}

Before we come to the last situation we will prove some estimates for
\begin{equation}
	K_{l_2} = K(\chi_1\otimes \chi_2,(\varpi^{t+l_2},\varpi^{-l_2}),v\varpi^{-l})
\end{equation}
when $a(\chi_1)=a(\chi_2)=l>1$.

\begin{lemma} \label{lm:est_KL2}
Suppose $a(\chi_1) = a(\chi_1) = l > 1$. Then $K_{l_2}$ is non-zero if $t=-l_2$ or $t=-a(\chi_1\chi_2^{-1})-l_2$ and $a(\chi_1\chi_2^{-1})>\frac{a_1}{2}$. If there are no degenerate critical points we have
\begin{eqnarray}
	\abs{K_{l_2}} &\leq& 2\zeta_F(1)^2 q^{-\frac{a_1}{2}+\frac{t}{4}}.\nonumber
\end{eqnarray}
Furthermore, if $t>-2a_1$, we have degenerate critical points exactly when $t=-2l_2=-2a(\chi_1\chi_2^{-1})$. On the other hand, if $t=-2a_1$ and $a(\chi_1\chi_2)<a_1$ degenerate critical points exist if and only if $-1\in\mathcal{O}^2$. For $t=-2a_1$ and $a(\chi_1\chi_2)=a_1$ there is always a degenerate critical point. In general we have the bound
\begin{equation}
	\abs{K_{l_2}} \leq 2\zeta_F(1)^2 q^{-\frac{a_1}{2}+\frac{t}{6}}.\nonumber
\end{equation}
\end{lemma}
\begin{proof}
After exchanging the roles of $x_1$ and $x_2$, if necessary, we can assume without loss of generality that $t+l_2 \leq -l_2$. An important invariant in the following calculations will be
$v(b_1-b_2)$. Note that we have
\begin{equation}
	b_1-b_2 = b_{\chi_1\chi_2^{-1}}\varpi^{a(\chi_1)-a(\chi_1\chi_2^{-1})}. \nonumber
\end{equation}
We write $S_{l_2}$ for the set $S$ defined in Lemma~\ref{lm:prototype_K_stat} to keep track of the $l_2$ dependence. 

In the following we will only deal with even $a_1$. The situation for odd $a_2$ is similar and the necessary adaptions can be extracted from the proof of Lemma~\ref{lm:St_twist_2_bound}. Thus, we set $a_1 = 2r$.

\textbf{Case I: $a_1+l_2+t\geq r$.} This leads to a very simple structure of $S_{l_2}$.  Indeed the congruence condition can be transformed into
\begin{equation}
	x_2= -\frac{b_2}{vx_1}+\p^r \text{ and } b_1-b_2 \in \p^r. \nonumber
\end{equation}
Thus, if $v(b_1-b_2) \geq r$, then every $x_1\in (\mathcal{O}/\p^r)^{\times}$  determines a unique $x_2$. Otherwise $S_{l_2}$ is empty. We compute
\begin{eqnarray}
	K_{l_2} &=& q^{-a(\chi_1)}\zeta_F(1)^2 \chi_2\left(-\frac{b_2}{v}\right)\psi(-b_2\varpi^{-l}) \sum_{x_1\in  (\mathcal{O}/\p^r)^{\times}} [\chi_1\chi_2^{-1}](x_1)\psi\left(x_1\varpi^{t+l_2}-\frac{b_2}{vx_1}\varpi^{-l_2}\right) \nonumber \\
	&=& q^{-r}\zeta_F(1) \chi_2\left(-\frac{b_2}{v}\right)\psi(-b_2\varpi^{-l})S_{\chi_1\chi_2^{-1}}\left(1,-\frac{b_2}{v}\varpi^{-2l_2-t},-l_2-t\right). \nonumber
\end{eqnarray}
To estimate the twisted Kloosterman sum we use Lemma~\ref{lm:est_TKS}. First, if $-l_2-t = a(\chi_1\chi_2^{-1})$, then the current assumptions imply $l_2 = a(\chi_1\chi_2^{-1})$ and we have
\begin{equation}
	\abs{K_{l_2}} \leq 2\zeta_F(1)^2 q^{-\frac{a_1}{2}-\frac{a(\chi_1\chi_2^{-1})}{3}} = 2\zeta_F(1)^2 q^{-\frac{a_1}{2}+\frac{t}{6}}. \nonumber
\end{equation}
Second, if $-l_2-t<a(\chi_1\chi_2^{-1})$, then the twisted Kloosterman sum vanishes. Finally, if $-l_2-t>a(\chi_1\chi_2^{-1})\geq 1$, then due to the support of twisted Kloosterman sums the only interesting situation is $t=-2l_2$. In this case we have square root cancellation. Indeed
\begin{equation}
	\abs{K_{l_2}} \leq 2\zeta_F(1)^2 q^{-\frac{a_1}{2}+\frac{t}{4}}. \nonumber
\end{equation}

\textbf{Case II: $a_1-l_2 \geq r > a_1+l_2+t$.} In this situation the set $S_{l_2}$ is slightly more complicated. It is described by the congruences
\begin{equation} 
	x_2 = -\frac{b_2}{vx_1} + \p^r \text{ and } x_1\varpi^{a_1+l_2+t} = -b_1+b_2 +\p^r. \nonumber
\end{equation}
We observe that this implies $a(\chi_1\chi_1^{-1})>\frac{a_1}{2}$, since otherwise there are no solutions for $x_1$. Furthermore, $S_{l_2}$ is empty unless $l_2=-t-a(\chi_1\chi_2^{-1})$. We can parametrize $x_1$ by
\begin{equation}
	x_1 = -b_{\chi_1\chi_2^{-1}}+\alpha \varpi^{a(\chi_1\chi_2^{-1})-r} \text{ for } \alpha \in \mathcal{O}/\p^{a_1-a(\chi_1\chi_2^{-1})}. \nonumber
\end{equation}
Each choice of $x_1$ determines $x_2$ by
\begin{equation}
	x_2=-\frac{b_2}{vx_1} = \frac{b_2}{v} \sum_{j\geq 0} \frac{\alpha^j}{b_{\chi_1\chi_2^{-1}}^{j+1}}\varpi^{ja(\chi_1\chi_2^{-1})-jr}. \nonumber
\end{equation}
Note that if $a(\chi_1\chi_2^{-1})=a_1$ then there is a unique $x_1$ and estimating trivially yields
\begin{equation}
	K_{l_2} \leq \zeta_F(1)^2q^{-a_1}\leq \zeta_F(1)^2 q^{-\frac{a_1}{2}+\frac{t}{4}} . \nonumber
\end{equation}

Otherwise we can use this parametrization and write
\begin{eqnarray}
	K_{l_2} &=& \chi_1(-b_{\chi_1\chi_2^{-1}})\chi_2\left(\frac{b_2}{vb_{\chi_1\chi_2^{-1}}}\right)\psi\left(-b_{\chi_1\chi_2^{-1}}\varpi^{-a(\chi_1\chi_2^{-1})}-b_2\varpi^{-l}+\frac{b_2}{vb_{\chi_1\chi_2^{-1}}}\varpi^{t+a(\chi_1\chi_2^{-1})}\right) \nonumber \\
	&& \quad \cdot\zeta_F(1)^2 q^{-a_1}\sum_{\alpha\in \mathcal{O}/\p^{a_1-a(\chi_1\chi_2^{-1})}}[\chi_1\chi_2^{-1}]\left(1-\frac{\alpha}{b_{\chi_1\chi_2^{-1}}}\varpi^{a(\chi_1\chi_2^{-1})-r}\right) \nonumber \\
	&& \qquad \cdot\psi\left( \alpha \varpi^{-r} +\frac{b_2}{v} \sum_{j\geq 1} \frac{\alpha^j}{b_{\chi_1\chi_2^{-1}}^{j+1}}\varpi^{t+(j+1)a(\chi_1\chi_2^{-1})-jr} \right). \nonumber
\end{eqnarray}
Since $a(\chi_1\chi_2^{-1})-r>0$ we can apply Lemma~\ref{lm:char_trick}. Exploiting the convergent Taylor expansion of $\log_F$ one computes
\begin{eqnarray}
	K_{l_2} &=&  \chi_1(-b_{\chi_1\chi_2^{-1}})\chi_2(\frac{b_2}{vb_{\chi_1\chi_2^{-1}}})\psi(-b_{\chi_1\chi_2^{-1}}\varpi^{-a(\chi_1\chi_2^{-1})}-b_2\varpi^{-l}+\frac{b_2}{vb_{\chi_1\chi_2^{-1}}}\varpi^{t+a(\chi_1\chi_2^{-1})})\zeta_F(1)^2 \nonumber \\
	&&\quad \cdot q^{-a(\chi_1\chi_2^{-1})} \int_{\mathcal{O}} \psi\bigg(\alpha (2+\frac{b_2}{vb_{\chi_1\chi_2^{-1}}^2}\varpi^{a(\chi_1\chi_2^{-1})-l_2})\varpi^{-r} \nonumber \\
	&&\qquad\qquad\qquad \qquad + \sum_{j\geq 2}\frac{\alpha^j}{jb_{\chi_1\chi_2^{-1}}^{j+1}}(jb_2v^{-1}\varpi^{a(\chi_1\chi_2^{-1})-l_2}-b_{\chi_1\chi_2^{-1}}^2)\varpi^{(j-1)a(\chi_1\chi_2^{-1})-jr} \bigg)d\alpha. \nonumber
\end{eqnarray}
The remaining integral vanishes if the linear term is a unit. Thus, we are in a non zero situation only if $l_2 = a(\chi_1\chi_2^{-1})$. However, this implies $0 = a(\chi_1\chi_2^{-1})-l_2 > r-l_2 \geq 0$ due to the assumptions in the case under consideration. We conclude that the current case does not contribute any non-zero situation.

\textbf{Case III: $a_1-l_2 <r$ and $t>-2l$.} First we observe that $-l_2>-a_1$, since otherwise we are in the situation where $t=-2l$. We compute
\begin{eqnarray}
	K_{l_2}&=& \int_{\mathcal{O}^{\times}}\chi_1(x_1)G((\varpi^{a_1-l_2}+vx_1)\varpi^{-a_1},\chi_2)\psi(\varpi^{t+l_2}x_1)d^{\times}x_1\nonumber \\
	&=& \zeta_F(1)q^{-\frac{a_1}{2}}\epsilon(\frac{1}{2},\chi_2^{-1})\int_{\mathcal{O}^{\times}}\chi_1(x_1)\chi_2^{-1}(\varpi^{a_1-l_2}+vx_1)\psi(\varpi^{t+l_2}x_1)d^{\times}x_1.\nonumber
\end{eqnarray}
Using the $p$-adic logarithm yields
\begin{eqnarray}
	K_{l_2}&=& \zeta_F(1)^2q^{-\frac{a_1}{2}-\kappa}\epsilon(\frac{1}{2},\chi_2^{-1}) \sum_{x\in(\mathcal{O}/p^{\kappa})^{\times}}\chi_1(x)\chi_2^{-1}(\varpi^{a_1-l_2}+vx)\psi(\varpi^{t+l_2}x) \nonumber\\
	&&\qquad \cdot\int_{\mathcal{O}} \psi\bigg(y(\varpi^{a_1+\kappa+t}+\frac{b_1}{x}-\frac{vb_2}{vx+\varpi^{a_1-l_2}})\varpi^{\kappa-a_1}\nonumber \\
	&&\qquad\qquad\qquad\qquad - \sum_{j\geq 2}\frac{(-1)^j}{j}y^j\left( \frac{b_1}{x^j}-\frac{b_2v^j}{(vx+\varpi^{a_1-l_2})^j} \right)\varpi^{j\kappa-a_1} \bigg)dy.\nonumber
\end{eqnarray}
From the linear term we obtain the quadratic congruence
\begin{equation}
	vx^2\varpi^{a_1+l_2+t}+x(vb_{\chi_1\chi_2^{-1}}\varpi^{a_1-a(\chi_1\chi_2^{-1})}+\varpi^{2a_1+t})+b_1\varpi^{a_1-l_2} \in \p^{a_1-\kappa} \nonumber
\end{equation}
which is necessary for the $t$-integral to be non-zero.  In particular looking at $\kappa=r$ we observe that we must have $a(\chi_1\chi_2^{-1})>r$. 

First, we suppose that $l_2+t<-l_2$. In this case the congruence has solutions $x$ that are units if and only if $t=-l_2-a(\chi_1\chi_2^{-1})$. We choose $\kappa = \lfloor \frac{a(\chi_1\chi_2^{-1})}{2} \rfloor$. The current assumptions imply that we can truncate the sum in the integral after the second term. We are left with a normal quadratic Gau\ss\  sum. Looking at the entries carefully reveals that there is exactly one admissible $x$ for which we can estimate the Gau\ss\  sum by $q^{-\{\frac{a(\chi_1\chi_2^{-1})}{2}\}}$. Therefore,
\begin{equation}
	\abs{K_{l_2}}\leq \zeta_F(1)^2 q^{-\frac{a_1+a(\chi_1\chi_2^{-1})}{2}}\leq \zeta_F(1)^2 q^{-\frac{a_1}{2}+\frac{t}{4}}. \nonumber
\end{equation}

Second, if $t=-2l_2 \neq -2a(\chi_1\chi_2^{-1})$ we choose $\kappa = \lfloor \frac{l_2}{2} \rfloor$. A familiar arguments yields
\begin{equation}
	\abs{K_{l_2}}\leq 2\zeta_F(1)^2q^{-\frac{a_1}{2}+\frac{t}{4}}. \nonumber
\end{equation}

At last we consider $t=-2a(\chi_1\chi_2^{-1})$. For this situation there might exists degenerate critical points. We have to solve the congruence
\begin{equation}
	x^2+x(b_{\chi_1\chi_2^{-1}}+\frac{\varpi^{a_1-a(\chi_1\chi_2^{-1})}}{v})+\frac{b_1}{v} \in \p^{a(\chi_1\chi_2^{-1})-\kappa}. \nonumber
\end{equation}
If the discriminant $\Delta=(b_{\chi_1\chi_2^{-1}}+\frac{\varpi^{a_1-a(\chi_1\chi_2^{-1})}}{v})^2-\frac{4b_1}{v}$ is a unit, we may argue as above. Thus, let us assume $\Delta\in\p$. We choose $\kappa = r$ and parametrise
\begin{equation}
	x=-\frac{b_{\chi_1\chi_2^{-1}}}{2}-\frac{\varpi^{a_1-a(\chi_1\chi_2^{-1})}}{2v} \pm \frac{Y}{2}\varpi^{\delta}+\alpha \varpi^{a(\chi_1\chi_2^{-1})-r-\delta} \text{ for } \alpha \in\mathcal{O}/\p^{a_1-a(\chi_1\chi_2^{-1})+\delta} \nonumber
\end{equation}
and
\begin{eqnarray}
	Y &=& \begin{cases} 0 &\text{ if } v(\Delta)\geq a(\chi_1\chi_2^{-1})-r, \\ Y_0 &\text{ if } v(\Delta)<a(\chi_1\chi_2^{-1})-r \text{ and } (\Delta)_0 = Y_0^2, \end{cases} \text{and }\nonumber\\
	 \delta &=& \begin{cases} \lfloor\frac{a(\chi_1\chi_2^{-1})-r}{2}\rfloor &\text{ if } v(\Delta)\geq a(\chi_1\chi_2^{-1})-r, \\ \delta_0 &\text{ if } v(\Delta)=2\delta_0<a(\chi_1\chi_2^{-1})-r. \end{cases} \nonumber
\end{eqnarray}
If we reinsert this parametrisation in the $x$-sum, we obtain 
\begin{eqnarray}
	K_{l_2}&=& \zeta_F(1)^2q^{\delta-a(\chi_1\chi_2^{-1})}\sum_{\pm} \gamma_{\pm} \int_{\mathcal{O}^{\times}}\psi\bigg(A_1x  \nonumber \\
	&&\qquad\qquad - \sum_{j\geq 2}\frac{(-1)^j}{j}x^j\left( \frac{b_1}{A}-\frac{b_2v^j}{(vA+\varpi^{a_1-a(\chi_1\chi_2^{-1})})^j }\right)\varpi^{j(a(\chi_1\chi_2^{-1})-r-\delta)-a_1}\bigg) \nonumber
\end{eqnarray}
for some $A_1\in F$, some $\gamma_{\pm}\in S^1$ and $A=-\frac{b_{\chi_1\chi_2^{-1}}}{2}-\frac{\varpi^{a_1-a(\chi_1\chi_2^{-1})}}{2v} \pm \frac{Y}{2}\varpi^{\delta}$. Observe that $A(vA+\varpi^{a_1-a(\chi_1\chi_2^{-1})})\in\mathcal{O}^{\times}$ so that  looking at the $j$-th coefficient yields
\begin{equation}
	\left( \frac{b_1}{A}-\frac{b_2v^j}{(vA+\varpi^{a_1-a(\chi_1\chi_2^{-1})})^j }\right) \in \p^{a_1-a(\chi_1\chi_2^{-1})}.\nonumber
\end{equation}
Furthermore, we check that 
\begin{equation}
	(vA+\varpi^{a_1-a(\chi_1\chi_2^{-1})})^jb_1-A^jb_2 \in A^jv^j(b_1-b_2)+jv^{j-1}A^{j-1}\varpi^{a_1-a(\chi_1\chi_2^{-1})}b_1 +\p^{2a_1-2a(\chi_1\chi_2^{-1})}.\nonumber
\end{equation}
This helps us to check that the second order term is in $\p^{a_1-a(\chi_1\chi_2^{-1})+\delta}$ and the third order term is in $3^{-1}\varpi^{a_1-a(\chi_1\chi_2^{-1})}\mathcal{O}^{\times}$. Note that we can truncate the Taylor series latest after the $a_1$-th term. Thus, in the worst case scenario we obtain the bound
\begin{equation}
	\abs{K_{l_2}}\leq 2\zeta_{F}(1)^{2}q^{-r-\frac{a(\chi_1\chi_2^{-1})}{3}}.\nonumber
\end{equation}

\textbf{Case IV: $t=-2l$.} In this case we will take a very familiar approach. First, we note that $l_2 = t+l_2 = l =a_1$. Thus, the congruences reduce to
\begin{equation}
	vx_2^2+(1-vb_{\chi_1\chi_2^{-1}}\varpi^{a_1-a(\chi_1\chi_2^{-1})})x_2+b_2 \in \p^r \text{ and } x_1 = x_2-b_{\chi_1\chi_2^{-1}}\varpi^{a_1-a(\chi_1\chi_2^{-1})}\in \p^r. \nonumber
\end{equation}
We can now solve the remaining quadratic congruence as in many of the previous cases. Its discriminant is given by
\begin{equation}
	\Delta = (1-vb_{\chi_1\chi_2^{-1}}\varpi^{a_1-a(\chi_1\chi_2^{-1})})^2-4vb_2 = 1-2v(b_1+b_2)+v^2b_{\chi_1\chi_2^{-1}}^2\varpi^{2a_1-2a(\chi_1\chi_2^{-1})}. \nonumber
\end{equation}
If $\Delta \in \mathcal{O}^{\times}$, then obviously $\sharp S\leq 2$. In this case we can estimate trivially and obtain
\begin{equation}
	\abs{K_{a_1}} \leq 2\zeta_F(1)^2 q^{-a_1}. \nonumber
\end{equation}

From now on we assume that $\Delta \in \p$. We define
\begin{equation}
	\delta = \begin{cases} 
		\lfloor \frac{r}{2} \rfloor &\text{ if } v(\Delta) \geq r, \\
		\delta_0 &\text{ if } v(\Delta) = 2\delta_0 <r, \\
	\end{cases}  \qquad
	Y = \begin{cases} 
		0 &\text{ if }v(\Delta) \geq r, \\
		Y_0 &\text{ if } (\Delta)_0 = Y_0^2 \text{ and } v(\Delta)<r.
	\end{cases} \nonumber
\end{equation}
Assuming $S_{a_1}$ is non empty, we parametrise it by
\begin{equation}
	S_{a_1} = \left\{ \left(-\frac{1}{2v}-\frac{b_1-b_2}{2}\pm \frac{Y}{2v}\varpi^{\delta}+\alpha \varpi^{r-\delta},-\frac{1}{2v}+\frac{b_1-b_2}{2}\pm \frac{Y}{2v}\varpi^{\delta}+\alpha \varpi^{r-\delta}\right) \colon \alpha \in \mathcal{O}/\p^{\delta} \right\}. \nonumber
\end{equation}
To shorten notation we set $A_{\pm} = -\frac{1}{2v}-\frac{b_1-b_2}{2}\pm \frac{Y}{2v}\varpi^{\delta}$  and $B_{\pm} = -\frac{1}{2v}+\frac{b_1-b_2}{2}\pm \frac{Y}{2v}\varpi^{\delta}$. 

We proceed by inserting the parametrisation of $S_{a_1}$ in the $S_{a_1}$-sum. This yields
\begin{eqnarray}
	K_{a_1} &=& \sum_{\pm}\chi_1(A_{\pm}) \chi_2(B_{\pm})\psi(A_{\pm}\varpi^{-a_1}+B_{\pm}\varpi^{-a_1}+vA_{\pm}B_{\pm}\varpi^{-a_1}) \zeta_F(1)^2 q^{-a_1} \nonumber \\
	&&\quad \cdot  \sum_{\alpha \in \mathcal{O}/\p^{\delta}}\chi_1(1+\frac{\alpha}{A_{\pm}}\varpi^{r-\delta})\chi_2(1+\frac{\alpha}{B_{\pm}}\varpi^{r-\delta})\psi((1\pm Y\varpi^{\delta})\alpha \varpi^{-r-\delta}+v\alpha^2\varpi^{-2\delta}) \nonumber 
\end{eqnarray}
As usual we use Lemma~\ref{lm:char_trick} and the $p$-adic logarithm to deal with the two characters. Observing $\delta\leq \frac{r}{2}$ enables us to truncate the Taylor expansion of the logarithm after the 3rd term. This yields
\begin{eqnarray}
	K_{a_1} &=& \sum_{\pm} \chi_1(A_{\pm}) \chi_2(B_{\pm})\psi(A_{\pm}\varpi^{-a_1}+B_{\pm}\varpi^{-a_1}+vA_{\pm}B_{\pm}\varpi^{-a_1}) \zeta_F(1)^2 q^{\delta-a_1} \nonumber \\
	&&  \cdot \int_{\mathcal{O}} \psi\bigg(y(1\pm Y\varpi^{\delta}+\frac{b_1}{A_{\pm}}+\frac{b_2}{B_{\pm}})\varpi^{-r-\delta} \nonumber \\
	&&\qquad\qquad +y^2(v-\frac{1}{2}(\frac{b_1}{A_{\pm}^2}+\frac{b_2}{B_{\pm}^2}))\varpi^{-2\delta}+\frac{y^3}{3}(\frac{b_1}{A_{\pm}^3}+\frac{b_2}{B_{\pm}^3})\varpi^{r-3\delta} \bigg) dy. \nonumber
\end{eqnarray} 
First, we observe that in order for $S_{l_2}$ to be non-empty we need $A_{\pm},B_{\pm}  \in \mathcal{O}^{\times}$. This translates to 
\begin{equation}
	v\not \in \pm (b_1-b_2)^{-1}+\p. \nonumber
\end{equation}
Since 
\begin{equation}
	A_{\pm}B_{\pm} \in \frac{1}{4v^2}-\frac{(b_1-b_2)^2}{4}+\p  \nonumber
\end{equation}
we conclude that $A_{\pm}B_{\pm} \in \mathcal{O}^{\times}$. 

Note that if the linear or the quadratic term are units then we have at least square root cancellation. Thus, we are left with showing that the coefficient in front of $t^3$ is (close to) a unit. The following computations are modulo $\p$. First, $\Delta\in\p$ implies
\begin{equation}
	1+v^2(b_1-b_2)^2 \in 2v(b_1+b_2)+\p. \nonumber
\end{equation}
We also compute
\begin{equation}
	A_{\pm}^2 \in \frac{b_1}{v}+\p \text{ and } B_{\pm}^2 \in \frac{b_2}{v}. \nonumber  
\end{equation}
Using this an easy computation shows
\begin{equation}
	\frac{b_1}{A_{\pm}^3}+\frac{b_2}{B_{\pm}^3} \in (A_{\pm}B_{\pm})^{-1}+\p \subset \mathcal{O}^{\times}. \nonumber
\end{equation}
Thus we may apply Lemma~\ref{lm:quoted_qp_bound} to obtain cube root cancellation. This gives
\begin{equation}
	\abs{K_{a_1}} \leq 2\zeta_F(1)^2q^{-a_1+\frac{r}{3}}. \nonumber
\end{equation}

\end{proof}

\begin{lemma} \label{lm:generic_ps_bound}
Let $\pi = \chi_1\abs{\cdot}^s\boxplus\chi_2\abs{\cdot}^{-s}$ where $s\in i\R$, and $a(\chi_1)=a(\chi_2)$ but $\chi_1\neq \chi_2$. Then we have
\begin{equation}
	\abs{W_{\pi}(g)}\leq 2q^{\frac{n}{12}}. \nonumber
\end{equation}
If $a(\chi_1\chi_2)<\frac{n}{2}$ and $-1\not\in \mathcal{O}^{\times 2}$, then we have the stronger bound
\begin{equation}
	\abs{W_{\pi}(g)}\leq 2. \nonumber
\end{equation}
\end{lemma}
The proof follows our usual strategy and is very similar in spirit to the proof of previous lemmata. Thus we will be very brief.
\begin{proof}
If $l=0$ or $l\geq n$ or $t>-2$, the formulas given in Lemma~\ref{lm:generic_ps_shape_Wf} can be easily estimated. The cases $0<l<n$ and $l\neq \frac{n}{2}$ can be covered by estimating the $S$-sum in Lemma~\ref{lm:prototype_K_stat} trivially. Finally, if $l=\frac{n}{2}$, we use Lemma~\ref{lm:generic_ps_shape_Wf} and Lemma~\ref{lm:est_KL2}. First, one observes that by the support properties of $K_{l_2}$ there is only one $l_2$ for each $t$ that contributes to the $l_2$-sum. The desired estimate follows directly after applying the upper bounds from  Lemma~\ref{lm:est_KL2}.
\end{proof}

\bibliographystyle{plain}
\bibliography{bibliography} 

\end{document}